\documentclass[11pt,a4paper,reqno,oneside]{amsart}
\usepackage[T1]{fontenc}
\usepackage[utf8]{inputenc}

\allowdisplaybreaks

\usepackage{lmodern}
\usepackage{microtype}
\usepackage{geometry}

\usepackage{mathtools}

\usepackage{esint}

\usepackage{amsthm}

\usepackage{hyperref}
\usepackage[capitalize]{cleveref}

\newcounter{chapter}

\numberwithin{equation}{chapter}

\newtheorem{proposition}{Proposition}[chapter]
\newtheorem{theorem}{Theorem}[chapter]
\newtheorem{lemma}{Lemma}[chapter]

\theoremstyle{definition}
\newtheorem{definition}{Definition}[chapter]

\theoremstyle{remark}
\newtheorem{remark}{Remark}[chapter]
\newtheorem{myexample}{Example}[chapter]
\crefname{myexample}{example}{examples}

\usepackage{enumerate}
\usepackage[abbrev]{amsrefs}
\renewcommand{\PrintDOI}[1]{%
  \href{http://dx.doi.org/#1}{doi:#1}%
}
\BibSpec{book}{%
    +{}  {\PrintPrimary}                {transition}
    +{,} { \textit}                     {title}
    +{.} { }                            {part}
    +{:} { \textit}                     {subtitle}
    +{,} { \PrintEdition}               {edition}
    +{}  { \PrintEditorsB}              {editor}
    +{,} { \PrintTranslatorsC}          {translator}
    +{,} { \PrintContributions}         {contribution}
    +{,} { }                            {series}
    +{,} { \voltext}                    {volume}
    +{,} { }                            {publisher}
    +{,} { }                            {organization}
    +{,} { }                            {address} 
    +{,} { \PrintDateB}                 {date}
    +{,} { }                            {status}
    +{,} { \PrintDOI}                   {doi}
    +{}  { \parenthesize}               {language}
    +{}  { \PrintTranslation}           {translation}
    +{;} { \PrintReprint}               {reprint}
    +{.} { }                            {note}
    +{.} {}                             {transition}
    +{}  {\SentenceSpace \PrintReviews} {review}
}

\newcommand{\Rset}{\mathbb{R}}
\newcommand{\Cset}{\mathbb{C}}
\newcommand{\Deriv}{\mathrm{D}}
\newcommand{\Nset}{\mathbb{N}}
\newcommand{\Sset}{\mathbb{S}}

\newcommand{\defeq}{\coloneqq}

\newcommand{\dif}{\;\mathrm{d}}

\DeclarePairedDelimiter{\abs}{\lvert}{\rvert}
\DeclarePairedDelimiter{\norm}{\lVert}{\rVert}
\DeclarePairedDelimiter{\seminorm}{\lvert}{\rvert}
\DeclarePairedDelimiter{\brk}{(}{)}

\DeclarePairedDelimiterX\dualprod[2]{\langle}{\rangle}{#1, #2}
\DeclarePairedDelimiterX\intvo[2]{(}{)}{#1, #2}
\DeclarePairedDelimiterX\intvc[2]{[}{]}{#1, #2}
\DeclarePairedDelimiterX\intvl[2]{(}{]}{#1, #2}
\DeclarePairedDelimiterX\intvr[2]{[}{)}{#1, #2}
\providecommand{\st}{\,\vert\,}

\newcommand\stSymbol[1][]{%
\nonscript\;#1\vert
\allowbreak
\nonscript\;
\mathopen{}}
\DeclarePairedDelimiterX\set[1]\{\}{%
\renewcommand\st{\stSymbol[\delimsize]}
#1
}
\DeclareMathOperator{\Lin}{Lin}

\newcommand{\compose}{\,\circ\,}

\DeclareMathOperator{\supp}{supp}
\DeclareMathOperator{\trace}{tr}
\DeclareMathOperator{\divergence}{div}

\newcommand{\familyname}[1]{\textsc{#1}}

\usepackage{constants}

\begin{document}

\title[Cancelling operators and endpoint Sobolev inequalities]{Injective ellipticity, cancelling operators, and endpoint Gagliardo-Nirenberg--Sobolev inequalities for vector fields}

\author{Jean Van Schaftingen}
\address{Université catholique de Louvain (UCLouvain), Institute de Recherche en Mathématique et Physique (IRMP), Chemin du Cyclotron 2 bte L7.01.01, 1348 Louvain-la-Neuve, Belgium}
\email{Jean.VanSchaftingen@UCLouvain.be}

\date{June 12, 2023}
\thanks{These lecture notes were written on the occasion of the summer school ``Geometric and analytic aspects of functional variational principles'' organized by the Centro Internazionale Matematico Estivo (CIME) in Cetraro (Cosenza, Italy) between June 27 and July 1, 2022
and published in A. \familyname{Cianchi}, V. \familyname{Maz'ya} and T. \familyname{Weth}, \emph{Geometric and Analytic Aspects of Functional Variational Principles: Cetraro, Italy 2022}, Springer (Cham), Lecture Notes in Mathematics, volume 2348, 259–317 (\url{https://doi.org/10.1007/978-3-031-67601-7_5}).}

\maketitle

\begin{abstract}
Although Ornstein's nonestimate entails the impossibility to control in general all the \(L^1\)-norm of derivatives of a function by the \(L^1\)-norm of a constant coefficient homogeneous vector differential operator, the corresponding endpoint Sobolev inequality has been known to hold in many cases: the gradient of scalar functions (Gagliardo and Nirenberg), the deformation operator (Korn-Sobolev inequality by M.J. Strauss), and the Hodge complex (Bourgain and Brezis).
The class of differential operators for which estimates holds can be characterized by a cancelling condition. 
The proof of the estimates rely on a duality estimate for \(L^1\)-vector fields lying in the kernel of a cocancelling  differential operator, combined with classical linear algebra and harmonic analysis techniques. 
This characterization unifies classes of known Sobolev inequalities and extends to fractional Sobolev and Hardy inequalities. 
A similar weaker condition introduced by Raiță characterizes the operators for which there is an \(L^\infty\)-estimate on lower-order derivatives.
\end{abstract}

\tableofcontents 

\section{Introduction}
The central question in the present notes is to determine for which constant coefficient homogeneous vector differential operator \(A (\Deriv)\) an estimate of the form 
\begin{equation}
\label{eq_oz3Eib2iefoh7igh2Au6uCho}
 \norm{\Deriv^\ell u}_{L^q(\Rset^n)} \le \norm{A (\Deriv) u}_{L^1 (\Rset^n)}
\end{equation}
holds for every vector field \(u \in C^\infty_c (\Rset^n, V)\).
After reviewing how classical harmonic analysis, including Fourier analysis and singular integral theory settle the case where \(L^p(\Rset^n)\) with \(p \in \intvo{1}{+\infty}\) replaces \(L^1(\Rset^n)\) in the right-hand side of \eqref{eq_oz3Eib2iefoh7igh2Au6uCho} and how Ornstein's nonestimate dooms this approach for \eqref{eq_oz3Eib2iefoh7igh2Au6uCho} in \cref{section_ellipticity_nonestimates},
we present the cancellation condition characterizing Bourgain-Brezis endpoint Sobolev estimates in \cref{section_cancelling} and the corresponding duality estimates for cocancelling operators in \cref{section_cocancelling}.
Finally, in \cref{section_variants} we connect the duality estimates with functions of bounded mean oscillation, we study the corresponding fractional Sobolev and Hardy inequalities, and we present Raiță’s characterization of operators for which uniform estimates holds.

The original material covered in this notes corresponds essentially to the references \citelist{\cite{VanSchaftingen_2008}\cite{VanSchaftingen_2006}\cite{VanSchaftingen_2013}\cite{Bousquet_VanSchaftingen_2014}\cite{Raita_2019}}, complemented by the survey article \cite{VanSchaftingen_2014} and the more informal lecture notes \cite{PasekyLectCanc}.

\section{Estimates for Vector Elliptic Operators}
\label{section_ellipticity_nonestimates}

The central topic of the present lecture notes is \emph{constant coefficients vector differential operators}.

\begin{definition}
  Given \(n \in \Nset \setminus \set{0}\) and finite-dimensional (real) vector spaces \(V\) and \(E\),
\(A (\Deriv)\) is a \emph{homogeneous constant coefficient differential operator
of order \(k \in \Nset \setminus \set{0}\) from 
\(V\) to \(E\) on \(\Rset^n\)} whenever 
there exists \(A \in \smash{\Lin (\Lin_{\mathrm{sym}}^k (\Rset^n, V), E)}\) such that for every \(u \in C^\infty (\Rset^n, V)\) and every \(x \in \Rset^n\), one has
\(A (\Deriv)u (x)= A [\Deriv^k u (x)]\).
\end{definition}

In other words, \(A (\Deriv)\) is a homogeneous constant coefficient differential operator
of order \(k \in \Nset \setminus \set{0}\) from 
\(V\) to \(E\) on \(\Rset^n\) when there exists a linear map \(A : \smash{\Lin_{\mathrm{sym}}^k} (\Rset^n, V) \to E\) such that 
the differential operator \(A (\Deriv)\) can be represented at any point \(x \in \Rset^n\) as the linear map \(A\) applied to the total \(k\)-th order derivative \(\Deriv^k u : \Rset^n \to \smash{\Lin_{\mathrm{sym}}^k} (\Rset^n, V)\), seen as a function from \(\Rset^n\) to symmetric \(k\)-linear maps from \(\smash{(\Rset^n)^k}\) to \(V\).

In terms of partial derivatives, \(A (\Deriv)\) is a homogeneous constant coefficient differential operator
  of order \(k \in \Nset \setminus \set{0}\) from 
  \(V\) to \(E\) on \(\Rset^n\) if and only if for each multiindex \(\alpha = (\alpha_1, \dotsc, \alpha_n) \in \Nset^n\) satisfying \(\abs{\alpha} \defeq \alpha_1 + \dotsb + \alpha_n = k\) there exists a linear map \(A_\alpha \in \Lin (V, E)\), such that for every \(u \in C^\infty (\Rset^n, V)\) and for every \(x \in \Rset^n\), we have
\begin{equation}
  \label{eq_Io6oigh0kee2ocaev}
  A (\Deriv) u (x) = \sum_{\substack{\alpha \in \Nset^n\\ \abs{\alpha} = k}} A_\alpha [\partial^\alpha u (x)] \; .
\end{equation}
Alternatively, the representation \eqref{eq_Io6oigh0kee2ocaev} can be written as 
\begin{equation*}
A (\Deriv) u (x) = \sum_{\substack{\alpha \in \Nset^n\\ \abs{\alpha} = k}} \partial^\alpha (A_\alpha [u]) (x)\;,
\end{equation*}
since \(A_\alpha\) does not depend on the variable \(x\) and commutes thus with the partial derivative \(\partial^\alpha\).
Our question in the present notes will be to determine when, whether and how a \(k\)-th order differential operator \(A (\Deriv)u\) is a good substitute for the total \(k\)-th order derivative \(\Deriv^k u\).

\subsection{Injective Ellipticity}
We first examine what the Fourier theory tells us about \(u\) in terms of \(A (\Deriv)u\). 
Given a Schwartz test function \(u \in \mathcal{S} (\Rset^n, V)\), its Fourier transform \(\mathcal{F} u : \Rset^n\to V + i V\) (here \(V + iV\) is understood as the complexification of the linear space \(V\)) is defined for every 
\(\xi \in \Rset^n\) as 
\begin{equation*}
(\mathcal{F} u) (\xi) 
\defeq 
\int_{\Rset^n} e^{-2\pi i \dualprod{\xi}{x}}  \, u (x) \dif x\;.
\end{equation*}
By classical properties of the Fourier transform, we have for every \(\xi \in \Rset^n\)
\begin{equation}
  \label{eq_du0zoovuociuBi9nu}
\mathcal{F} (A (\Deriv) u) (\xi)
= A \brk[\big]{(2 \pi i)^k \mathcal{F} u (\xi)  \otimes \xi^{\otimes k}}
= (2 \pi i)^k A (\xi)[\mathcal{F} u (\xi)]\;,
\end{equation}
where \(A (\xi) \in \Lin (V + iV, E + iE)\) also denotes the complexification of the operator \(A(\xi)\), which is defined for every \(v \in V\) by
\begin{equation}
  \label{eq_deeque4ungie3eijeMooCiof}
  A (\xi) [v] \defeq A (v \otimes \xi^{\otimes k})
\end{equation}
and is the symbol of the differential operator \(A (\Deriv)\).
Equivalently, the identity \eqref{eq_deeque4ungie3eijeMooCiof} can be rewritten by the representation formula \eqref{eq_Io6oigh0kee2ocaev} as 
\begin{equation}
  A (\xi) =\sum_{\substack{\alpha \in \Nset^n\\ \abs{\alpha} = k}} \xi^\alpha A_\alpha
  \in \Lin (V, E)\;,
\end{equation}
where for every \(\alpha =  (\alpha_1, \dotsc, \alpha_n) \in \Nset^n\) and \(\xi= (\xi_1, \dotsc, \xi_n) \in \Rset^n\), we have written \(\xi^\alpha \defeq \xi_1^{\alpha_1} \dotsm \xi_n^{\alpha_n} \in \Rset\).
By the Parseval identity for the Fourier transform and by the identity \eqref{eq_du0zoovuociuBi9nu}, we observe that 
\begin{equation}
  \label{eq_mi2eijaefaazae5io1ef7ohM}
\int_{\Rset^n} \abs{A (\Deriv) u}^2 = 
(2 \pi)^{2k} \int_{\Rset^n} \abs{A (\xi) [\mathcal{F} u(\xi)]}^2  \dif \xi\;.
\end{equation}

In order to use the right-hand side of the inequality \eqref{eq_mi2eijaefaazae5io1ef7ohM} to control derivatives of \(u\), we will use the \emph{injective ellipticity} condition, which will guarantee that the integrand in the right-hand side integral of \eqref{eq_mi2eijaefaazae5io1ef7ohM} does not vanish at \(\xi \in \Rset^n \setminus \set{0}\) unless \(\mathcal{F} u(\xi)\) does.

\begin{definition}[Injective ellipticity%
\footnote{%
The definition of injective ellipticity (\cref{definition_injectively_elliptic}) appeared for overdetermined differential operators \citelist{\cite{Hormander_1958}*{Th.\ 1}\cite{Spencer_1969}*{Def.\ 1.7.1}\cite{Cantor_1981}*{Def.\ 3.4}\cite{Solonnikov_1971}\cite{Schulze_1979}*{Def.\ 1.1}} (when \(\dim V = 1\), see also \citelist{\cite{Agmon_1959}*{\S 7}\cite{Agmon_1965}*{Def.\ 6.3}}); it is also called \emph{right ellipticity} \cite{Dacorogna_Gangbo_Kneuss_2018}*{Def.\ 26}, which might seem somehow inconsistent with the \emph{left}-invertibility of the operator it is equivalent to;
in the study of Sobolev inequalities it is standard \cite{VanSchaftingen_2013}*{Def.\ 1.1}.
The terminology \emph{injectively elliptic} is borrowed from \cite{Grubb_1990}*{\S 4} in the context of pseudo-differential operators (see also \cite{BoossBavnbek_Wojciechowski_1993}*{Rem.\ 18.2 (d) (2a)}).}%
]
  \label{definition_injectively_elliptic}
  Given \(n \in \Nset \setminus \set{0}\) and finite-dimensional vector spaces \(V\) and \(E\),
a homogeneous constant coefficient differential operator \(A (\Deriv)\)
    of order \(k \in \Nset \setminus \set{0}\) from 
    \(V\) to \(E\) on \(\Rset^n\) 
    is \emph{injectively elliptic} whenever for every \(\xi \in \Rset^n\setminus \set{0}\), one has
    \(\ker A (\xi) = \set{0}\). 
\end{definition}

In other words, the operator \(A (\Deriv)\) is injectively elliptic if and only if for every \(\xi \in \Rset^n\) and \(v \in V\), the condition
\(A(\xi)[v] = 0\) implies either \(\xi = 0\) or \(v = 0\).

In the one-dimensional case \(n = 1\), injectively elliptic operators are those operators that can be written as an injective function of the derivative; this allows then one to recover the full derivative \(\Deriv^k\) from \(A (\Deriv)\).

\begin{proposition}[One-dimensional injectively elliptic operators]
\label{proposition_1d_injectively_elliptic}
Given finite-dimensional vector spaces \(V\) and \(E\),
a homogeneous constant coefficient differential operator \(A (\Deriv)\) of order \(k \in \Nset \setminus \set{0}\) from \(V\) to \(E\) on \(\Rset\) is injectively elliptic if and only if there exists  \(A \in \Lin (V,E)\) such that \(\ker A = \set{0}\) and for every \(u \in C^\infty (\Rset, V)\), 
\[
 A (\Deriv) u = A (u^{(k)})\;.
\] 
\end{proposition}

Let us now review a few examples of injectively elliptic operators.

\begin{myexample}[Total derivative]
  The \emph{\(k\)-th order total derivative} \(A (\Deriv) = \Deriv^k\) is injectively elliptic as a differential operator from \(\Rset\) to \(\smash{\Lin_{\mathrm{sym}}^k (\Rset^n, \Rset)}\) on \(\Rset\).
Indeed, if \(\xi \in \Rset^n\) and \(v \in \Rset\) satisfy \(A (\xi)[v] = \xi^{\otimes k} v = 0\),
then either \(\xi = 0\) or \(v = 0\).
\end{myexample}

\begin{myexample}[Laplacian]
  The \emph{Laplacian} \(A (\Deriv) \defeq \Delta\) is injectively elliptic as a differential operator from \(\Rset\) to \(\Rset\) on \(\Rset^n\).
Indeed, if \(\xi \in \Rset^n\) and \(v \in \Rset\) satisfy the condition \(A (\xi)[v] =  \abs{\xi}^2 v = 0\),
then either \(\xi = 0\) or \(v = 0\).
\end{myexample}

\begin{myexample}[Cauchy-Riemann operator]
  The \emph{Cauchy-Riemann operator} \(A (\Deriv) = \Bar{\partial}\), which is the differential operator from \(\Cset\) to \(\Cset\) on \(\Rset^2\) defined for each function \(u \in C^\infty (\Rset^2, \Cset)\) by \(A (\Deriv) u = (\partial_1 u + i \partial_2 u)/2\) is injectively elliptic. Indeed, for every \(\xi = (\xi_1, \xi_2) \in \Rset^2\) and \(v \in \Cset\), if \(A (\xi) v = (\xi_1 + i \xi_2) v/2 =0\) then either \(\xi = 0\) or \(v = 0\).
\end{myexample}

\begin{myexample}[Hodge complex]
  The \emph{Hodge complex operator} \(A (\Deriv) = (d, d^*)\) on differential forms from \(\smash{\bigwedge^m \Rset^n}\) to
  \(\smash{\bigwedge^{m + 1} \Rset^n \times \bigwedge^{m - 1} \Rset^n}\), with \(m \in \set{1, \dotsc, n - 1}\) is injectively elliptic. Here \(d\) and \(d^*\) denote respectively the exterior differential and codifferential.
Indeed, for every \(\xi \in \Rset^n\) and \(v \in \bigwedge ^m \Rset^n\), if 
\(A (\xi) [v] = (\xi \wedge v, \xi \lrcorner v) = 0\),
then  one has the Lagrange identity \(\abs{A (\xi)[v]}^2 = 
\abs{\xi \wedge v}^2 + \abs{\xi \lrcorner v}^2 = \abs{\xi}^2 \abs{v}^2\) and thus 
either \(\xi = 0\) or \(v = 0\).
\end{myexample}

\begin{myexample}[Symmetric derivative]
\label{test}
  The \emph{symmetric derivative,} also known as \emph{deformation operator,} defined for every \(u \in C^\infty(\Rset^n, \Rset^n)\) and every \(x \in \Rset^n\) by \(A (\Deriv) u (x) \defeq \smash{\Deriv_{\mathrm{sym}} u (x)} \defeq (Du (x) + Du (x)^*)/2 \in \Lin(\Rset^n, \Rset^n)\) is injectively elliptic.
  Indeed if \(\xi \in \Rset^n\) and \(v \in \Rset^n\) satisfy \(A (\xi)[v] = (v \otimes \xi + \xi \otimes v)/2 = 0\), then \(\abs{A(\xi)[v]}^2 =  (\abs{v}^2 \abs{\xi}^2 + \dualprod{\xi}{v}^2)/2=0\), and hence either \(\xi = 0\) or \(v = 0\).
\end{myexample}

\begin{myexample}[Trace-free symmetric derivative]
The operator \(A (\Deriv) u \defeq \Deriv_\mathrm{sym} u  + \lambda \operatorname{div} u \operatorname{id}\)
  is injectively elliptic if and only if \(n \ge 2\) or \(n = 1\) and \(\lambda \ne -1\).
When \(\lambda = -1/n\), the operator \(A(\Deriv)\) is the \emph{trace-free symmetric derivative} which is injectively elliptic if and only if \(n \ge 2\).  
To prove the injective ellipticity, we have for every \(\xi \in \Rset^n\) and \(v \in \Rset^n\), \(A (\xi)[v] = (v \otimes \xi + \xi \otimes v)/2 + \lambda \dualprod{\xi}{v}\operatorname{id}\)
and thus \(\smash{\abs{A (\xi)[v]}^2} = \abs{\xi}^2 \abs{v}^2/2 + (1/2 + 2 \lambda + n \lambda^2) \dualprod{\xi}{v}^2 \).
We note that  \(1 + 2 \lambda + n\lambda^2= (1 + \lambda)^2 + (n-1) \lambda^2 > 0\) if and only if either \(n = 1\) and \(\lambda \ne -1\) or \(n \ge 2\); in these cases if \(A (\xi)[v] = 0\) we have either \(\xi = 0\) or \(v = 0\), and the operator \(A (\Deriv)\) is injectively elliptic.
\end{myexample}

Injective ellipticity turns out to be always a necessary condition to obtain \(L^p\) estimates on \(\Deriv^k u\).

\begin{theorem}[Neccessity of the injective ellipticity for equivalence of norms]
\label{theorem_Lp_injective_elliptic_necessary}
  Let \(n \in \Nset \setminus \set{0}\), let \(V\) and \(E\) be finite-dimensional vector spaces, let \(A (\Deriv)\) be a homogeneous constant coefficient differential operator
of order \(k \in \Nset \setminus \set{0}\) from 
\(V\) to \(E\) on \(\Rset^n\), and let \(p \in [1, +\infty)\).
If there exists a constant \(C \in \intvo{0}{+\infty}\) such that for every \(u \in C^\infty_c (\Rset^n, V)\), 
  \begin{equation}
  \label{eq_TajoJohmeiniep1ulo0ieho6}
  \int_{\Rset^n} \abs{\Deriv^k u}^p
  \le 
  C
  \int_{\Rset^n} \abs{A (\Deriv)[u]}^p\;,
  \end{equation}
  then the operator \(A (\Deriv)\) is injectively elliptic.
\end{theorem}

\begin{proof}
We fix a function \(\eta \in C^\infty_c (\Rset^n, \Rset)\) such that \(\eta = 1\) on the unit ball \(B_1 (0) \subseteq \Rset^n\), and we define for every \(v \in V\) and 
  \(\xi \in \Rset^n\), the function \(u_{\xi, v} : \Rset^n \to V\) for each \(x \in \Rset^n\) by 
  \begin{equation*}
    u_{\xi, v} (x) \defeq v \sin (\dualprod{\xi}{x}) \eta (x)\;.
  \end{equation*}
We compute, since \(\eta = 1\) on \(B_1 (0)\), 
\begin{equation}
  \label{eq_shoeki4ju1Kei2ito}
\begin{split}
    \int_{\Rset^n} \abs{\Deriv^k u_{\xi, v}}^p \dif x
  \ge 
    \int_{B_1 (0)} \abs{\Deriv^k u_{\xi, v}}^p
  &= 
  \int_{B_1 (0)} \abs{v}^p \abs{\xi}^{kp} \abs{\sin^{(k)} \brk{\dualprod{\xi}{x}}}^p \dif x\\
  &\ge \Cl{cst_ohphohQuin2oxi6uPhe0iewa} \abs{v}^p \abs{\xi}^{k p}\;,
\end{split}
\end{equation}
if \(\abs{\xi} \ge 1\), for some constant \(\Cr{cst_ohphohQuin2oxi6uPhe0iewa} \in \intvo{0}{+\infty}\).
On the other hand, we have by \cref{lemma_AD_Leibnitz} below
\begin{equation}
  \label{eq_aeGh4Tie4voo2Dohw}
\int_{\Rset^n} \abs{A (\Deriv) u_{\xi, v}}^p \dif x
\le 
\C
\int_{\Rset^n} \abs{\eta}^p \abs{A (\xi)[v]}^p 
+ \abs{v}^p \sum_{j = 1}^k \abs{\Deriv^j \eta}^p \abs{\xi}^{(k - j)p}\;.
\end{equation}
When \(\abs{\xi}\) is large enough we deduce from our assumption \eqref{eq_TajoJohmeiniep1ulo0ieho6} and 
from the estimates \eqref{eq_shoeki4ju1Kei2ito} and \eqref{eq_aeGh4Tie4voo2Dohw} that
\begin{equation}
\label{eq_io3thui2aeLoh4ieTiash0fe}
    \abs{v}^p \abs{\xi}^{kp} 
  \le 
    \C
    \abs{A (\xi) [v]}^p\;.
\end{equation}
By linearity of \(A\), the inequality \eqref{eq_io3thui2aeLoh4ieTiash0fe} holds for every \(\xi \in \Rset^n\) and the operator \(A (\Deriv)\) is injectively elliptic in view of \cref{definition_injectively_elliptic}.
\end{proof}

In the proof of \cref{theorem_Lp_injective_elliptic_necessary}, we used the fact that commutators between \(A (\Deriv)\) and a multiplication only involve lower-order derivatives of \(u\):

\begin{lemma}[Generalized Leibnitz rule]
  \label{lemma_AD_Leibnitz}
Let \(n \in \Nset \setminus \set{0}\), let \(V\) and \(E\) be finite-dimensional vector spaces, and let \(A (\Deriv)\) be a homogeneous constant coefficient differential operator
of order \(k \in \Nset \setminus \set{0}\) from \(V\) to \(E\) on \(\Rset^n\).
 For every \(u \in C^\infty (\Rset^n, V)\) and every \(\varphi \in C^\infty (\Rset^n, E)\), one has
\begin{equation}
\label{eq_eec4Phaomeivuu6vu5ar2yoo}
  A (\Deriv) (\varphi u)-  
  \varphi A (\Deriv) u = \sum_{j = 1}^k \tbinom{k}{j} A \brk[\big]{\operatorname{Sym} (\Deriv^j \varphi \otimes \Deriv^{k - j} u)}\;.
\end{equation}
\end{lemma}

Here, \(\operatorname{Sym}\) denotes the symmetrization of tensors.

\begin{proof}[Proof of \cref{lemma_AD_Leibnitz}]
From the generalized Leibnitz formula, we have
\begin{equation*}
 \Deriv^k (\varphi u)
 = \sum_{j = 0}^k \tbinom{k}{j} \operatorname{Sym} \brk{\Deriv^j \varphi \otimes \Deriv^{k - j} u}\;,
\end{equation*}
and the conclusion \eqref{eq_eec4Phaomeivuu6vu5ar2yoo} follows.
\end{proof}

\subsection{Fourier Analysis and \texorpdfstring{\(L^2\)}{L²} Estimates}

The injective ellipticity is necessary and sufficient for having an \(L^2\) estimate of the \(k\)-th order total derivative.

\begin{theorem}[\(L^2\) estimate for injectively elliptic operators]
  \label{theorem_L2}
Let \(n \in \Nset \setminus \set{0}\), 
 let \(V\) and \(E\) be finite-dimensional vector spaces,
and let \(A (\Deriv)\) be a homogeneous constant coefficient differential operator of order \(k \in \Nset \setminus \set{0}\) from \(V\) to \(E\) on \(\Rset^n\). 
If \(A (\Deriv)\) is injectively elliptic, 
then there exists a constant \(C \in \intvo{0}{+\infty}\) such that for every \(u \in C^\infty_c (\Rset^n, V)\), 
\begin{equation*}
\int_{\Rset^n} \abs{\Deriv^k u}^2
\le C
\int_{\Rset^n} \abs{A (\Deriv)[u]}^2\;.
\end{equation*}
\end{theorem}

The proof of \cref{theorem_L2} will rely on the following quantitative reformulation of injective ellipticity.

\begin{lemma}[Quantitative characterization of injective ellipticity]
  \label{lemma_right_elliptic_inverse_ineq}
  Let \(n \in \Nset \setminus \set{0}\), let \(V\) and \(E\) be finite-dimensional vector spaces,
and let \(A (\Deriv)\) be a homogeneous constant coefficient differential operator
of order \(k \in \Nset \setminus \set{0}\) from 
\(V\) to \(E\) on \(\Rset^n\).
The operator \(A (\Deriv)\) is injectively elliptic if and only if there exists a constant \(C \in \intvo{0}{+\infty}\) such that for
every \(\xi \in \Rset^n\) and every \(v \in V\),
\begin{equation}
\label{eq_aiv9tohRohSeac4heeXahLi2}
\abs{\xi}^k \abs{v} \le C \abs{A (\xi)[v]}\;.
\end{equation}
\end{lemma}

\begin{proof}
If the inequality \eqref{eq_aiv9tohRohSeac4heeXahLi2} is satisfied and if \(\xi \in \Rset^n\setminus \set{0}\), then the inequality \eqref{eq_aiv9tohRohSeac4heeXahLi2} implies that 
the mapping \(A (\xi)\) is injective on \(V\). 

Conversely, if the operator \(A (\Deriv)\) is injectively elliptic, then the function \((\xi, v) \in K \mapsto \abs{A (\xi)[v]} \in \Rset\)
  is continuous and positive on the compact set \(K \defeq \set{(\xi, v) \in \Rset^n \times V \st \abs{\xi} = \abs{v} = 1}\),
hence it is bounded from below by a positive constant on that set. 
By homogeneity and linearity, we have for every \(\xi \in \Rset^n \setminus \set{0}\) and \(v \in V \setminus \set{0}\)
\begin{equation*}
  \abs{A (\xi)[v]} = \abs{A (v \otimes \xi^{\otimes k})} = 
  \abs{\xi}^k \abs{v}\abs[\big]{A \brk[\big]{\tfrac{v}{\abs{v}} \otimes (\tfrac{\xi}{\abs{\xi}})^{\otimes k}}} 
  = 
  \abs{\xi}^k \abs{v}\abs[\big]{A \bigl(\tfrac{\xi}{\abs{\xi}}\bigr)\bigl[\tfrac{v}{\abs{v}}\bigr]}\;
\end{equation*}
and the conclusion \eqref{eq_aiv9tohRohSeac4heeXahLi2} follows.
\end{proof}

We are now in position to prove the \(L^2\) estimate on \(\Deriv^k u\) for injectively elliptic operators of \cref{theorem_L2}.

\begin{proof}[Proof of \cref{theorem_L2}]
  By Parseval's identity for the Fourier transform and by \cref{lemma_right_elliptic_inverse_ineq}, we
have for some constant \(\Cl{cst_hei4SiePh7lae6goo0ugooch} \in \intvo{0}{+\infty}\)
\begin{equation*}
  \begin{split}
\int_{\Rset^n} \abs{\Deriv^k u}^2
&= \brk{2\pi}^{2k} \int_{\Rset^n} \abs{u(\xi)}^2\abs{\xi}^{2k} \dif \xi\\
&\le \Cr{cst_hei4SiePh7lae6goo0ugooch}^2 \brk{2\pi}^{2k}\int_{\Rset^n} \abs{A(\xi)[u (\xi)]}^2 \dif \xi
= \Cr{cst_hei4SiePh7lae6goo0ugooch}^2 \int_{\Rset^n} \abs{A (\Deriv)[u]}^2\;,
\end{split}
\end{equation*}
and the conclusion follows.
\end{proof}

\subsection{Representation Kernel}

In order to extend \cref{theorem_L2} to \(L^p\), with \(p \in \intvo{1}{+\infty} \setminus \set{2}\), we will rely on a representation formula of \(u\) by \(A (\Deriv) u\).
In order to construct such a formula, one can first observe on the Fourier frequency domain that since the symbol \(A (\xi)\) is injective on \(V\) for every \(\xi \in \Rset^n \setminus \set{0}\), it should be possible to recover \(u\) from \(A (\Deriv) u\) by the identity \eqref{eq_du0zoovuociuBi9nu}.
In order to implement this, we start from the observation that for every \(\xi \in \Rset^n \setminus \set{0}\),
the operator \(A (\xi)^* \circ A (\xi) : V \to V\) is invertible, where \(A (\xi)^*:E \to V\) denotes the adjoint operator of \(A (\xi)\),\footnote{
The definition of the adjoint \(A (\xi)^*\) of \(A (\xi)\) and hence of its pseudo-inverse \(A (\xi)^\dagger\) in \eqref{eq_oohie4sohseiFiv8} depends on the choice of Euclidean structures on the spaces \(V\) and \(E\). These structures allow to choose consistently a left-inverse and the final results will always be independent on this arbitrary choice of Euclidean structures.
} so that \(A (\xi)\) has a well-defined pseudo-inverse
\begin{equation}
  \label{eq_oohie4sohseiFiv8}
A (\xi)^\dagger \defeq \brk[\big]{A (\xi)^* \circ A (\xi)}^{-1} \compose A (\xi)^*:E \to V\;,
\end{equation}
--- this is the classical construction for least-square solutions of linear systems through normal equations --- and the symbol \(A (\xi)^\dagger\) satisfies 
\begin{equation*}
A (\xi)^\dagger \compose A (\xi) = \operatorname{id}_V\;.
\end{equation*}
The next natural step would be to define a representation kernel (or Green function) \(G_A : \Rset^n \setminus\set{0} \to \Lin(E, V)\) through the requirement that for every \(\xi \in \Rset^n \setminus \set{0}\)
\begin{equation*}
  \mathcal{F} (G_A \ast u) (\xi) = 
  (\mathcal{F} G_A) (\xi) [\mathcal{F}(u) (\xi)] =  (2 \pi i)^k A (\xi)^\dagger [\mathcal{F} (A (\Deriv)u) (\xi)]\;,
\end{equation*}
and therefore, with a suitable inversion formula, to define 
\begin{equation}
  \label{eq_jeiNg2ohpheaquaSo}
  G_A \defeq \frac{1}{(2 \pi i)^k} \mathcal{F}^{-1} A^\dagger\;,
\end{equation}
The problem is that in general \(A^\dagger\) in the right-hand side of \eqref{eq_jeiNg2ohpheaquaSo} is not even a distribution on \(\Rset^n\) when \(k \ge n\). Through a careful analysis, the next proposition bypasses this mathematical obstacle to construct a representation kernel.

\begin{proposition}[%
Representation kernel of an injectively elliptic operator%
\footnote{
The construction of the representation kernel in \cref{proposition_fundamental_solution} for injectively elliptic operators is due to
\citelist{\cite{Bousquet_VanSchaftingen_2014}*{Lem.\ 2.1}\cite{Raita_2019}}.
We give here a self-contained construction of the Green function, based on the extension of the homogeneous distribution \(A (\xi)^\dagger\) from \(\Rset^n \setminus \set{0}\) to homogeneous distributions on \(\Rset^n\) \cite{Hormander_1990_I}*{Th. 3.2.3 \& 3.2.4}, the preservation of the smoothness in the extension and its temperate character \cite{Hormander_1990_I}*{Th.\ 7.1.18},
and the homogeneity of the resulting Fourier transform \cite{Hormander_1990_I}*{Th.\ 7.1.16}.
}%
]
  \label{proposition_fundamental_solution} 
Let \(n \in \Nset \setminus \set{0}\), let \(V\) and \(E\) be finite-dimensional vector spaces, 
and let \(A (\Deriv)\) be a homogeneous constant coefficient differential operator
of order \(k \in \Nset \setminus \set{0}\) from 
\(V\) to \(E\) on \(\Rset^n\).
If the operator \(A (\Deriv)\) is injectively elliptic, then there exists a function \(G_A \in C^\infty 
  (\Rset^n \setminus \set{0}, \Lin (E, V)) \cap L^1_{\mathrm{loc}} (\Rset^n \setminus \set{0}, \Lin (E, V))\), such that 
  \begin{enumerate}[(i)]
    \item
    \label{it_iez1EicaeNaanopipheyaeru}
      for every \(u \in \mathcal{S} (\Rset^n, V)\), 
    \begin{equation}
    \label{eq_jal9Niez6Vu2ohved5OorooG}
    u (x) = \int_{\Rset^n} G_A (x - y)[A (\Deriv) u (y)] \dif y\;,
    \end{equation}
    \item 
    \label{it_oocee4ao4DaiK8sheWeechee}
 for every \(x \in \Rset^n\) and every \(\lambda \in \Rset \setminus \set{0}\), 
    \[
    G_A (\lambda x) = \lambda^{k - n} \brk[\big]{G_A (x) - \ln \abs{\lambda} \, P_A (x)}\;,
    \]
    where the function \(P_A : \Rset^n \to \Lin (E, V)\) is a homogeneous polynomial of degree \(k - n\) when \(k \ge n\) and is \(0\) when \(k < n\),
        \item
    \label{it_Zigh3aiy3joo5maewiechua6}
    if \(\ell \in \Nset\) satisfies \(\ell > k - n\), then 
    \[
      \int_{\Sset^{n - 1}}  \Deriv^\ell G_A (x) \dif x = 0\;,
    \]
    \item \label{it_wee9airoSei7iex2ude} given \(f \in \mathcal{S} (\Rset^n, E)\) such that for every \(\xi \in \Rset^n\), \(\mathcal{F} f (\xi) \in A (\xi)[V + iV]\), the function \(u \in C^\infty (\Rset^n, V)\) defined for each \(x \in \Rset^n\) by 
    \begin{equation}
    \label{eq_au0uuzoghaem0aeN3aeGhah6}
    u (x) \defeq \int_{\Rset^n} G_A (x - y)[f (y)] \dif y\;,
    \end{equation}
    satisfies \(A (\Deriv) u = f\),
    \item 
    \label{it_Aix3Ahshe4ahchiekai4Ieph}
    for every \(e \in \bigcap_{\xi \in \Rset^n \setminus \set{0}} A (\xi)[V]\) we have \(A (\Deriv) G_A[e] = e \delta_0\) on \(\Rset^n\) in the sense of distributions,
    \item 
    \label{it_supahJuacocei6efiey7aife}
    if  \(k \ge n\), then for every \(x \in \Rset^n\) and for every \(e \in \bigcap_{\xi \in \Rset^n \setminus \set{0}} A (\xi)[V]\), we have
    \[
    P_A (x) [e] = \int_{\Sset^{n - 1}} \frac{\dualprod{\xi}{x}^{k - n} A (\xi)^{-1} [e]}{(k - n)!(2 \pi i)^n} \dif \xi\;.
    \]
  \end{enumerate}
\end{proposition}

\begin{proof}
Fixing an even function \(\psi \in C^\infty_c(\Rset^n, \Rset)\) such that \(\psi=0\) in a neighbourhood of \(0\) and such that for every \(x \in \Rset^n \setminus \set{0}\),
\begin{equation}
  \label{eq_ahpeithouG3yuothai5laete}
 \int_{0}^{+\infty} \frac{\psi(x/t)}{t} \dif t = 1\;,
\end{equation}
we define the function \(H_A : \Rset^n \to \Lin (E +i E, V+ iV)\)
for every \(x \in \Rset^n\) by
\begin{equation}
  \label{eq_teek1coong8vi7Thoh1ohr4d}
 H_A (x)
 \defeq \int_{\Rset^n} \frac{e^{2\pi i \dualprod{\xi}{x}}}{(2 \pi i)^k} \psi (\xi) A (\xi)^\dagger \dif \xi\;,
\end{equation}
where \(A (\xi)^\dagger \in \Lin(E,V)\) is the left-inverse of \(A (\xi)\) that was defined in \eqref{eq_oohie4sohseiFiv8}, which is well-defined since the operator \(A (\Deriv)\) is injectively elliptic (\cref{definition_injectively_elliptic}).
Since \(\psi\) is an even function, for every \(x \in \Rset^n\) 
we have by \eqref{eq_teek1coong8vi7Thoh1ohr4d}, when \(k \in \Nset \setminus \set{0}\) is an even number,
\begin{equation}
  \label{eq_caefaicah8ieT4nee8oomaah}
 H_A (x)
 = \int_{\Rset^n} \frac{\cos \brk{2\pi \dualprod{\xi}{x}}}{(2 \pi)^k (-1)^{k/2}} \psi (\xi) A (\xi)^\dagger \dif \xi\;,
\end{equation}
since \(\psi A^\dagger\) is then an even function, 
and, when \(k \in \Nset \setminus \set{0}\) is an odd number,
\begin{equation}
  \label{eq_negua9AhHoothie0uKubusor}
  H_A (x)
  = \int_{\Rset^n} \frac{\sin \brk{2\pi \dualprod{\xi}{x}}}{(2 \pi)^k (-1)^{(k-1)/2}} \psi (\xi) A (\xi)^\dagger \dif \xi\;,
\end{equation}
since \(\psi A^\dagger\) is then an odd function. It follows from \eqref{eq_caefaicah8ieT4nee8oomaah} or \eqref{eq_negua9AhHoothie0uKubusor} that \(H_A (x) \in \Lin (E, V)\) and that for every \(x \in \Rset^n\),
\begin{equation}
  \label{eq_reiciet1soo3rahm2Ash0quo}
 H_A (-x) = (-1)^k H_A (x)\;.
\end{equation}
By properties of Fourier transforms, we also have \(H_A \in \mathcal{S} (\Rset^n, \Lin (E, V))\).
In particular, for every \(x \in \Rset^n \setminus \set{0}\), we have \(t \in \Rset \mapsto H_A (t x) \in \mathcal{S} (\Rset, \Lin (E, V))\).

When \(k < n\), we define the function \(G_A : \Rset^n \setminus \set{0} \to \Lin (E, V)\) for each \(x \in \Rset^n \setminus \set{0}\) by 
\begin{equation}
  \label{eq_Ejee3jaeghoht2weez8ae1ru}
G_A (x) 
\defeq 
    \int_0^{+\infty} 
    H_A (t x)
    \,
      t^{n - k - 1} \dif t\;.
\end{equation}
If \(u \in \mathcal{S} (\Rset^n, V)\), then we have by \eqref{eq_Ejee3jaeghoht2weez8ae1ru} 
\begin{equation}
\label{eq_eith1seejohvaingai9jaiCa}
\begin{split}
  \int_{\Rset^n} G_A (x - y)& [A(\Deriv) u (y)] \dif y\\
  &= \int_{\Rset^n} \int_0^{+\infty} 
    H_A (t (x - y))
    \,
      t^{n - k - 1} \dif t [A (\Deriv) u (y)] \dif y\\
  &= \int_0^{+\infty} \int_{\Rset^n} H_A (t (x - y))[A(\Deriv)u(y)] \dif y \, t^{n - k - 1} \dif t\;.
\end{split}
\end{equation}
Next, we compute by the definition of \(H_A\) in  \eqref{eq_teek1coong8vi7Thoh1ohr4d}
\begin{equation*}
\begin{split}
 \int_{\Rset^n} H_A (t (x - y))&[A(\Deriv)u(y)] \dif y\\
 & = \int_{\Rset^n} \int_{\Rset^n} \frac{e^{2\pi i \dualprod{\xi}{t (x - y)}}}{(2 \pi i)^k} \psi (\xi) A (\xi)^\dagger[A(\Deriv) u (y)] \dif \xi \dif y\\
 &= \int_{\Rset^n} \int_{\Rset^n} \frac{e^{2\pi i \dualprod{t \xi}{x - y}}}{(2 \pi i)^k} \psi (\xi) A (\xi)^\dagger[A(\Deriv) u (y)] \dif y \dif \xi,
 \end{split}
 \end{equation*}
so that by a change of variable \(\zeta = t\xi\)
\begin{equation}
\label{eq_Maong2ied2aiweigie4feeng}
\begin{split}
 \int_{\Rset^n} H_A (t (x - y))&[A(\Deriv)u(y)] \dif y\\
 & = \int_{\Rset^n} e^{2\pi i \dualprod{t \xi}{x}} \psi (\xi) A (\xi)^\dagger [A(t \xi) [\mathcal{F} u (t \xi)]] \dif \xi\\
 &= t^{k - n} \int_{\Rset^n} e^{2\pi i \dualprod{\zeta}{x}} \psi (\zeta/t) \mathcal{F} u (\zeta) \dif \zeta\;.
\end{split}
\end{equation}
Hence by \eqref{eq_eith1seejohvaingai9jaiCa} and \eqref{eq_Maong2ied2aiweigie4feeng}, we get 
by \eqref{eq_ahpeithouG3yuothai5laete} and the Fourier inversion formula. 
\begin{equation}
\begin{split}
\int_{\Rset^n} G_A (x - y) [A(\Deriv) u (y)] \dif y &= \int_{0}^{+\infty} \int_{\Rset^n} e^{2\pi i \dualprod{\zeta}{x}} \psi (\zeta/t) \mathcal{F} u (\zeta) \dif \zeta \frac{\dif t }{t}\\
& = \int_{\Rset^n} \mathcal{F} u (\zeta) \dif \zeta = u (x)\;,
\end{split}
\end{equation}
We have thus proved that \eqref{it_iez1EicaeNaanopipheyaeru} holds when \(k < n\).

Differentiating \eqref{eq_Ejee3jaeghoht2weez8ae1ru}, we have for every \(\ell \in \Nset\) that \(G_A \in C^\ell \brk{\Rset^n \setminus \set{0}, \Lin (E, V)}\), with
for every \(x \in \Rset^n \setminus \set{0}\),
\begin{equation}
  \label{eq_shixae4newoh5ahV4}
\Deriv^\ell G_A (x) =  \int_0^{+\infty} 
\Deriv^\ell H_A (t x) \,
t^{n + \ell - k - 1} \dif t\;.
\end{equation}
For every \(\lambda \in (0, +\infty)\), we have by 
\eqref{eq_Ejee3jaeghoht2weez8ae1ru}  and by the change of variable \(s = \lambda t\),
\begin{equation}
  \label{eq_ooy9ofi8saij1Iec8eiYei7A}
\begin{split}
 G (\lambda x) &=
 \int_0^{+\infty} \hspace{-1em}H_A (t \lambda x) \,t^{n - k - 1} \dif t = \frac{1}{\lambda^{n - k}} \int_0^{+\infty} \hspace{-1em} H_A (s x) \, s^{n - k  - 1} \dif s = \frac{G_A (x)}{\lambda^{n - k}}\;,
 \end{split}
\end{equation}
so that the assertion \eqref{it_oocee4ao4DaiK8sheWeechee} holds for \(k < n\) in view of \eqref{eq_reiciet1soo3rahm2Ash0quo} and 
\eqref{eq_ooy9ofi8saij1Iec8eiYei7A}.

When \(k\ge n\), we define the function \(G_A : \Rset^n \setminus \set{0} \to \Lin (E, V)\) for each \(x \in \Rset^n \setminus \set{0}\) by 
\begin{multline}
  \label{eq_oonee0oogai3ieM2veiy2oow}
  G_A (x) 
  \defeq 
  \int_0^{+\infty} 
  \frac{\dualprod{\Deriv^{k - n + 1} H_A (t x)}{x^{\otimes k - n  + 1}}}{(k - n)!} \ln \frac{1}{t} \dif t \\
  + \frac{\dualprod{\Deriv^{k - n} H_A (0)}{x^{\otimes k - n}}}{(k - n)!} \sum_{j = 1}^{k - n} \frac{1}{j}\; .
\end{multline}
Assuming that \(x = 0\), we first have by \eqref{eq_oonee0oogai3ieM2veiy2oow}
\begin{equation}
\label{eq_deehi1laeth9rei2deuKeera}
\begin{split}
  \int_{\Rset^n}  &G_A (x-y) [A(\Deriv) u (y)] \dif y\\ 
   &=
  \int_0^{+\infty} 
  \int_{\Rset^n}
  \frac{
  \dualprod{\Deriv^{k - n + 1} H_A (-t y)}{(- y)^{\otimes k - n  + 1}}[A (\Deriv) u (y)]}{(k - n)!} \dif y \ln \frac{1}{t} \dif t\\
  &\quad + \int_{\Rset^n} 
  \frac{\dualprod{\Deriv^{k - n} H_A (0)}{(-y)^{\otimes k - n}}[A (\Deriv) u(y)]}{(k - n)!}\dif y \sum_{j = 1}^{k - n} \frac{1}{j}\;.
\end{split}
\end{equation} 
In order to compute out the first-term in the right-hand side of \eqref{eq_deehi1laeth9rei2deuKeera}, we first work out the inner integral for every \(t \in \intvo{0}{+\infty}\) in view of \eqref{eq_teek1coong8vi7Thoh1ohr4d} as
\begin{equation}
\label{eq_Cie1Ge3ahyeu9eThayidethi}
\begin{split}
 \int_{\Rset^n}&
  \dualprod{\Deriv^{k - n + 1} H_A (- ty)}{(- y)^{\otimes k - n  + 1}}[A (\Deriv) u (y)] \dif y \\
  &=\int_{\Rset^n}\int_{\Rset^n} \frac{e^{- 2\pi i \dualprod{\xi}{t y}}}{(2 \pi i)^{n - 1}} (-\dualprod{\xi}{y})^{k - n + 1} \psi (\xi) A (\xi)^\dagger[A(\Deriv) u (y)]\dif y \dif \xi\\
  &= \int_{\Rset^n}\int_{\Rset^n} \frac{e^{-2\pi i \dualprod{t \xi}{y}}}{(2 \pi i)^{n - 1}} \psi (\xi) A (\xi)^\dagger[ (-\dualprod{\xi}{y})^{k - n + 1} A(\Deriv) u (y)]\dif y \dif \xi\\
  &=  \int_{\Rset^n}\frac{\psi (\xi) A (\xi)^\dagger[\Deriv^{k - n + 1} \mathcal{F} (A(\Deriv)u) (t \xi) 
  [\xi^{\otimes k - n +1}]]}{(2 \pi i)^k}\dif \xi\;.
\end{split}
\end{equation}
By \(k - n + 1\) successive integration by parts we have for every \(\xi \in \Rset^n\)
\begin{equation}
\label{eq_eilooLiPho9XooW8es5taiYe}
\begin{split}
\int_0^{+\infty} 
\frac{\Deriv^{k - n + 1} \mathcal{F} (A(\Deriv)u) (t \xi) 
  [\xi^{\otimes k - n +1}]}{(k - n)!} & \ln \frac{1}{t} \dif t \\
  &= \int_0^{+\infty} \frac{\mathcal{F} (A(\Deriv)u) (t \xi)}{t^{k - n + 1}} \dif t\\
  &= \int_0^{+\infty} \frac{(2 \pi i)^k A(t \xi)\mathcal{F} u(t \xi)}{t^{k - n + 1}} \dif t \;,
\end{split}
\end{equation}
and thus by  \eqref{eq_Cie1Ge3ahyeu9eThayidethi} and \eqref{eq_eilooLiPho9XooW8es5taiYe}, we get by a change of variable \(\zeta = t \xi\)
\begin{equation}
\label{eq_Ahthohtetai2uoyaiquoo0Sa}
\begin{split}
\int_{\Rset^n}&
  \dualprod{\Deriv^{k - n + 1} H_A (- ty)}{(- y)^{\otimes k - n  + 1}}[A (\Deriv) u (y)] \dif y \\
&
\qquad\qquad\qquad\qquad 
=\int_{\Rset^n} \int_0^{+\infty} \frac{\psi (\xi) A(t \xi)^\dagger[A (\xi)[\mathcal{F} u(t \xi)]]}{t^{k -n + 1}} \dif t \dif \xi\\
&
\qquad\qquad\qquad\qquad 
=  \int_{\Rset^n} \int_0^{+\infty} \frac{\psi (\zeta/t) \mathcal{F} u(\zeta)}{t} \dif t \dif \zeta  = u (0)\;,
\end{split}
\end{equation}
in view of \eqref{eq_ahpeithouG3yuothai5laete}.
It remains to evaluate the second term in the right-hand side of \eqref{eq_deehi1laeth9rei2deuKeera}: by integration by parts we have 
\begin{equation}
\label{eq_jautooNg2OGhah0hoo3uroor}
\int_{\Rset^n} 
  \frac{\dualprod{\Deriv^{k - n} H_A (0)}{(-y)^{\otimes k - n}}[A (\Deriv) u(y)]}{(k - n)!}\dif y = 0\;,
\end{equation}
since \(A (\Deriv)\) is a homogeneous differential operator of order \(k \in \Nset \setminus \set{0}\).
Combining \eqref{eq_deehi1laeth9rei2deuKeera}, \eqref{eq_Ahthohtetai2uoyaiquoo0Sa} and \eqref{eq_jautooNg2OGhah0hoo3uroor}, we have proved \eqref{it_iez1EicaeNaanopipheyaeru} when \(k \ge n\).

Next, we claim that for every \(\ell \in \set{0, \dotsc, k - n}\) and \(x \in \Rset^n\), we have
\begin{multline}
  \label{eq_si2eiG6Mee5raeb1}
  \Deriv^\ell 
  G_A (x) 
  = 
  \int_0^{+\infty} 
   \frac{\dualprod{\Deriv^{k - n + 1} H_A (t x)}{x^{\otimes k - n  - \ell + 1}}}{(k - n - \ell)!} \ln \frac{1}{t} \dif t \\
  +  \frac {\dualprod{\Deriv^{k - n} H_A (0)}{x^{\otimes k - n -\ell}}}{(k - n - \ell)!} \sum_{j = 1}^{k - n - \ell} \frac{1}{j}\;.
\end{multline}
Indeed, for \(\ell = 0\), \eqref{eq_si2eiG6Mee5raeb1} reduces to \eqref{eq_oonee0oogai3ieM2veiy2oow}; 
if we assume that \eqref{eq_si2eiG6Mee5raeb1} holds for some \(\ell \in \set{0, \dotsc, k - n - 1}\) and if we differentiate \eqref{eq_si2eiG6Mee5raeb1} we obtain
\begin{equation}
  \label{eq_gu9Cho1EihaiGhef}
  \begin{split}
  \Deriv^{\ell + 1} &
  G_A (x)\\
  =& 
  \int_0^{+\infty} \Bigl(\frac{\dualprod{\Deriv^{k - n + 1} H_A (t x)}{x^{\otimes k - n  - \ell}}}{(k - n - \ell - 1)!} + \frac{\dualprod{\Deriv^{k - n + 1} H_A (t x)}{x^{\otimes k - n  - \ell}}} {(k - n - \ell)!}\\
  &\hspace{10em} \qquad + \frac{\dualprod{\Deriv^{k - n + 2} H_A (t x)}{x^{\otimes k - n  - \ell + 1}}t} {(k - n - \ell)!} \Bigr)\ln \frac{1}{t} \dif t\\
  &+ \frac{\dualprod{\Deriv^{k - n} H_A (0)}{x^{\otimes k - n - \ell - 1}}} {(k - n - \ell - 1)!} \sum_{j = 1}^{k - n - \ell} \frac{1}{j} \;;
\end{split}
\end{equation}
integrating by parts the last term in the integrand on the right-hand side of \eqref{eq_gu9Cho1EihaiGhef}, we have 
\begin{equation}
  \label{eq_johD5Zahvo8ieThi}
  \begin{split}
  \int_0^{+\infty} \hspace{-.8em}&
  \dualprod{\Deriv^{k - n + 2} H_A (t x)}{x^{\otimes k - n  - \ell + 1}} t  \ln \frac{1}{t} \dif t\\
   &= \int_0^{+\infty} 
  \dualprod{\Deriv^{k - n + 1} H_A (t x)}{x^{\otimes k - n  - \ell}} \brk[\Big]{1 - \ln \frac{1}{t}} \dif t\\
  &= - \int_0^{+\infty} \hspace{-.9em}
  \dualprod{\Deriv^{k - n + 1} H_A (t x)}{x^{\otimes k - n  - \ell}} \ln \frac{1}{t} \dif t - \dualprod{\Deriv^{k - n} H_A (0) }{x^{\otimes k - n - \ell - 1}}\;,
  \end{split} 
\end{equation}
and it follows then from \eqref{eq_gu9Cho1EihaiGhef} and \eqref{eq_johD5Zahvo8ieThi} that 
\begin{multline*}
  \Deriv^{\ell + 1} 
  G_A (x) 
  = 
   \int_0^{+\infty} 
  \frac{\dualprod{\Deriv^{k - n + 1} H_A (t x)}{x^{\otimes k - n  - \ell}}}{(k - n - \ell - 1)!} \ln \frac{1}{t} \dif t \\
  + \frac{\dualprod{\Deriv^{k - n} H_A (0)}{x^{\otimes k - n -\ell - 1}}}{(k - n - \ell - 1)!} \sum_{j = 1}^{k - n - \ell - 1} \frac{1}{j} \;,
\end{multline*}
and thus \eqref{eq_si2eiG6Mee5raeb1} holds for each \(\ell \in \set{0,\dotsc, n - k }\) by induction.

Differentiating \eqref{eq_si2eiG6Mee5raeb1} with \(\ell = k - n\), it also follows by integration by parts that 
\[
\begin{split}
  \Deriv^{k - n + 1} 
  G_A (x) &= 
  \int_0^{+\infty} \bigl(\Deriv^{k - n + 1} H_A (t x)  
  + \dualprod{\Deriv^{k - n + 2} H_A (t x)}{x} t \bigr)\ln \frac{1}{t} \dif t \\
  & = \int_0^{+\infty} \Deriv^{k - n + 1} H_A (t x) \dif t\;,
\end{split}
\]
and thus \eqref{eq_shixae4newoh5ahV4} holds for \(\ell > k - n\) also when \(k \ge n\).

\smallbreak
For each \(x \in \Rset^n \setminus \set{0}\) and \(\lambda \in (0, +\infty)\), we have by \eqref{eq_oonee0oogai3ieM2veiy2oow} and by the change of variable \(t = s/\lambda\),
\begin{equation*}
  \begin{split}
    G_A (\lambda x) 
    &= 
    \int_0^{+\infty} 
    \frac{\dualprod{\Deriv^{k - n + 1} H_A (t \lambda x)}{(\lambda x)^{\otimes k - n  + 1}}}{(k - n)!} \ln \frac{1}{t} \dif t \\
    & \hspace{14em} + \frac{\dualprod{\Deriv^{k - n} H_A (0)}{(\lambda x)^{\otimes k - n}}}{(k - n)!}\sum_{j = 1}^{k - n} \frac{1}{j}\\
 &= \lambda^{k - n} \brk[\Big]{G_A (x) 
    +\ln \lambda \int_0^{+\infty} 
    \frac{\dualprod{\Deriv^{k - n + 1} H_A (s x)}{x^{\otimes k - n  + 1}}}{(k - n)!} \dif s}\\
    & = \lambda^{k - n} \brk[\big]{G_A (x) - \ln \lambda \,   P_A (x)}\;,
  \end{split}
\end{equation*}
where, in view of \eqref{eq_teek1coong8vi7Thoh1ohr4d}, the function \(P_A : \Rset^n \to \Lin (E + i E, V + i V)\) is defined for every \(x \in \Rset^n\) by 
\begin{equation}
\label{eq_zohfee0ha0uParaixohfaeb0}
\begin{split}
    P_A (x) \defeq & -
    \int_0^{+\infty} 
    \frac{
    \dualprod{\Deriv^{k - n + 1} H_A (t x)}{x^{\otimes k - n  + 1}}}{(k - n)!}\dif t\\
      = &-
    \int_{0}^{+\infty} 
    \int_{\Rset^{n}}
    \frac{\dualprod{\xi}{x}^{k - n + 1}e^{2 \pi i t  \dualprod{\xi}{x}}\psi (\xi) A (\xi)^\dagger }{(k - n)!(2 \pi i)^{n - 1}}  \dif \xi  \dif t\\
     =& -
    \int_{0}^{+\infty} 
    \int_{0}^{+\infty}
    \int_{\Sset^{n - 1}}
    \frac{\dualprod{\zeta}{x}^{k - n + 1}e^{2 \pi i t r \dualprod{\zeta}{x}}\psi (r \zeta) A (\zeta)^\dagger }{(k - n)!(2 \pi i)^{n - 1}}  \dif \zeta 
    \dif r \dif t\\
    =& -
    \int_{\Sset^{n - 1}} \int_{0}^{+\infty} 
    \int_{0}^{+\infty}
    \frac{\dualprod{\zeta}{x}^{k - n}e^{2 \pi i sr}\psi (r \zeta) A (\zeta)^\dagger  }{(k - n)!(2 \pi i)^{n - 1}}  
    \dif r \dif s \dif \zeta\;,
  \end{split}
\end{equation}
by the change of variables \(\xi = r \zeta\) and \(s = t \dualprod{\zeta}{x}\).
The innermost integrals can be computed by the Riemann-Lebesgue lemma as  
\begin{equation}
\label{eq_sui5eenaew3we3ush5Oojiaw}
\begin{split}
 \int_{0}^{+\infty} 
 \int_{0}^{+\infty}
    e^{2 \pi i sr}\psi (r \xi)
    \dif r \dif s 
    &= \lim_{S \to + \infty}
    \int_0^S \int_0^{+\infty} e^{2 \pi i sr}\psi (r \xi) \dif r \dif s \\
    &= \lim_{S \to \infty} \int_0^{+\infty} \frac{e^{2 \pi i S r} - 1}{2 \pi i r} \psi (r \xi) \dif r\\
    &= -\int_{0}^{+\infty} 
  \frac{\psi (r\xi)}{2 \pi i r}
    \dif r  = -\frac{1}{2 \pi i}\;,
\end{split}
\end{equation}
in view of \eqref{eq_ahpeithouG3yuothai5laete},
and therefore for every \(x \in \Rset^n\), by \eqref{eq_zohfee0ha0uParaixohfaeb0} and \eqref{eq_sui5eenaew3we3ush5Oojiaw} we have
\begin{equation}
\label{eq_VafaivohZiethah1OeTh5shu}
  P_{A} (x) 
  = 
    \int_{\Sset^{n - 1}}
    \frac{\dualprod{\xi}{x}^{k - n}A (\xi)^\dagger }{(k - n)!(2 \pi i)^{n}}
    \dif \xi\; .
\end{equation}
Since for every \(\xi \in \Rset^n\) and \(x \in \Rset^n\), \(\dualprod{\xi}{-x}^{k - n}A(-\xi)^\dagger = (-1)^n \dualprod{\xi}{x}^{k - n}A(\xi)^\dagger\), we have \(P_A = 0\) when \(n\) is odd.
When \(n\) is even,  we have for every \(x \in \Rset^n\), \(P_A (x) \in \Lin (E, V)\), and we have thus proved \eqref{it_oocee4ao4DaiK8sheWeechee} when \(k \ge n\).

If \(e \in \bigcap_{\xi \in \Rset^n \setminus \set{0}} A (\xi)[V]\), then for every \(\xi \in \Rset^n\) we have \(A (\xi)^{\dagger} [e]= A(\xi)^{-1} [e]\) and the assertion \eqref{it_supahJuacocei6efiey7aife} follows immediately from the identity \eqref{eq_VafaivohZiethah1OeTh5shu}.

We consider now a function \(f \in \mathcal{S} (\Rset^n, E)\) which satisfies for every \(\xi \in \Rset^n\) the condition that \(\mathcal{F} f (\xi) \in A (\xi)[V + iV]\).
If we let the function \(u: \Rset^n \to V\) be defined by \eqref{eq_au0uuzoghaem0aeN3aeGhah6}, then 
we have by the differentiation formula \eqref{eq_shixae4newoh5ahV4}, since \(k > k - n\),
\begin{equation}
  \label{eq_dahsheireo9oi9Aeshieshee}
  A (\Deriv) u (x) 
= \int_0^{+\infty} \int_{\Rset^n} A (\Deriv) H_A (t y)[f (y - x)] t^{n - 1} \dif t \dif y\;.
\end{equation}
We compute by the definition of \(H_A\) in \eqref{eq_teek1coong8vi7Thoh1ohr4d}
\begin{equation}
\label{eq_iNgeizahguu6Jaishahw6uaM}
\begin{split}
  \int_{\Rset^n} A (\Deriv) H_A (t y)&[f (x - y)] \dif y\\
  &= \int_{\Rset^n} \int_{\Rset^n} A (\xi)[A (\xi)^\dagger [e^{2 \pi i \dualprod{t \xi}{y}} f (x - y)]] \psi (\xi) \dif \xi \dif y\\
  &= \int_{\Rset^n} \int_{\Rset^n} A (\xi)[A (\xi)^\dagger [e^{2 \pi i \dualprod{t \xi}{y}} f (x - y)]] \psi (\xi) \dif y \dif \xi \\
&= \int_{\Rset^n} A (t \xi)[A (t\xi)^\dagger [e^{2 \pi i \dualprod{t \xi}{x}} \mathcal{F} f ( t \xi)]] \psi (\xi) \dif \xi\\
&= \int_{\Rset^n} e^{2 \pi i \dualprod{t\xi}{x}} \mathcal{F} f ( t \xi) \psi (\xi) \dif \xi\\
&= \frac{1}{t^n} \int_{\Rset^n} e^{2 \pi i \dualprod{\zeta}{x}} \mathcal{F} f (\zeta) \psi (\zeta/t) \dif \zeta
\;.
\end{split}
\end{equation}
and it follows thus from \eqref{eq_ahpeithouG3yuothai5laete}, \eqref{eq_dahsheireo9oi9Aeshieshee} and \eqref{eq_iNgeizahguu6Jaishahw6uaM} that  
\[
\begin{split}
A (\Deriv) u (x) 
& = \int_0^{+\infty} \int_{\Rset^n} e^{2 \pi i \dualprod{\zeta}{x}} \mathcal{F} f (\zeta) \frac{\psi (\zeta/t)}{t} \dif \zeta \dif t \\
& = \int_{\Rset^n} e^{2 \pi i \dualprod{\zeta}{x}} \mathcal{F} f (\zeta) \dif \zeta = f (x)\;.
\end{split}
\]
We have thus proved the assertion \eqref{it_wee9airoSei7iex2ude}.

In order to prove the assertion \eqref{it_Aix3Ahshe4ahchiekai4Ieph}, given 
\begin{equation*}e \in \smash{\bigcap_{\xi \in \Rset^n \setminus \set{0}} A (\xi)[V]}\;,
\end{equation*}
we fix a function \(\eta \in C^\infty_c (\Rset^n, \Rset)\) such that \(\smash{\int_{\Rset^n} \eta }= 1\) and for every \(\delta \in \intvo{0}{+\infty}\), we consider the function \(f_\delta \in C^\infty_c (\Rset^n, E)\) defined for each \(x \in \Rset^n\) by \(f_\delta (x) \defeq e \eta(x/\delta)/\delta^n\), so that \(\int_{\Rset^n} f_\delta = e\).
We define \(u_\delta : \Rset^n \to V\) for each \(x \in \Rset^n\) by 
\begin{equation*}
    u_\delta (x) \defeq \smash{\int_{\Rset^n}} G_A (x - y)[f_\delta (y)] \dif y\;.
\end{equation*}
For every \(\xi \in \Rset^n\), we have 
\begin{equation*}
\mathcal{F} f_\delta (\xi) = e \mathcal{F} \eta(\delta \xi) \in \Cset e \subseteq \smash{\bigcap_{\xi \in \Rset^n \setminus \set{0}} A (\xi)[V + i V]}\;.
\end{equation*}
By \eqref{it_wee9airoSei7iex2ude}, we have \(\smash{A(\Deriv) u_\delta = f_\delta}\) in \(\Rset^n\).
Since \(G_A \in \smash{L^1_{\mathrm{loc}} (\Rset^n, \Lin(E, V))}\), we have \(u_\delta \to G_A[e]\) in the sense of distributions as \(\delta \to 0\). Since \(f_\delta \to e \delta_0\) in the sense of distributions as \(\delta \to 0\), we conclude that \(A (\Deriv) G_A[e] = e \delta_0\) in \(\Rset^n\) in the sense of distributions and we have proved the assertion \eqref{it_Aix3Ahshe4ahchiekai4Ieph}.

We finally prove \eqref{it_Zigh3aiy3joo5maewiechua6}. For \(\ell > n - k\), if \(\varphi \in \smash{C^\infty_c (\Rset^n \setminus \set{0}, \Rset)}\) and if for every \(x \in \Rset^n\), 
\(\smash{\int_0^{+\infty}
\varphi (x/t) \, t^{\ell - k - 1} \dif t = 1}\), then by \eqref{eq_shixae4newoh5ahV4} we have
\[
\begin{split}
\int_{\Sset^{n - 1}}  \Deriv^\ell G_A (x) \dif x
&= \int_{\Sset^{n - 1}} \varphi (x r) \, \Deriv^\ell G_A (x)\, r^{k - \ell - 1} \dif x\\
& = 
  \int_{\Rset^n} \varphi (z)\,  \Deriv^\ell G_A (z) \dif z\\
  &= \int_{\Rset^n} \int_0^{+\infty} 
  \varphi (z)\,
  \Deriv^\ell H_A (t z) \,
  t^{n + \ell - k - 1} \dif t \dif z \\
  &= \int_{\Rset^n} \int_0^{+\infty} 
  \varphi (z/t)\,
  \Deriv^\ell H_A (z) \,
  t^{\ell - k - 1} \dif z\dif t \\
  &= \int_{\Rset^n}  
  \Deriv^\ell H_A (z) \dif z = 0\;;
\end{split}
\]
this proves \eqref{it_Zigh3aiy3joo5maewiechua6}.
\end{proof}

\subsection{Compatibility Conditions}

The condition \(\mathcal{F} f (\xi) \in A (\xi)[V + i V]\) in  \cref{proposition_fundamental_solution} \eqref{it_wee9airoSei7iex2ude} can be reformulated as being in the kernel of some differential operator by the next proposition.

\begin{proposition}[Compatibility conditions]
  \label{proposition_compatibility_conditions}
Let \(n \in \Nset \setminus \set{0}\), let \(V\) and \(E\) be finite-dimensional vector spaces, and let \(A (\Deriv)\) be a homogeneous constant coefficient differential operator
of order \(k \in \Nset \setminus \set{0}\) from 
\(V\) to \(E\) on \(\Rset^n\).
If \(A (\Deriv)\) is injectively elliptic, then there exists a homogeneous constant coefficient differential operator \(L (\Deriv)\)
from 
\(E\) to \(F\) on \(\Rset^n\) such that for every \(\xi \in \Rset^n \setminus \set{0}\),
\begin{equation}
\label{eq_phicee9aad0ugh2thaeTueph}
  A (\xi)[V] = \ker L (\xi)\;.
\end{equation}
\end{proposition}

The operator \(L (\Deriv)\) arising in \cref{proposition_compatibility_conditions} describes the \emph{compatibility conditions} of the operator \(A (\Deriv)\). It describes the obstructions to the solution of systems of the form \(A (\Deriv)u = f\) and the structure of the range of the differential operator \(A (\Deriv)\).

Such conditions are well-known for the gradient operator, whose image are curl-free vector fields: one has indeed for every \(i, j \in \set{1, \dotsc, n}\),
\[
  \partial_j (\partial_i u) = \partial_i (\partial_j u)\;.
\]
Similarly for the \emph{Hodge complex,} where \(A (\Deriv) = (d, d^*)\), one has \((d, d^*) A (\Deriv) = (d^2, d^*{}^2) = 0\).
A more subtle but still classical setting are the compatibility conditions for the symmetric derivative or deformation operator known in linear elasticity as the \emph{Saint-Venant compatibility conditions}:\footnote{For a discussion of Saint-Venant compatibility conditions in physical context see for example \cite{Timoshenko_Goodier_1951}*{Ch.\ 9} and in the mathematical theory for smooth functions see \cite{Ciarlet_2013}*{\S 6.18}} for every \(i, j, k, \ell \in \set{1, \dotsc, n}\), we have 
\begin{equation}
\label{eq_Xah5quaejieRahche9roova8}
 \partial_{k\ell} (\partial_i u^j +\partial_j u^i)
 +\partial_{ij} (\partial_k u^\ell +\partial_\ell u^k)
 = \partial_{kj} (\partial_i u^\ell +\partial_\ell u^i)
 +\partial_{i\ell} (\partial_k u^j +\partial_j u^k)\;.
\end{equation}

\begin{proof}[Proof of \cref{proposition_compatibility_conditions}]
We define the operator \(L (\Deriv)\) by setting for every \(\xi \in \Rset^n \setminus \set{0}\),
\begin{equation}
\label{eq_xooch1Ocui9jechohnaeFo2I}
 L (\xi) \defeq \det \bigl(A (\xi)^* \compose A (\xi)\bigr)\,
 \bigl(\operatorname{id}_E\, -\, A (\xi) \compose (A (\xi)^* \compose A (\xi))^{-1} \compose A (\xi)^* \bigr)
\end{equation}
and \(L (0) \defeq 0\).
We observe that for every \(\xi \in \Rset^n \setminus \set{0}\),
\[
\ker L (\xi) = \set[\big]{ e \in E \st A (\xi) \compose (A (\xi)^* \compose A (\xi))^{-1} \compose A (\xi)^* [e] = e}
= A (\xi)[V]\;,
\]
so that \eqref{eq_phicee9aad0ugh2thaeTueph} holds. 
Moreover, since \(A (\xi)^* \compose A (\xi)\) is a polynomial in \(\xi\), then 
\begin{equation*}
\det (A (\xi)^* \compose A (\xi)) (A (\xi)^* \compose A (\xi))^{-1}
\end{equation*}
 is also a polynomial in \(\xi\) and thus \(L (\xi)\) is a polynomial in \(\xi\).
\end{proof}

\begin{remark}
The proof of \cref{proposition_compatibility_conditions} yields an operator \(L (\Deriv)\) of order \(2 k \dim V\), which is much more than what is needed in typical examples.
\end{remark}

\subsection{Singular Integrals and \texorpdfstring{\(L^p\)}{Lᵖ} Estimates}

We are now is position to obtain \(L^p\) estimates for injectively elliptic operators.

\begin{theorem}[\(L^p\) estimate for injectively elliptic operators]
\label{theorem_right_elliptic_estimate}
Let \(n \in \Nset \setminus \set{0}\), let \(V\) and \(E\) be finite-dimensional vector spaces, let \(A (\Deriv)\) be a homogeneous constant coefficient differential operator
of order \(k \in \Nset \setminus \set{0}\) from 
\(V\) to \(E\) on \(\Rset^n\), and let \(p \in (1, +\infty)\).
If \(A (\Deriv)\) is injectively elliptic, then there exists a constant  \(C \in \intvo{0}{+\infty}\) such that for every \(u \in C^\infty_c (\Rset^n, V)\), 
  \begin{equation}
  \label{eq_ae1IThu7lai7maequoob2egh}
  \int_{\Rset^n} \abs{\Deriv^k u}^p
  \le 
  C
  \int_{\Rset^n} \abs{A (\Deriv)[u]}^p\;.
  \end{equation}
\end{theorem}

Our main tool will be the Calderón-Zygmund singular integral theorem.

\begin{theorem}[Singular integral theorem%
\footnote{%
The \(L^p\) theory of singular integrals of \cref{theorem_singular_integral} originates in the work of Alberto \familyname{Calderón} and Antoni \familyname{Zygmund} \cite{Calderon_Zygmund_1952} (see also \cite{Stein_1970}*{Ch.\ II Th.\ 3}).%
}%
]
  \label{theorem_singular_integral}
  Let \(n \in \Nset \setminus \set{0}\) and let \(V\) and \(E\) be finite-dimensional vector spaces.
  If \(K \in C^1 (\Rset^n \setminus\set{0}, \Lin (E, F))\) is homogeneous of degree \(-n\) and if 
  \begin{equation*}
    \int_{\Sset^{n - 1}} K = 0\;,
  \end{equation*}
  then for every \(p \in \intvo{1}{+\infty}\), there exists a constant \(C \in \intvo{0}{+\infty}\) such that if 
we define for each \(\varepsilon \in \intvo{0}{+\infty}\) the operator \(\mathcal{K}_\varepsilon\) for every \(f \in L^p (\Rset^n, E)\) by 
  \begin{equation*}
    (\mathcal{K}_{\varepsilon} f) (x) \defeq \int_{\Rset^n \setminus B_\varepsilon (0)} K (h) [f (x - h)]  \dif h\;,
  \end{equation*}
  then \(\mathcal{K}_\varepsilon\) is well-defined and 
  we have for every \(f \in L^p (\Rset^n, E)\),
  \begin{equation*}
    \int_{\Rset^n} \abs{\mathcal{K}_{\varepsilon} f}^p
    \le C
    \int_{\Rset^n} \abs{f}^p\;.
  \end{equation*}
  Moreover, \(\mathcal{K}_\varepsilon f\) converges in \(L^p (\Rset^n, F)\) as \(\varepsilon \to 0\).
\end{theorem}

In order to apply \cref{theorem_singular_integral} and prove \cref{theorem_right_elliptic_estimate}, we will rely on the following singular representation formula for \(\Deriv^k u\) that follows from 
the construction of representation kernels (\cref{proposition_fundamental_solution}).

\begin{lemma}[Singular integral representation formula]
\label{lemma_representation_Dku}
Under the assumptions of \cref{proposition_fundamental_solution} and with \(G_A : \Rset^n \setminus \set{0} \to \Lin (E, V)\) given by the same proposition, for every \(u \in C^\infty_c(\Rset^n, V)\) and every \(x \in \Rset^n\), one has for every \(r \in \intvo{0}{\infty}\)
\begin{equation}
\label{eq_vohquaeci7Aitoth8ohteich}
\begin{split}
 \Deriv^k u (x) &= 
 \int_{\Sset^{n - 1}} \Deriv^{k - 1} G_A \brk{z}[A (\Deriv) u (x)] \otimes z \dif z
 + \int_{\Rset^n \setminus B_r\brk{0}} \Deriv^{k} G_A (x - y)[ A (\Deriv) u (y)] \dif y\\
 &\qquad +\int_{B_r\brk{0}} \Deriv^{k} G_A (x - y)[ A (\Deriv) u (y) - A (\Deriv) u (x)] \dif y\\
 &= \int_{\Sset^{n - 1}} \Deriv^{k - 1} G_A \brk{z}[A (\Deriv) u (x)] \otimes z \dif z\\
 &\qquad + \lim_{\varepsilon \to 0} \int_{\Rset^n \setminus B_\varepsilon(0)} \Deriv^{k} G_A (h)[A (\Deriv) u (x - h)] \dif h\;.
 \raisetag{4em}
 \end{split}
\end{equation}
\end{lemma}
\begin{proof}
The first \(k -1\) derivatives of \(G_A\) are locally integrable in view of \cref{proposition_fundamental_solution} \eqref{it_oocee4ao4DaiK8sheWeechee} and thus differentiating \(k-1\) times \eqref{eq_jal9Niez6Vu2ohved5OorooG}, we get for every \(x \in \Rset^n\)
\begin{equation}
\label{eq_Xoo6theoxuch2iesoos2wini}
  \Deriv^{k - 1} u (x) = \int_{\Rset^n} \Deriv^{k - 1} G_A (x - y)[A (\Deriv) u (y)] \dif y\;.
\end{equation}
We fix a function \(\eta \in C^\infty_c (\Rset^n, \Rset)\) such that \(\eta = 1\) on \(B_{1/2}(0)\) and \(\eta = 0\) on \(\Rset^n \setminus B_1(0)\).
 If \(R > 0\) is large enough so that \(\supp u \subseteq B_R(0)\), for every \(\varepsilon \in \intvo{0}{R/2}\), we rewrite \eqref{eq_Xoo6theoxuch2iesoos2wini} in view of \cref{proposition_fundamental_solution} \eqref{it_Zigh3aiy3joo5maewiechua6} as 
\begin{equation}
  \label{eq_ka6geaheeTumahqueiKae2ch}
\begin{split}
 \Deriv^{k - 1} u (x) &=
 \int_{\Rset^n} \eta\brk[\big]{\tfrac{\abs{x - y}}{R}} \brk[\big]{1- \eta\brk[\big]{\tfrac{\abs{x - y}}{\varepsilon}}}\Deriv^{k - 1} G_A (x - y)[ A (\Deriv) u (y) - A (\Deriv) u (x)] \dif y\\
 &\quad +  \int_{\Rset^n} \eta\brk[\big]{\tfrac{\abs{h}}{\varepsilon}} \Deriv^{k - 1} G_A (h)[ A (\Deriv) u (x - h) ] \dif h\;.
 \end{split}
\end{equation}
Differentiating both integrals in the right-hand side of \eqref{eq_ka6geaheeTumahqueiKae2ch}, we get 
\begin{equation}
  \label{eq_Oi0gaSa5iemeike1oi5ior8i}
\begin{split}
 &\Deriv^{k} u (x) \\
 &=\int_{\Rset^n}  \eta\brk[\big]{\tfrac{\abs{x - y}}{R}} \brk[\big]{1- \eta\brk[\big]{\tfrac{\abs{x - y}}{\varepsilon}}} \Deriv^{k} G_A (x - y)[A (\Deriv) u (y) - A (\Deriv) u (x)] \dif y\\
 &\quad + \int_{\Rset^n} (\tfrac{1}{R}  \eta'\brk[\big]{\tfrac{\abs{x - y}}{R}) - \tfrac{1}{\varepsilon} \eta'\brk[\big]{\tfrac{\abs{x - y}}{\varepsilon}}} \Deriv^{k - 1} G_A (x - y)[A (\Deriv) u (y) - A (\Deriv) u (x)]\otimes\tfrac{x-y}{\abs{x - y}} \dif y\\
 &\quad +  \int_{\Rset^n} \eta\brk[\big]{\tfrac{\abs{h}}{\varepsilon}} \Deriv^{k - 1} G_A (h)[\Deriv A (\Deriv) u (x - h) ] \dif h\;.
 \raisetag{2em}
 \end{split}
\end{equation}
Letting \(\varepsilon \to 0\) in \eqref{eq_Oi0gaSa5iemeike1oi5ior8i}, the conclusion \eqref{eq_vohquaeci7Aitoth8ohteich} follows from \cref{proposition_fundamental_solution} \eqref{it_Zigh3aiy3joo5maewiechua6}, Lebesgue's dominated convergence theorem and the computation 
\[
\begin{split}
  &\int_{\Rset^n}  \eta'\brk[\big]{\tfrac{\abs{x - y}}{R}} 
 \Deriv^{k - 1} G_A (x - y)[A (\Deriv) u (x)]\otimes\tfrac{x-y}{R \abs{x - y}} \dif y\\
 &\qquad = \int_0^\infty \int_{\Sset^{n - 1}} \eta'\brk[\big]{\tfrac{r}{R}}  \Deriv^{k - 1} G_A \brk{z}[A (\Deriv) u (x)] \otimes \tfrac{z}{R} \dif z \dif r\\
 &\qquad = - \int_{\Sset^{n - 1}} \Deriv^{k - 1} G_A \brk{z}[A (\Deriv) u (x)] \otimes z \dif z.\qedhere
\end{split}
\]
\end{proof}

\begin{proof}[Proof of \cref{theorem_right_elliptic_estimate}]
This follows from \cref{theorem_singular_integral} and \cref{lemma_representation_Dku}.
\end{proof}

\subsection{Failure of the Endpoint Estimates}

The restriction that \(1 < p < \infty\) in the estimate \eqref{eq_ae1IThu7lai7maequoob2egh} of  \cref{theorem_right_elliptic_estimate} is essential.
Indeed, when \(p =\infty\), Karel \familyname{de Leeuw} and Hazleton \familyname{Mirkil} have proved that there is no nontrivial estimate \citelist{\cite{deLeeuw_Mirkil_1962}\cite{deLeeuw_Mirkil_1964}}.

\begin{theorem}[\(L^\infty\) nonestimate]
  \label{theorem_deLeeuw_Mirkil}
  Let \(n \in \Nset \setminus \set{0}\), let \(V\), \(E\) and \(F\) be finite-dimensional vector spaces,
  let \(A (\Deriv)\) be a homogeneous constant coefficient differential operator
  of order \(k \in \Nset \setminus \set{0}\) from 
  \(V\) to \(E\) on \(\Rset^n\),
  and let \(B (\Deriv)\) be a be a homogeneous constant coefficient differential operator
  of order \(k\) from 
  \(V\) to \(F\) on \(\Rset^n\).
  If there exists a constant \(C \in \intvo{0}{+\infty}\) such that for every \(u \in C^\infty_c (\Rset^n, V)\) one has
  \begin{equation}
  \label{eq_eiB1ooDaiNeel9Naez5ooth5}
    \sup_{\Rset^n} \, \abs{B (\Deriv) u} \le C \sup_{\Rset^n} \, \abs{A (\Deriv) u},
  \end{equation}
  then there exists \(L \in \Lin (E, F)\) such that 
  \(B (\Deriv) = L \compose A (\Deriv)\).
\end{theorem}

\begin{proof}
Let assume for contradiction that there does not exist  \(L \in \Lin (V, E)\) such that 
\(B (\Deriv) = L \compose A (\Deriv)\). Then by duality, there exists a homogeneous polynomial \(P : \Rset^n \to V\) of degree \(k\) such that 
\(A (\Deriv)P = 0\) and \(B (\Deriv) P = e \in E \setminus \set{0}\) on \(\Rset^n\).
  
We choose a function \(\eta \in C^\infty_c (\Rset^n, \Rset)\) such that \(\eta = 1\) on \(B_1 (0)\). 
We define for every \(\varepsilon \in (0, +\infty)\) the function \(u_\varepsilon : \Rset^n \to V\) for each \(x \in \Rset^n\) by 
\[
u_\varepsilon (x) \defeq \eta (x) P (x) \ln \frac{1}{\abs{x}^2 + \varepsilon^2}\;.
\]
In view of \cref{lemma_AD_Leibnitz}, we have for every \(x \in \Rset^n\) and \(\varepsilon \in \intvo{0}{1}\)
\[
\begin{split}
\abs{A (\Deriv) u_\varepsilon (x)} 
&\le 
\Cl{cst_UciPa9reeNgegho4ahsh1ooz} \brk[\Big]{\abs[\Big]{A(\Deriv)(\eta P)(x) \ln \frac{1}{\abs{x}^2 + \varepsilon^2}} + 
\sum_{j = 1}^k \frac{\abs {\Deriv^{k - j} (\eta P) (x)}}{(\abs{x}^2 +  \varepsilon^2)^{j/2}}}\\
&\le \C \brk[\Big]{1 + \sum_{j = 1}^k \frac{\abs {\Deriv^{k - j} (\eta P) (x)}}{\abs{x}^j}}
\le \C\;,
\end{split}
\]
whereas 
\(
\abs{B (\Deriv) u_\varepsilon (0)}
= \abs{e}
\ln  1/\varepsilon\;
\);
the contradiction follows from \eqref{eq_eiB1ooDaiNeel9Naez5ooth5} as \(\varepsilon \to 0\).
\end{proof}

For \(p = 1\), Donald S.\ \familyname{Ornstein} has proved there is no nontrivial estimate \cite{Ornstein_1962}.\footnote{Although Ornstein's original paper does not cover explicitly the vector case in \cref{theorem_Ornstein} its proofs seems to do, and more recent approaches explicitly do \citelist{\cite{Kirchheim_Kristensen_2016}\cite{Kirchheim_Kristensen_2011}} (for nonestimates through convex integration see also \cite{Conti_Faraco_Maggi_2005}).}

\begin{theorem}[\(L^1\) nonestimate]
  \label{theorem_Ornstein}
Let \(n \in \Nset \setminus \set{0}\), let \(V\), \(E\) and \(F\) be finite-dimensional vector spaces,
let \(A (\Deriv)\) be a homogeneous constant coefficient differential operator of order \(k \in \Nset \setminus \set{0}\) from 
\(V\) to \(E\) on \(\Rset^n\),
and let \(B (\Deriv)\) be a be a homogeneous constant coefficient differential operator
of order \(k\) from 
\(V\) to \(F\) on \(\Rset^n\).
If there exists a constant \(C \in \intvo{0}{+\infty}\) such that for every \(u \in C^\infty_c (\Rset^n, V)\) one has
\begin{equation*}
\int_{\Rset^n} \abs{B (\Deriv) u} \le C \int_{\Rset^n} \abs{A (\Deriv)u}\;,
\end{equation*}
then there exists \(L \in \Lin (E, F)\) such that 
\(B (\Deriv) = L \compose A (\Deriv)\).
\end{theorem}

The proof of \cref{theorem_Ornstein} is also based on the existence of a homogeneous polynomial \(P : \Rset^n \to V\) of degree \(k\) such that 
\(A (\Deriv)P = 0\) and \(B (\Deriv) P = e \in E \setminus \set{0}\) on \(\Rset^n\). However, the construction of a family of compactly supported functions for which the inequality fails is much more delicate and relies on techniques of convex integration.

\section{Cancelling Operators}
\label{section_cancelling}

Since in view of Ornstein’s \(L^1\) nonestimate a vector differential operator cannot be we cannot controlled nontrivially in \(L^1\) by another vector differential operator, one can consider whether lower-order derivatives can be controlled through Sobolev-type embeddings.

\subsection{Sobolev Embeddings and Injective Ellipticity}

The classical Sobolev embedding theorem on the Euclidean space states that the integrability of a derivative implies some higher-integrability of lower-order derivatives.

\begin{theorem}[Sobolev embedding theorem%
\footnote{%
The Sobolev embedding theorem (\cref{theorem_Sobolev}) is due to Sergei Lvovich \familyname{Sobolev} when \(p \in (1, m/(k-\ell))\) \cite{Sobolev_1938} and to 
  Emilio \familyname{Gagliardo} \cite{Gagliardo_1958} and simultaneuously to Louis \familyname{Nirenberg} \cite{Nirenberg_1959} in the endpoint case \(p = 1\).%
  }%
  ]
\label{theorem_Sobolev}
Let \(n \in \Nset \setminus \set{0}\), let \(V\) be a finite-dimensional vector space, let \(k \in \Nset \setminus\set{0}\) and \(\ell \in \Nset\) satisfy \(0 < k - \ell < n\), and let \(p \in [1, \frac{n}{k - \ell})\).
There exists a constant \(C \in \intvo{0}{+\infty}\) such that for every \(u \in C^\infty_c (\Rset^n, V)\),
\[
  \brk[\Big]{\int_{\Rset^n} \abs{\Deriv^{\ell} u}^\frac{np}{n - (k - \ell)p}}^{1 - \frac{(k - \ell)p}{n}}
  \le 
  C 
  \int_{\Rset^n} \abs{\Deriv^{k} u}^p\;.
\]
\end{theorem}

When \(1 < p < n/(k - \ell)\), as a consequence of the Sobolev embedding theorem (\cref{theorem_Sobolev}) and of \cref{theorem_right_elliptic_estimate}, if \(A (\Deriv)\) is a homogeneous constant coefficient differential operator of order \(k \in \Nset \setminus \set{0}\) from 
\(V\) to \(E\) on \(\Rset^n\) which is injectively elliptic, then there exists a constant \(C \in \intvo{0}{+\infty}\) such that for every \(u \in C^\infty_c (\Rset^n, V)\), 
we have the inequality 
  \begin{equation}
  \label{eq_bah2acooqueBei4Ahph}
   \brk[\Big]{\int_{\Rset^n} \abs{\Deriv^{\ell} u}^\frac{np}{n - (k - \ell)p}}^{1 - \frac{(k - \ell)p}{n}}
  \le 
  C
  \int_{\Rset^n} \abs{A (\Deriv)[u]}^p\;.
  \end{equation}
The injective ellipticity of \(A(\Deriv)\) turns out to be a necessary condition for \eqref{eq_bah2acooqueBei4Ahph} to hold when \(\ell = k - 1\) and \(1 \le p < n\).
  
\begin{theorem}[%
Necessity of the injective ellipticity for Sobolev embeddings%
\footnote{%
The necessity of injective ellipticity for Sobolev estimates (\cref{theorem_Sobolev_elliptic_necessary}) follows from well-known techniques; an explicit proof appears for example in \cite{VanSchaftingen_2013}*{Prop.\ 5.1}.
}%
]
  \label{theorem_Sobolev_elliptic_necessary}
Let \(n \in \Nset \setminus \set{0}\), let \(V\) and \(E\) be finite-dimensional vector spaces, let \(A (\Deriv)\) be a homogeneous constant coefficient differential operator
of order \(k \in \Nset \setminus \set{0}\) from 
\(V\) to \(E\) on \(\Rset^n\),
let \(B (\Deriv)\) be a homogeneous constant coefficient differential operator
of order \(k - 1\) from 
\(V\) to \(E\) on \(\Rset^n\),
and let \(p \in \intvr{1}{n}\).
If there exists a constant \(C \in \intvo{0}{+\infty}\) such that for every \(u \in C^\infty_c (\Rset^n, V)\), 
\begin{equation}
  \label{eq_Eili1woe5gaeCh4uch2pohvi}
   \brk[\Big]{\int_{\Rset^n} \abs{B (\Deriv) u}^\frac{np}{n - p}}^{1 - \frac{p}{n}}
  \le 
  C
  \int_{\Rset^n} \abs{A (\Deriv)[u]}^p\;,
\end{equation}
  then for every \(\xi \in \Rset^n \setminus \set{0}\), one has \(\ker B (\xi) \subseteq \ker A (\xi)\).
\end{theorem}
\begin{proof}
Let \(\xi \in \Rset^n \setminus \set{0}\) and \(v \in V\).
We take functions \(\theta \in C^\infty_c (\Rset, \Rset)\setminus \set{0}\) and 
\(\eta \in C^\infty_c (\Rset^n, \Rset)\) such that \(\eta (0) = 1\).
For each \(\lambda \in \Rset\), we define the function \(u_\lambda : \Rset^n \to V\) for each \(x \in \Rset^n\) by
\[
  u_\lambda (x) \defeq \theta (\lambda \dualprod{\xi}{x}) \eta (x) v\;.
\]
By \cref{lemma_AD_Leibnitz} and by Minkowski's inequality, we have 
\begin{multline*}
 \abs[\Big]{\brk[\Big]{\int_{\Rset^n} \abs{A (\Deriv) u_\lambda}^p}^\frac{1}{p}
 - \brk[\Big]{\lambda^{kp} \,\abs{A (\xi)[v]} \int_{\Rset^n} \abs{\theta^{(k)} (\lambda \dualprod{\xi}{x})}^p \abs{\eta(x)}^p \dif x}^\frac{1}{p}}\\
\smash{\le \C \sum_{j = 1}^k \, \brk[\Big]{\lambda^{(k - j)p} \!\int_{\Rset^n} \abs{\theta^{(k - j)}(\lambda \dualprod{\xi}{x})}^p\abs{\Deriv^j \eta (x)}^p \dif x}^\frac{1}{p}\;}.
\end{multline*}
and thus if \(\lambda \ge 1\), 
\begin{equation}
  \label{eq_ohhah7Fua5eechaeTaeNieng}
 \brk[\Big]{\int_{\Rset^n} \abs{A (\Deriv) u_\lambda}^p}^\frac{1}{p}
 \le \lambda^{k - \frac{1}{p}} \abs[\big]{A (\xi)[v]} \brk[\Big]{\, \int_{\Rset} \abs[\big]{\theta^{(k)}}^p}^\frac{1}{p}
 \brk[\Big]{\int_{\xi^\perp} \abs[\big]{\eta}^p}^\frac{1}{p}
 + \C \lambda^{k - 1 - \frac{1}{p}}\;.
\end{equation}
Hence if \(v \in \ker A (\xi)\), the inequality \eqref{eq_ohhah7Fua5eechaeTaeNieng} implies
\begin{equation}
  \label{eq_pai3quiulahroo0XoCh9yiep}
  \limsup_{\lambda \to \infty} \frac{1}{\lambda^{k - \frac{1}{p}}}\brk[\Big]{\int_{\Rset^n} \abs{A (\Deriv) u_\lambda}^p}^\frac{1}{p} < \infty\;.
\end{equation}
By a similar reasoning on the operator \(B (\Deriv)\), we have 
\begin{equation}
  \label{eq_eatheiz0waipaiZ8chiechoh}
\lim_{\lambda \to \infty}
\frac{1}{\lambda^{k - 1 - \frac{1}{p} + \frac{1}{n}} }
 \brk[\Big]{\int_{\Rset^n} \abs{B (\Deriv) u_\lambda}^\frac{np}{n - p}}^{\frac{1}{p} - \frac{1}{n}}
 = \abs[\big]{B (\xi)[v]} 
 \, \int_{\Rset} \abs[\big]{\theta^{(k - 1)}}^p \; \int_{\xi^\perp} \abs[\big]{\eta}^p\;,
\end{equation}
and thus by \eqref{eq_Eili1woe5gaeCh4uch2pohvi}, \eqref{eq_pai3quiulahroo0XoCh9yiep} and \eqref{eq_eatheiz0waipaiZ8chiechoh}, we have \(\abs{B (\xi)[v]} = 0\), that is, \(v \in \ker B (\xi)\).
\end{proof}

\begin{remark}
  Ellipticity is not necessary for estimates of lower-order derivatives: Indeed, one has for example for every \(u \in C^\infty_c (\Rset^{2n},\Rset)\),
\[
\brk[\Big]{\int_{\Rset^{2n}} \abs{u}^2}^\frac{1}{2}
\le C \int_{\Rset^{2n}} \abs{\partial_1 \dotsm \partial_n u} + \abs{\partial_{n + 1} \dotsm \partial_{2n} u}\;.
\]
\end{remark}

\subsection{Sobolev Embeddings and Cancelling Operators}

In the case \(p = 1\), injective ellipticity is still not sufficient to have a Sobolev estimate.

\begin{myexample}
  We define for each \(\lambda \in \intvo{0}{+\infty}\) the function \(u_\lambda : \Rset^n \to \Rset^n\) by setting for each \(x \in \Rset^n\)
  \[
    u_\lambda (x) 
    \defeq
   \smash{ \frac{x}{\abs{x}^n} }\eta (\abs{x}) (1 - \eta( \abs{x}/\lambda))\;,
  \]
  with \(\eta \in C^\infty (\intvo{0}{+\infty}, \Rset)\) such that \(\eta = 1\) on \(\intvo{0}{1}\) and \(\eta = 0\) on \(\intvo{2}{+\infty}\).
  We have then \(\operatorname{curl} u_\lambda = 0\) on \(\Rset^n\),
  \[
    \smash{\int_{\Rset^n} \abs{\operatorname{div} u_\lambda} \le \C}
  \]
  and 
  \[
    \lim_{\lambda \to +\infty} \smash{ \int_{B_1(0)} \abs{u_\lambda}^\frac{n}{n - 1} 
    \ge 
    \int_{B_1 (0)} \frac{\mathrm{d}x}{\abs{x}^n} = +\infty\;}.
  \]
\end{myexample}

When \(p = 1\), by \cref{theorem_Ornstein} the Calderón-Zygmund theory of singular integrals (\cref{theorem_right_elliptic_estimate}) fails, but one can still ask whether there are some endpoint Sobolev estimates for vector differential operators.

We will show that such estimates under an additional cancellation condition defined as follows.

\begin{definition}[Cancelling operator%
\footnote{%
The definition of cancelling operators (\cref{definition_cancelling}) is due to the author \cite{VanSchaftingen_2013}.
}%
]
  \label{definition_cancelling}
  Let \(n \in \Nset\setminus \set{0}\) and let \(V\) and \(E\) be finite-dimensional vector spaces. 
  A homogeneous constant coefficient differential operator \(A (\Deriv)\) 
of order \(k \in \Nset \setminus \set{0}\) from 
\(V\) to \(E\) on \(\Rset^n\) is \emph{cancelling} whenever
\[
 \bigcap_{\xi \in \Rset^n \setminus \set{0}} A (\xi)[V] = \set{0}\;.
\]
\end{definition}

When \(n = 1\), an operator \(A(\Deriv)\) is cancelling if and only if \(A (\Deriv) = 0\); hence this notion only makes sense in higher dimensions \(n \ge 2\).

Formally, the cancellation condition means that \(A (D) u = e \delta_0\) does not have a solution unless \(e = 0\), as this would imply that \((2 \pi i)^k A (\xi) \mathcal{F} u (\xi) = \mathcal{F} (A (D) u)(\xi) =  \mathcal{F} (e \delta_0) = e\); this will be proved in \cref{propositionEquivalentCancelling}.

\medbreak 
The cancellation property characterizes when endpoint Sobolev inequalities hold.
\begin{theorem}[%
Endpoint Sobolev inequality for cancelling operators%
\footnote{%
The necessity and sufficiency of cancellation for vector Sobolev estimates is due to the author \cite{VanSchaftingen_2013}*{Prop.\ 4.6 and  5.5}.}%
]
\label{theorem_cancelling_necessary_Sobolev}
Let \(n \in \Nset\setminus\set{0}\), let \(V\) and \(E\) be finite-dimensional vector spaces, 
and let \(A (\Deriv)\) be a homogeneous constant coefficient differential operator
of order \(k \in \Nset \setminus \set{0}\) from 
\(V\) to \(E\) on \(\Rset^n\).
Assume that \(A (\Deriv)\) is injectively elliptic and that \(\ell \in \Nset \setminus \set{0}\) satisfies \(0 < k - \ell < n\).
There exists a constant \(C \in \intvo{0}{+\infty}\) such that for every \(u \in C^\infty_c (\Rset^n, V)\) 
\begin{equation}
\label{eq_eeghaib2eiviu0cib4Pei8ut}
  \brk[\Big]{\int_{\Rset^n} \abs{\Deriv^\ell u}^\frac{n}{n - (k - \ell)}}^{1 - \frac{k - \ell}{n}}
\le 
C
\int_{\Rset^n} \abs{A (\Deriv)[u]}\;,
\end{equation}
if and only if the operator \(A (\Deriv)\) is cancelling.
\end{theorem}

It should be noted that \cref{theorem_cancelling_necessary_Sobolev} imposes strong boundary conditions on \(u\) --- requiring in fact compact support --- for the estimate \eqref{eq_eeghaib2eiviu0cib4Pei8ut} to hold. Working with weaker boundary conditions would introduce additionnal restrictions on the admissible class of operators \citelist{\cite{Brezis_VanSchaftingen_2007}\cite{Gmeineder_Raita_2019}\cite{Gmeineder_Raita_VanSchaftingen_2021}
}.%

We close this section by the proof of the necessity of the cancellation condition in \cref{theorem_cancelling_necessary_Sobolev}.
Roughly speaking one would like to take \(u (x) \defeq G_A(x)[e]\) for some \(e \in \bigcap_{\xi \in \Rset^n \setminus \set{0}} A(\xi)[V]\) so that \(A(\Deriv) u = \delta e_0\) (see \cref{proposition_fundamental_solution} \eqref{it_Aix3Ahshe4ahchiekai4Ieph}), but unfortunately this function \(u\) is neither smooth nor compactly supported. 
The next lemma provides a careful approximation of such a function.

\begin{lemma}[Regularization of the representation kernel%
\footnote{%
The statement of \cref{lemma_regularization_G_A} is due to the author \cite{VanSchaftingen_2013}*{Prop.\ 5.5}. We give here a proof based on the properties of the representation kernel instead of a direct construction in Fourier space as in \cite{VanSchaftingen_2013}.%
}%
]
\label{lemma_regularization_G_A}
Let \(n \in \Nset\setminus\set{0}\), let \(V\) and \(E\) be finite-dimensional vector spaces,  and let \(A (\Deriv)\) be a homogeneous constant coefficient differential operator
of order \(k \in \Nset \setminus \set{0}\) from 
\(V\) to \(E\) on \(\Rset^n\).
If \(A (\Deriv)\) is injectively elliptic and if \(e \in \bigcap_{\xi \in \Rset^n \setminus\set{0}} A (\xi)[V]\), then there exists a sequence \((u_j)_{j \in \Nset}\) in \(C^\infty_c (\Rset^n, V)\) such that 
\begin{enumerate}[(i)]
\item for every \(j \in \Nset\), \(\supp u_j \subset B_2(0)\),
 \item \(\lim_{j \to \infty} \int_{\Rset^n} \abs{A (\Deriv) u_j} \le C \abs{e}\),
 \item \( u_j \to  G_A[e]\) in \(C^\infty_{\mathrm{loc}} (B_1(0) \setminus \set{0}, E)\) as \(j \to \infty\).
\end{enumerate}
\end{lemma}

\begin{proof}[Proof of \cref{lemma_regularization_G_A}]
 We take a function \(\varrho \in C^\infty_c (\Rset^n, \Rset)\) such that \(\int_{\Rset^n} \varrho = 1\).
 We define first for every \(\lambda \in \intvo{0}{+\infty}\) the function \(v_\lambda : \Rset^n \to V\) by setting for every \(x \in \Rset^n\)
 \[
  v_\lambda (x) \defeq
  \lambda^n \int_{\Rset^n} G_A (x - y)[\varrho (\lambda y) e] \dif y\;.
 \]
 By our assumption, for every \(\xi \in \Rset^n\), we have  
 \begin{equation*} 
 \mathcal{F} (\varrho e) (\xi) = (\mathcal{F}\varrho (\xi)) e \in \Cset e \subseteq A (\xi)[V + iV]
 \end{equation*}
 and thus by \cref{proposition_fundamental_solution} \eqref{it_wee9airoSei7iex2ude}, for every \(x \in \Rset^n\), 
 \(
  A (\Deriv) v_\lambda (x) = \lambda^n \varrho (\lambda x)
 \).
 Moreover, as \(\lambda \to +\infty\), we have \(v_\lambda \to G_A[e]\) in \(C^\infty_{\mathrm{loc}} (\Rset^n \setminus \set{0}, E)\).
 We take now a function \(\eta \in C^\infty (\Rset^n, \Rset)\) such that \(\eta = 0\) on \(\Rset^n \setminus B_2(0)\) and \(\eta = 1\) on \(B_1(0)\) and we set \(u_j \defeq \eta v_{\lambda_j}\) for some sequence \((\lambda_j)_{j \in \Nset}\) in \(\intvo{0}{+\infty}\) going to \(+\infty\).
\end{proof}

\begin{proof}[Proof of the necessity of cancellation in \cref{theorem_cancelling_necessary_Sobolev}]
We let
\(
e \in \bigcap_{\xi \in \Rset^n \setminus \set{0}} A (\xi)[V] 
\)
and we let the sequence \((u_j)_{j \in \Nset}\) be given by \cref{lemma_regularization_G_A} for this vector \(e\).
By assumption \eqref{eq_eeghaib2eiviu0cib4Pei8ut}, we have for every \(j \in \Nset\),
\[
\brk[\Big]{\int_{\Rset^n} \abs{\Deriv^{\ell} u_j}^\frac{n}{n - (k - \ell)}}^{1 - \frac{k - \ell}{n}}
 \le 
 \Cl{cst_keef1xeetheil5uMo3u} \int_{\Rset^n} \abs{A (\Deriv) u_j}\;.
\]
It follows thus by Fatou's lemma that 
\begin{equation}
\label{eq_OoVae3dae5uubieG6Thopae9}
 \brk[\Big]{\int_{B_1 (0)} \abs{\Deriv^\ell G_A [e]}^\frac{n}{n - (k - \ell)}}^{1 - \frac{k - \ell}{n}}
 \le \C \abs{e} <+\infty\;.
\end{equation}
Since \(\ell > n - k\), the function \(\Deriv^\ell G_A[e] : \Rset^n \setminus \set{0} \to V\) is homogeneous of degree \(n - (k - \ell)\) (by \cref{proposition_fundamental_solution} \eqref{it_oocee4ao4DaiK8sheWeechee}); \eqref{eq_OoVae3dae5uubieG6Thopae9} then implies that \(G_A[e] = 0\) on \(\Rset^n\) and thus that \(e = 0\) in view of \cref{proposition_fundamental_solution} \eqref{it_Aix3Ahshe4ahchiekai4Ieph}.
\end{proof}

\subsection{Examples of Cancelling Operators}

We review now several examples of cancelling operators.

\begin{myexample}[Derivative]
The derivative \(A (\Deriv) = \Deriv^k\) is cancelling on \(\Rset^n\) if and only if \(n \ge 2\).
Indeed, for every \(\xi \in \Rset^n\), \(v,w \in \Rset^n\) and \(\eta_1, \dotsc, \eta_k \in \Rset^n\), we have 
\[
\begin{split}
 \dualprod{A (\xi)[v]}{w \otimes \eta_1 \otimes \dotsb \otimes \eta_k}&=
 \dualprod{v \otimes \xi^{\otimes k}}{w \otimes \eta_1 \otimes \dotsb \otimes \eta_k}\\
 &= \dualprod{v}{w} \dualprod{\eta_1}{\xi} \dotsm \dualprod{\eta_k}{\xi}\;.
\end{split}
\]
Since \(n \ge 2\), for every \(\eta_1 \in \Rset^n\), there exists \(\xi \in \Rset^n \setminus \set{0}\) such that \(\dualprod{\eta_1}{\xi} = 0\), and thus 
\[
  \bigcap_{\xi \in \Rset^n \setminus \set{0}}
    A (\xi)[V]
    \subseteq \bigcap_{\substack{\eta_1, \dotsc, \eta_k \in \Rset^n\\ w \in V}} (w \otimes \eta_1 \otimes \dotsb \otimes \eta_k)^\perp =\set{0}\;.
\]
\end{myexample}

\begin{myexample}[Hodge complex]
The operator \(A (\Deriv) = (d, d^*)\) acting on \(C^\infty (\Rset^n, \bigwedge^m \Rset^n)\) is cancelling if and only if \(m \not \in \set{1, n -1}\). Here, \(d\) and \(d^*\) are respectively the exterior differential and codifferential.
Indeed, if \((v, v_*) \in \bigcap_{\xi \in \Rset^n \setminus \set{0}} A (\xi)[V]\), then for every \(\xi \in \Rset^n\), \(\xi \wedge v = 0\) and \(\xi \lrcorner v_* = 0\). Since \(m \not \in \set{1, n - 1}\), this implies that \(v = 0\) and \(v_*= 0\).
As a consequence of \cref{theorem_cancelling_necessary_Sobolev}, one gets Jean \familyname{Bourgain} and Haïm \familyname{Brezis}'s endpoint Sobolev inequality \citelist{\cite{Bourgain_Brezis_2004}\cite{Bourgain_Brezis_2007}*{Cor.\ 17}}:\footnote{Another proof of the endpoint Sobolev estimate for forms \eqref{eq_shein0thah0fie7hoot6ic6E} was given by Loredana \familyname{Lanzani} and Elias M. \familyname{Stein} \cite{Lanzani_Stein_2005}} for every \(u \in C^\infty_c (\Rset^n, \bigwedge^m \Rset^n)\), if \(m \in \set{2, \dotsc, n - 2}\), 
\begin{equation}
\label{eq_shein0thah0fie7hoot6ic6E}
  \brk[\Big]{\int_{\Rset^n} \abs{u}^\frac{n}{n- 1}}^{1 - \frac{1}{n}}
  \le C \int_{\Rset^n} \abs{du} + \abs{d^* u}\;.
\end{equation}
On the other hand when \(m = 1\) and \(n \ge 3\), we have 
\(\bigcap_{\xi \in \Rset^n \setminus \set{0}} A (\xi)[V] = \set{0} \times \bigwedge^{0} \Rset^n\) and when \(m = n - 1 \) and \(n \ge 3\), we have \(\bigcap_{\xi \in \Rset^n \setminus \set{0}} A (\xi)[V] = \smash{\bigwedge^{n} \Rset^n \times \set{0}}\), and finally when \(m = 1\) and \(n = 2\), we have \(\smash{\bigcap_{\xi \in \Rset^n \setminus \set{0}} A (\xi)[V]} = \smash{\bigwedge^{2} \Rset^n\times \bigwedge^{0} \Rset^n}\), so that the operator \(A (\Deriv)\) is not cancelling and the estimate \eqref{eq_shein0thah0fie7hoot6ic6E} does not hold.
\end{myexample}

\begin{myexample}[Symmetric derivative]
The symmetric derivative operator
\(
  A (\Deriv)u = (\Deriv u + (\Deriv u)^*)/2
\) is cancelling if and only if \(n \ge 2\).
Indeed, for every \(w \in \Rset^n\) and \(v \in \Rset^n\),
\[
  \dualprod{A (\xi)[v]}{w \otimes w} = \frac{\dualprod{\xi \otimes v + v \otimes \xi}{w \otimes w}}{2}= \dualprod{\xi}{w} \dualprod{v}{w}\;.
\]
Since \(n \ge 2\), we can choose \(\xi \in \Rset^n \setminus \set{0}\) such  that \(\dualprod{\xi}{w} = 0\), and thus for every \(v \in \Rset^n\),
\(
\dualprod{A (\xi)[v]}{w \otimes w} = \set{0}
\),
and hence
\[
   \bigcap_{\xi \in \Rset^n \setminus \set{0}}
    A (\xi)[\Rset^n]
    \subseteq \Lin^2_{\mathrm{sym}} (\Rset^n, \Rset) \cap
    \bigcap_{w \in \Rset^n}  (w \otimes w)^\perp
    = \set{0}\;.
\]
As a consequence of \cref{theorem_cancelling_necessary_Sobolev}, we recover Monty J. \familyname{Strauss}’s Korn-Sobolev inequality \cite{Strauss_1973}
\begin{equation}
  \brk*{\int_{\Rset^n} \abs{u}^\frac{n}{n- 1}}^{1 - \frac{1}{n}}
  \le C \int_{\Rset^n} \abs{\Deriv_{\mathrm{sym}} u}\;.
\end{equation}
\end{myexample}

\begin{myexample}[Trace-free symmetric derivative%
\footnote{%
The cancellation of the trace-free symmetric derivative if and only if \(n \ge 3\) was obtained as a consequence of its \(\Cset\)-ellipticity by Dominic \familyname{Breit}, Lars \familyname{Diening} and Franz \familyname{Gmeineder}  \cite{Breit_Diening_Gmeineder_2020}*{Ex.\ 2.2}; their proof adapts to \(\lambda > -\smash{\frac{1}{2}}\).
The \(\Cset\)-ellipticity is related to the notion of conformal Killing vectors \cite{Dain_2006}.%
}%
]
If we consider
\(
  A (\Deriv)u \defeq \Deriv_{\mathrm{sym}} u + \lambda \operatorname{div} u
\), then \(A(\Deriv)\) is cancelling when either \(n = 2\) and \(\lambda \ne \smash{ -\frac{1}{2}}\), or \(n\ge 3\).
In particular, in the case \(\lambda = \smash{-\frac{1}{n}}\), \(A (\Deriv)\) is the trace-free symmetric derivative which is cancelling if and only if \(n \ge 3\).
Indeed, for every \(v, w, \xi \in \Rset^n\) such that \(\dualprod{\xi}{w}=0\), we have
\(\dualprod{A (\xi)[v]}{\xi \otimes \xi}=(1 + \lambda) \abs{\xi}^2 \dualprod{\xi}{v}\)
and \(\dualprod{A (\xi)[v]}{w \otimes w} = \lambda \abs{w}^2\dualprod{\xi}{v}\), so that if \(\abs{w} = \abs{\xi}\), one has
\(\dualprod{A (\xi)[v]}{\lambda \xi \otimes \xi - (1 + \lambda) w \otimes w} = 0\) and it follows that
\[
   \bigcap_{\xi \in \Rset^n \setminus \set{0}}
    A (\xi)[\Rset^n] \subseteq \Lin^2_{\mathrm{sym}} (\Rset^n, \Rset) \cap \bigcap_{\substack{\xi, w \in \Rset^n \setminus \set{0}\\ \dualprod{\xi}{w} = 0\\ \abs{\xi} = \abs{w}}} \set{\lambda \xi \otimes \xi - (1 + \lambda) w \otimes w}^\perp
   = \set{0}\;.
\]  
Indeed if \(\lambda \ne -\frac{1}{2}\) and \(n \ge 2\), then any element in the common image should be a tensor represented in any orthonormal basis as a matrix with zero diagonal; when \(\lambda = -\frac{1}{2}\), then the sum of any two distinct elements in the diagonal should be \(0\), which imply they should all be zero if \(n \ge 3\). As a consequence of \cref{theorem_cancelling_necessary_Sobolev}, we get that when either \(n = 2\) and \(\lambda \ne \smash{ -\frac{1}{2}}\), or \(n\ge 3\), for every \(u \in C^\infty_c (\Rset^n, \Rset^n)\),
\begin{equation} 
\brk*{\int_{\Rset^n} \abs{u}^\frac{n}{n- 1}}^{1 - \frac{1}{n}}
  \le C \int_{\Rset^n} \abs{\Deriv_{\mathrm{sym}} u + \lambda \operatorname{div} u}\;.
\end{equation}

\end{myexample}

\begin{myexample}[Laplacian]
\label{example_Laplacian_not_cancelling}
The Laplacian \(A (\Deriv) = \Delta\) is \emph{not} cancelling as an operator from \(V\) to \(V\) on \(\Rset^n\). 
Indeed, for every \(\xi \in \Rset^n \setminus \set{0}\), one has
\(
  A (\xi) [V] = \set{\abs{\xi}^2 v \st v \in V} = V
\),
and thus \(
  \bigcap_{\xi \in \Rset^n \setminus \set{0}}
  A (\xi)[V] = V \ne \set{0}
\),
so that the Laplacian operator is not cancelling.
\end{myexample}

\begin{myexample}[Cauchy-Riemann operator]
\label{example_Cauchy_Riemann_not_cancelling}
  The \emph{Cauchy-Riemann operator} \(A (\Deriv) = \Bar{\partial}\), which is a differential operator from \(\Cset\) to \(\Cset\) on \(\Rset^2\) defined for each function \(u \in C^\infty (\Rset^2, \Cset)\) by \(A (\Deriv) u = (\partial_1 u + i \partial_2 u)/2\) is \emph{not} cancelling. Indeed, for every \(\xi = (\xi_1, \xi_2) \in \Rset^2 \setminus \set{0}\), one has \(A (\xi) \Cset =\Cset\) and thus 
\(
  \bigcap_{\xi \in \Rset^n \setminus \set{0}}
  A (\xi)[\Cset] = \Cset \ne \set{0}
\),
so that the operator \(A (\Deriv)\) is not cancelling.
\end{myexample}

The noncancellation properties of \cref{example_Laplacian_not_cancelling,example_Cauchy_Riemann_not_cancelling} are particular cases of the general fact that if the differential operator \(A (\Deriv)\) is injectively elliptic and if \(\dim V = \dim E\), then for every \(\xi \in \Rset^n\), one has \(\dim A(\xi)[V] = \dim E\) and thus \(A (\xi)[V] = E\) and finally \(\bigcap_{\xi \in \Rset^n \setminus\set{0}} A (\xi)[V] = E\).

The cancellation property can be obtained as a consequence of stronger properties such as \(\Cset\)-ellipticity \cite{Gmeineder_Raita_2019}*{Lem.\ 3.2} or boundary ellipticity \cite{Gmeineder_Raita_VanSchaftingen}. Many of the examples above satisfy a \emph{strong cancellation} property in which the intersection is taken on any two-dimensional subspace \cite{Gmeineder_Raita_VanSchaftingen_2021}.

\section{Duality Estimates}
\label{section_cocancelling}

Given a cancelling and injectively elliptic homogeneous constant coefficient differential operator \(A(\Deriv)\), we have constructed in \cref{proposition_compatibility_conditions} an operator \(L(\Deriv)\), describing the associated compatibility conditions, such that for every \(u \in \smash{C^\infty_c (\Rset^n, V)}\), one has \(L(\Deriv)A (\Deriv)u = 0\) and such that, by \eqref{eq_phicee9aad0ugh2thaeTueph},
\begin{equation*}
\bigcap_{\xi \in \Rset^n \setminus \set{0}} \ker L (\xi) = 
\bigcap_{\xi \in \Rset^n \setminus \set{0}} A(\xi)[V]
= \set{0}\;.
\end{equation*}
Setting \(f \defeq A(\Deriv) u\), we investigate the estimates in terms of \(\norm{f}_{L^1 (\Rset^n)}\) that can be obtained taking into account the condition \(L (\Deriv) f = 0\).

\subsection{Estimates for Cocancelling Operators}

Motivated by the compatibility conditions of \cref{proposition_compatibility_conditions} and the definition of cancelling operators (\cref{definition_cancelling}), we define cocancelling operators: 
\begin{definition}[Cocancelling operator%
\footnote{%
The identification of the cocancellation condition of \cref{definition_cocancelling} as a necessary and sufficient condition for duality estimates in critical Sobolev spaces is due to the author \cite{VanSchaftingen_2013}; similar conditions appears in characterizations of the dimension of measures lying in the kernel of operators \citelist{\cite{Roginskaya_Wojciechowski_2006}\cite{ArroyoRabasa_DePhilippis_Hirsch_Rindler_2019}\cite{DePhilippis_Rindler_2016}}.%
}%
]
  \label{definition_cocancelling}
  Let \(n \in \Nset \setminus \set{0}\) and let \(E\) and \(F\) be finite-dimensional vector spaces.
  A homogeneous constant coefficient differential operator \(L(\Deriv)\) from \(E\) to \(F\) on \(\Rset^n\) is \emph{cocancelling} whenever 
\begin{equation*}
\bigcap_{\xi \in \Rset^n \setminus \set{0}} \ker L (\xi) = \set{0}\;. 
\end{equation*}
\end{definition}

When \(n = 1\), one notes that \(L(\Deriv)\) is cocancelling if and only if \(L (\Deriv) = 0\); hence the notion will only be relevant for \(n \ge 2\).

\begin{proposition}[Characterization of cocancelling operators%
\footnote{%
\Cref{proposition_equiv_cocanc} characterizing cocancelling operators is due to the author \cite{VanSchaftingen_2013}.%
}%
]
  \label{proposition_equiv_cocanc}
  Let \(n \in \Nset\setminus \set{0}\), let \(E\) and \(F\) be finite-dimensional vector spaces,
  and let \(L(\Deriv)\) be a homogeneous linear differential operator of order \(k \in \Nset \setminus \set{0}\) on \(\Rset^n\) from \(E\) to \(F\).
  The following are equivalent
  \begin{enumerate}[(i)]
    \item\label{itcocancelling} 
    the operator \(L(\Deriv)\) is cocancelling, 
    \item\label{itVanishDirac} for every \(e \in E\) such that \(L(\Deriv)\, (e \delta_0) =0\) in the sense of distributions, one has \(e = 0\),
    \item\label{itVanishL1} for every \(f \in L^1(\Rset^n, E)\) such that \(L(\Deriv)f = 0\) in the sense of distributions, one has \(
    \int_{\Rset^n} f = 0                                                                                                                                                                                                                                                \),
    \item\label{itVanishCinfty} for every \(f \in C^\infty_c(\Rset^n, E)\) such that \(L(\Deriv)f = 0\), one has 
    \(
      \int_{\Rset^n} f = 0
    \).
  \end{enumerate}
\end{proposition}

Here \(\delta_0\) denotes Dirac's measure at \(0\) on \(\Rset^n\).

\begin{proof}[Proof of \cref{proposition_equiv_cocanc}]
  We assume that \eqref{itcocancelling} holds, that is, that the operator \(L(\Deriv)\) is cocancelling and we consider a vector \(e \in E\) such that \(L(\Deriv)(e \delta_0)=0\) in the sense of distributions.
  For every \(\varphi \in \mathcal{S} (\Rset^n, F)\), by definition of the distributional derivative and properties of the Fourier transform \(\mathcal{F} \varphi\) of \(\varphi\), we have
  \[ 
  \begin{split}
    \int_{\Rset^n} \dualprod{(2\pi i)^kL(\xi)[e]}{\varphi(\xi)} \dif \xi
    = \dualprod{L(\Deriv) (e \delta_0)}{\mathcal{F}^{-1} \varphi}
    = 0\;.
  \end{split}
  \]
  and hence, since \(L : \Rset^n \to \Lin(E, F)\) is continuous, for every \(\xi \in \Rset^n\), we have
  \(
  L(\xi)[e]=0
  \) and thus \(e \in \bigcap_{\xi \in \Rset^n \setminus \set{0}} \ker L (\xi)\).
  By definition of cocancelling operator (\cref{definition_cocancelling}), we have \(e=0\) and thus \eqref{itVanishDirac} holds.
  
  We assume now that \eqref{itVanishDirac} holds. 
  Given \(f \in L^1(\Rset^n, E)\) such that \(L(\Deriv)f=0\) and given \(\lambda \in \intvo{0}{+\infty}\), we define the function 
  \(f_\lambda : \Rset^n \to E\) for each \(x \in \Rset^n\) by
  \(
  f_\lambda(x) \defeq \smash{\frac{1}{\lambda^n} f\bigl(\frac{x}{\lambda}\bigr)}
  \).
  Since \(f_\lambda \to \delta_0 \smash{\int_{\Rset^n} f}\) in the sense of distributions as \(\lambda \to 0\),
  we have \(L(\Deriv)f_\lambda \to L(\Deriv) \smash{(\delta_0 \int_{\Rset^n}f)}\) in the sense of distributions as \(\lambda \to 0\).
  Since  \(L (\Deriv)\) is homogeneous and has constant coefficients, we have for every \(\lambda \in \intvo{0}{+\infty}\), \(L (\Deriv) f_\lambda = 0\) and hence \(L(\Deriv)(\delta_0 \int_{\Rset^n} f)=0\). 
  By our assumption \eqref{itVanishDirac} , we deduce that \(\int_{\Rset^n} f = 0\) and thus \eqref{itVanishL1} holds.
  
  Since we have \(C^\infty_c (\Rset^n, E) \subset L^1 (\Rset^n, E)\), the assertion \eqref{itVanishL1} implies immediately the assertion \eqref{itVanishCinfty}.
  
 Finally, let us assume that \eqref{itVanishCinfty} holds. 
 Let \(e \in \bigcap_{\xi \in \Rset^n \setminus \set{0}} \ker L(\xi) \) and let \(\psi \in C^\infty_c(\Rset^n, E)\) such that \(\int_{\Rset^n} \psi=1\). 
 For every \(x \in \Rset^n\), we have by the Fourier inversion formula,
  \[
  \bigl(L(\Deriv) (\psi e)\bigr) (x)=\int_{\Rset^n} e^{2\pi i \dualprod{\xi}{x}} (2\pi i)^k L(\xi)[e] \mathcal{F} \psi(\xi) \dif \xi = 0\;.
  \]
  By \eqref{itVanishCinfty}, we conclude that \(e = \int_{\Rset^n} \psi e = 0\). 
  The operator \(L(\Deriv)\) is thus cocancelling by definition (\cref{definition_cocancelling}) and \eqref{itcocancelling} holds.
\end{proof}

The cocancelling operator is a necessary and sufficient condition for some duality estimate with critical Sobolev spaces.

\begin{theorem}
[Duality estimate for cocancelling operators%
\footnote{%
Estimates of the form \eqref{eq_koo4lui9Eipo7XeenaeghaoR} originates in an estimate on circulation integrals by Jean \familyname{Bourgain}, Haïm \familyname{Brezis} and Petru \familyname{Mironescu} \cite{Bourgain_Brezis_Mironescu_2004} (see also \cite{VanSchaftingen_2004_circ}).
  The estimate  \eqref{eq_koo4lui9Eipo7XeenaeghaoR} was first proved when \(L (\Deriv)\) is the divergence \citelist{\cite{Bourgain_Brezis_2004}\cite{Bourgain_Brezis_2007}} (see also \cite{VanSchaftingen_2004_div}); it was then extended to the case where \(L (\Deriv)\) is in some class of second-order operators \cite{VanSchaftingen_2004_ARB} or higher-operators \cite{Bourgain_Brezis_2007}, and finally for the higher-order divergence \cite{VanSchaftingen_2008} and then deduced by an algebraic argument for a general cocancelling operator \cite{VanSchaftingen_2013}.
  The approach presented here is a direct proof of the estimate \eqref{eq_koo4lui9Eipo7XeenaeghaoR}. 
}%
]
\label{theorem_cocancelling}
  Let \(n \in \Nset \setminus \{0, 1\}\), let \(V\) and \(E\) be finite-dimensional vector spaces, 
  and let \(L (\Deriv)\) be a homogeneous constant coefficient differential operator from \(E\) to \(F\) on \(\Rset^n\). There exists a constant \(C \in \intvo{0}{+\infty}\) such that
for every \(f \in C^\infty (\Rset^n, E) \cap L^1 (\Rset^n, E)\) that satisfies \(L (\Deriv) f = 0\) and every \(\varphi \in C^\infty_c (\Rset^n, E)\), 
\begin{equation}
  \label{eq_koo4lui9Eipo7XeenaeghaoR}
  \abs[\Big]{\int_{\Rset^n} \dualprod{f}{\varphi}\,}
  \le 
  C \int_{\Rset^n} \abs{f}\; \brk[\Big]{\int_{\Rset^n} \abs{\Deriv \varphi}^n}^\frac{1}{n}
\end{equation}
if and only if the operator \(L (\Deriv)\) is cocancelling.
\end{theorem}

\Cref{theorem_cocancelling} can be seen as a substitute for the missing embedding of the critical Sobolev space 
\(\dot{W}^{1, n} (\Rset^n, E)\) in \(L^\infty (\Rset^n, E)\), which would be equivalent to having \eqref{eq_koo4lui9Eipo7XeenaeghaoR} without the condition \(L (\Deriv) f = 0\).

\begin{proof}[Proof of the necessity of cocancellation in \cref{theorem_cocancelling}]
  We consider a sequence \((\varphi_j)_{j \in \Nset}\) in \(C^\infty_c (\Rset^n)\), such that 
  \begin{enumerate}[(a)]
    \item 
      \label{it_ugai2paThaexogaic6ee1aur}
      for every \(j \in \Nset\) and \(x \in \Rset^n\), one has \(0 \le \varphi_j (x) \le 1\),
    \item 
      \label{eq_av9ohCh7Eicohjuug1Ahlohz}
      for every \(x \in \Rset^n\), the sequence \((\varphi_j (x))_{j \in \Nset}\) converges to \(1\),
    \item 
      \label{eq_aK8owohKeCh9Loo3si9migha}
      \(\lim_{n \to \infty} \int_{\Rset^n} \abs{\Deriv \varphi_j}^n = 0\).
  \end{enumerate}
  (One can take for instance \(\varphi_j (x)\defeq \phi((\ln \abs{x})/j)\) with \(\phi \in C^\infty (\Rset, \Rset)\) such that \(\phi = 1\) on \(\intvl{-\infty}{1}\), \(0 \le \phi \le 1\) on \(\intvc{1}{2}\) and \(\phi = 0\) on \(\intvr{2}{+\infty}\).)
  By our assumption \eqref{eq_koo4lui9Eipo7XeenaeghaoR} we have for every \(j \in \Nset\) and \(e \in E\)
  \begin{equation}
    \label{eq_ahzohcahfooB9aushiemoo4e}
    \abs[\Big]{\int_{\Rset^n} \dualprod{f}{\varphi_j e}\,}
    \le 
    C \abs{e} \int_{\Rset^n} \abs{f}\; \brk[\Big]{\int_{\Rset^n} \abs{\Deriv \varphi_j}^n}^\frac{1}{n}\;.
  \end{equation}
By Lebesgue's dominated convergence theorem, we have in view of \eqref{it_ugai2paThaexogaic6ee1aur} and \eqref{eq_av9ohCh7Eicohjuug1Ahlohz}
\begin{equation}
  \label{eq_aeSieyaiWaez2gahsahSh8ve}
\lim_{j \to \infty} \int_{\Rset^n} \dualprod{f}{\varphi_j} = \int_{\Rset^n} \dualprod{f}{e}\;.
\end{equation}
It follows then from \eqref{eq_ahzohcahfooB9aushiemoo4e} and \eqref{eq_aeSieyaiWaez2gahsahSh8ve}, taking into account \eqref{eq_aK8owohKeCh9Loo3si9migha}, that 
\(
\int_{\Rset^n} f = 0
\),
and thus by \cref{proposition_equiv_cocanc}, the operator \(L (\Deriv)\) is cocancelling.
\end{proof}

We review now several examples of cocancelling operators.

\begin{myexample}[Divergence]
  The \emph{divergence} operator \(L (\Deriv) = \operatorname{div}\)  on \(\Rset^n\) is cocancelling when \(n \ge 2\). 
  Indeed, for every \(\xi \in \Rset^n\), 
  \(
  \ker L (\xi) = \set{\xi}^\perp
  \),
  and thus, since \(n \ge 2\),  
  \(
  \smash{
  \bigcap_{\xi \in \Rset^{n} \setminus \set{0}} \ker L (\xi)} = \set{0}
  \).
  As a consequence of \cref{theorem_cocancelling}, one gets that if \(f \in L^1 (\Rset^n, \Rset^n)\) satisfies \(\operatorname{div} f = 0\) in the sense of distributions then \eqref{eq_koo4lui9Eipo7XeenaeghaoR} holds.\footnote{This result was first obtained as a consequence of a stronger estimate by Haïm \familyname{Brezis} and Jean \familyname{Bourgain} \citelist{\cite{Bourgain_Brezis_2004}\cite{Bourgain_Brezis_2007}}; a direct proof of the result was given by the author \cite{VanSchaftingen_2004_div}.}
  This estimate can be seen, through Stanislav K. \familyname{Smirnov}’s result on the approximation of divergence-free measures \cite{Smirnov_1993}, to be equivalent to the estimate on circulation integrals of Jean \familyname{Bourgain}, Haïm \familyname{Brezis} and Petru \familyname{Mironescu} \cite{Bourgain_Brezis_Mironescu_2004}:\footnote{A direct proof was given by  the author \cite{VanSchaftingen_2004_circ}; the estimate is equivalent when \(n = 2\) to the classical isoperimetric theorem, raising intriguing questions on sharp constants in higher dimensions \cite{Brezis_VanSchaftingen_2008}.} if \(\Gamma \subset \Rset^n\) is a \emph{closed} curve with tangent vector \(t\) and length \(\abs{\Gamma}\), then for every vector field \(\varphi \in C^\infty_c (\Rset^n, \Rset^n)\), one has
\begin{equation}
 \abs[\Big]{\int_{\Gamma} \dualprod{\varphi}{t}\,}
 \le \C \,\abs{\Gamma}\, \brk*{\int_{\Rset^n} \abs{\Deriv \varphi}^n}^\frac{1}{n}\;.
\end{equation}
\end{myexample}

\begin{myexample}[Exterior derivative]
  If \(E = \bigwedge^m \Rset^n\), with \(m \in \{0, \dotsc, n - 1\}\), the \emph{exterior derivative} \(L (\Deriv) = d\)
  is cocancelling. Indeed, one has 
  \[
  \ker L (\xi) = \set*{\alpha \in {\textstyle \bigwedge^{m} \Rset^n} \st \alpha \wedge \xi = 0}\;.
  \]
  Thus if \(\alpha \in \ker L (\xi)\), then \(\alpha\) has its restriction to \(\set{\xi}^\perp\) being equal to \(0\).
  If this happens for all \(\xi \in \Rset^n\), then \(\alpha = 0\).
\end{myexample}

\begin{myexample}[Saint-Venant operator]
The Saint-Venant differential operator \(L (\Deriv)\) from \(\Lin^2_{\mathrm{sym}} (\Rset^n, \Rset)\) to \(\Lin^4 (\Rset^n, \Rset)\) defined for every \(u \in C^\infty (\Rset^n, \Lin^2_{\mathrm{sym}} (\Rset^n, \Rset))\), and \(v_1, \dotsc, v_4 \in \Rset^n\) by 
\begin{equation*}
\begin{split}
(L (\Deriv) u)[v_1, \dotsc, v_4] 
&\defeq \dualprod{v_1 \otimes v_2}{\Deriv^2 u[v_3, v_4]}
+ \dualprod{v_3 \otimes v_4}{\Deriv^2 u[v_1, v_2]} \\
&\qquad - \dualprod{v_1 \otimes v_4}{\Deriv^2 u[v_2, v_3]} - \dualprod{v_2 \otimes v_3}{\Deriv^2 u[v_1, v_4]}
\end{split}
\end{equation*}
is cocancelling when \(n \ge 2\).
This operator corresponds to the Saint-Venant compatibility conditions \eqref{eq_Xah5quaejieRahche9roova8} of the symmetric derivative. 
If for every \(\xi \in \Rset^n \setminus \set{0}\), we have \(e \in \ker L (\xi) \subset \Lin^2_{\mathrm{sym}} (\Rset^n, \Rset)\), then for every \(\xi \in \Rset^n \setminus \set{0}\) and every \(v \in \Rset^n\),
\[
(L(\xi)e)[v, v, \xi, \xi]
= e(v, v) \abs{\xi}^4 + e(\xi, \xi) \dualprod{\xi}{v}^2 - 2 e(v, \xi) \dualprod{\xi}{v}\abs{\xi}^2 = 0\;,
\]
so that if one has \(\dualprod{\xi}{v} = 0\), then 
\(
e (v, v) \abs{\xi}^4
= 0
\) and it follows then that since \(n \ge 2\)
\begin{equation*}
\bigcap_{\xi \in \Rset^n \setminus \set{0}} \ker L (\xi)
\subseteq \bigcap_{\substack{\xi \in \Rset^n \setminus \set{0}\\
v \in \xi^\perp}} (v \otimes v)^\perp
= \set{0}\;.
\end{equation*}
As a consequence of \cref{theorem_cocancelling}, one gets that if \(f \in L^1(\Rset^n, \Lin^2_{\mathrm{sym}} (\Rset^n, \Rset))\) satisfies \(L(\Deriv) f= 0\), then the estimate \eqref{eq_koo4lui9Eipo7XeenaeghaoR} holds.\footnote{The estimate \eqref{eq_koo4lui9Eipo7XeenaeghaoR} under the Saint-Venant compatibility conditions \eqref{eq_Xah5quaejieRahche9roova8} was obtained by the author \cite{VanSchaftingen_2004_ARB}.}
\end{myexample}

\begin{myexample}[Higher-order divergence]
 The higher-order divergence \(L (\Deriv) = \operatorname{div}^k\) is cocancelling on symmetric \(k\)-linear forms.
 Indeed, for every \(\xi \in \Rset^n\) and \(\tau \in \Lin_{\mathrm{sym}}^k(\Rset^n, \Rset)\), one has \(
   L(\xi)[\tau] = \tau[\xi, \dotsc, \xi]
 \)
 and \(\bigcap_{\xi \in \Rset^n \setminus \set{0}} \ker L (\xi) = \set{0}\).
\end{myexample}

\subsection{Estimates under a Cocancelling Condition}

The essential ingredient in the original proof of the estimate in \cref{theorem_cocancelling} by Jean \familyname{Bourgain} and Haïm \familyname{Brezis} is an approximation result obtained through a Littlewood-Paley decomposition: For every \(\varepsilon \in \intvo{0}{+\infty}\), there exists a constant \(M_\varepsilon \in \intvo{0}{+\infty}\) such that for every function \(u \in \dot{W}^{1, n} (\Rset^n, \Rset)\), one can construct a function \(v \in (\dot{W}^{1, n} \cap L^\infty)(\Rset^n, \Rset)\) satisfying
\begin{equation*}
  \norm{\Deriv v}_{L^n (\Rset^n)} + \norm{v}_{L^\infty (\Rset^n)} \le M_\varepsilon \norm{\Deriv u}_{L^n(\Rset^n)}\;,
\end{equation*}
and 
\begin{equation*}
  \norm{\Deriv' u - \Deriv'v}_{L^n (\Rset^n)} \le \varepsilon \norm{\Deriv u}_{L^n (\Rset^n)}\;,
\end{equation*}
where \(\Deriv'\) is the derivative with respect to the \(n -1\) first variables; this gives a stronger estimate with \(\norm{f}_{L^1 (\Rset^n)}\) replaced by the weaker norm \(\norm{f}_{L^1 (\Rset^n) + \dot{W}^{-1,n/(n-1)}(\Rset^n)}\) \citelist{\cite{Bourgain_Brezis_2004}\cite{Bourgain_Brezis_2007}}.%
\footnote{
  This approximation result generalized Jean \familyname{Bourgain} and Haïm \familyname{Brezis}'s similar approximation result in their work on the divergence equation \citelist{\cite{Bourgain_Brezis_2002}\cite{Bourgain_Brezis_2003}}, and yields some estimates for a large class of cocancelling operators \cite{Bourgain_Brezis_2007}. Algebraic arguments to reach the whole class of cocancelling operators are due to the author \citelist{\cite{VanSchaftingen_2008}\cite{VanSchaftingen_2013}}.
} 

We present here a proof of \cref{theorem_cocancelling} based on Morrey's embedding on \(\Rset^{n - 1}\) and some appropriate integration by parts.%
\footnote{%
This strategy goes back to the author's elementary proof of an inequality of Jean \familyname{Bourgain}, Haïm \familyname{Brezis} and Petru \familyname{Mironescu} on circulation integrals    \cite{VanSchaftingen_2004_circ}, which was successively extended to divergence-free vector fields \cite{VanSchaftingen_2004_div}, some class of second order operator \cite{VanSchaftingen_2004_ARB} and the higher-order divergence \cite{VanSchaftingen_2008}.
      The latter proof together with appropriate algebraic manipulations yielded the general estimate for cocancelling operators \cite{VanSchaftingen_2013}.%
}

\subsubsection{About Hölder-Continuous Functions}

The first analytical tool that we will used for the proof of the sufficiency part in \cref{theorem_cocancelling} will be the Morrey-Sobolev embedding applied on some hyperplanes.

\begin{proposition}[Morrey-Sobolev embedding%
\footnote{%
For the statement and the proof of the Morrey-Sobolev embedding (\cref{proposition_Morrey_Sobolev}), see for example \citelist{\cite{Hormander_1990_I}*{Th.\ 4.5.11}\cite{Mazya_2011}*{\S 1.4.5}\cite{Adams_Fournier_2003}*{Lem.\ 4.28}\cite{Morrey_1966}*{Th.\ 3.5.2}\cite{Brezis_2011}{Th.\ 9.12}\cite{Willem_2013}*{Lem.\ 6.4.13}\cite{Evans_2010}*{\S 5.6.2}}.%
}%
]
  \label{proposition_Morrey_Sobolev}
Let \(\ell \in \Nset \setminus \set{0}\), let \(V\) be finite-dimensional vector space and \(p \in (\ell, +\infty)\). 
There exists a constant \(C \in \intvo{0}{+\infty}\) such that for every 
\(\psi \in C^\infty_c (\Rset^\ell, V)\), 
\begin{equation*}
\abs{\psi (y) - \psi (x)} \le C \abs{y - x}^{1 - \frac{\ell}{p}}\brk[\Big]{\int_{\Rset^\ell} \abs{\Deriv \psi}^p }^\frac{1}{p}\;.
\end{equation*}
\end{proposition}

The Morrey-Sobolev embedding can be reinterpreted as a bound on the Hölder seminorm \(\seminorm{\varphi}_{C^{0, 1-\ell/p} (\Rset^\ell, V)}\).

\begin{definition}[Hölder seminorm]
  \label{definition_Holder}
  Given \(\ell \in \Nset \setminus \set{0}\), \(\alpha \in \intvo{0}{1}\) and a finite-dimensional vector space \(V\), we define the \emph{Hölder seminorm} for every function \(\psi : \Rset^\ell \to V\) by 
\begin{equation*}
\seminorm{\psi}_{C^{0, \alpha} (\Rset^\ell)}
\defeq
\sup \,\set[\Big]{ 
\frac{ \abs{\psi (y) - \psi (x)}}
  {\abs{y - x}^\alpha }\st x, y \in \Rset^\ell }\;.
\end{equation*}
\end{definition}

The next result shows how estimates in terms of the uniform norm of the function and of its derivative can be combined into an estimate in terms of the Hölder seminorm.

\begin{lemma}[Interpolation to Hölder spaces]
  \label{lemma_Holder_interpolation}
Let \(\ell \in \Nset \setminus \set{0}\), let \(V\) be a linear space, let the mapping \(T : C^\infty_c (\Rset^\ell, V)\to \Rset\) be linear, and let \(\alpha \in (0, 1)\).
There exists a constant \(C \in \intvo{0}{+\infty}\) for which if there are constants \(A, B \in \Rset\) such that for every \(\psi \in C^\infty_c (\Rset^\ell, V)\), one has
\begin{equation}
\label{eq_too8IesooHofoed9Tiexai5A}
\abs{\dualprod{T}{\psi}} \le A\norm{\psi}_{L^\infty(\Rset^\ell)}
\end{equation}
and 
\begin{equation}
\label{eq_Aequeu3wa3ePh6Ya3uikook0}
\abs{\dualprod{T}{\psi}} \le B \norm{\Deriv \psi}_{L^\infty(\Rset^\ell)}\;,
\end{equation}
then for every \(\psi \in C^\infty_c (\Rset^\ell, V)\), we have
\begin{equation}
\label{eq_Tiphiquoh1tie9AesaiceiCe}
  \abs{\dualprod{T}{\psi}} \le C A^{1 - \alpha} B^\alpha \seminorm{\psi}_{C^{0, \alpha}(\Rset^\ell)}\;.
\end{equation}
\end{lemma}
\begin{proof}
We fix a function \(\eta \in C^\infty_c (\Rset^\ell, \Rset)\) such that \(\int_{\Rset^\ell} \eta = 1\). 
Given a function \(\psi \in C^\infty_c (\Rset^\ell, \Rset)\), we define for each \(\lambda \in \intvo{0}{+\infty}\), the function \(\psi_{\lambda} : \Rset^\ell \to V\) by setting for each \(x \in \Rset^\ell\)
\begin{equation}
\label{eq_eingaiYah0Eejie8ooQui9ni}
  \psi_\lambda (x) 
  \defeq 
  \int_{\Rset^\ell} \eta (y) \psi (x - \lambda y) \dif y= 
  \int_{\Rset^\ell} \eta \brk*{\tfrac{x}{\lambda} - z} \psi (\lambda z) \dif z\;.
\end{equation}
Since \(\int_{\Rset^\ell} \eta = 1\), we observe that for every \(x \in \Rset^\ell\) we have by \eqref{eq_eingaiYah0Eejie8ooQui9ni}
\begin{equation}
\label{eq_ieghaifahnge3eeSeiyiequ7}
    \psi_\lambda (x) - \psi (x) 
  = 
    \int_{\Rset^\ell} \eta (y)\, \bigl(\psi (x - \lambda y) - \psi (x)\bigr) \dif y\;, 
\end{equation}
and thus by \eqref{eq_ieghaifahnge3eeSeiyiequ7} and by \cref{definition_Holder},
\begin{equation}
  \label{eq_ahx1Ucheik4heGohx}
\begin{split}
\abs{\psi_\lambda (x) - \psi (x)} 
&\le \int_{\Rset^\ell} \abs{\eta (y)} \, \abs{\psi (x - \lambda y) - \psi (x)} \dif y\\
&\le \seminorm{\psi}_{C^{0, \alpha} (\Rset^\ell)} \lambda^\alpha 
\int_{\Rset^\ell} \abs{\eta (y)} \, \abs{y}^\alpha \dif y  = \Cl{cst_nooSu1eyohreiy8ei}   \seminorm{\psi}_{C^{0, \alpha} (\Rset^\ell)} \lambda^\alpha\; .
\end{split}
\end{equation}
On the other hand, we have for each \(x \in \Rset^\ell\), by \eqref{eq_eingaiYah0Eejie8ooQui9ni} again 
\begin{equation}
  \label{eq_ka7IcieLie5Isaifa7xieshu}
\begin{split}
\Deriv \psi_\lambda (x)
&= 
\frac{1}{\lambda} \int_{\Rset^\ell} \Deriv \eta \brk*{\tfrac{x}{\lambda} - z} \psi (\lambda z) \dif z
\\
&=
\frac{1}{\lambda} \int_{\Rset^\ell} \Deriv \eta (y) \psi (x - \lambda y) \dif y\\ 
&= \frac{1}{\lambda} \int_{\Rset^\ell} \Deriv \eta (y) \bigl(\psi (x - \lambda y) - \psi (x)\bigr)\dif y\;,
\end{split}
\end{equation}
since \(\int_{\Rset^\ell} \Deriv \eta = 0\).
Hence, we have by \eqref{eq_ka7IcieLie5Isaifa7xieshu} and by \cref{definition_Holder}
\begin{equation}
  \label{eq_Ahdai9gareT3phein}
  \begin{split}
    \abs{\Deriv \psi_\lambda (x)}
    &\le
    \frac{1}{\lambda} 
    \int_{\Rset^\ell} 
    \abs{\Deriv \eta (y)}\, \abs{\psi (x - \lambda y) - \psi (x)} \dif y\\
    &\le 
    \frac{\seminorm{\psi}_{C^{0, \alpha} (\Rset^\ell)}}{\lambda^{1 - \alpha}}
      \int_{\Rset^\ell} \abs{\Deriv \eta (y)} \,\abs{y}^\alpha \dif y\\
      &= 
      \Cl{cst_iehuje0xei8mong1W} \frac{\seminorm{\psi}_{C^{0, \alpha} (\Rset^\ell)}}{\lambda^{1 - \alpha}}\;.
  \end{split}
\end{equation}
By our assumption \eqref{eq_too8IesooHofoed9Tiexai5A} and \eqref{eq_Aequeu3wa3ePh6Ya3uikook0} and by \eqref{eq_ahx1Ucheik4heGohx} and \eqref{eq_Ahdai9gareT3phein}, we have 
\begin{equation}
  \label{eq_maeNg8quaeho8ohDaye2ohVa}
\begin{split}
\abs{\dualprod{T}{\psi}}
&\le
\abs{\dualprod{T}{\psi  - \psi_\lambda}} + \abs{\dualprod{T}{\psi_\lambda}}\\
&\le 
A \norm{\psi - \psi_\lambda}_{L^\infty (\Rset^\ell)} + B \norm{\Deriv \psi_\lambda}_{L^\infty (\Rset^\ell)}\\
&\le 
\Bigl(
\Cr{cst_nooSu1eyohreiy8ei} A  \lambda^\alpha
+ 
 \frac{\Cr{cst_iehuje0xei8mong1W} B}{\lambda^{1 - \alpha}}\Bigr)\seminorm{\psi}_{C^{0, \alpha} (\Rset^\ell)}\;.
\end{split}
\end{equation}
If we choose \(\lambda = B/A\) in \eqref{eq_maeNg8quaeho8ohDaye2ohVa}, we obtain \eqref{eq_Tiphiquoh1tie9AesaiceiCe} with \(C \defeq \Cr{cst_nooSu1eyohreiy8ei}
+ 
\Cr{cst_iehuje0xei8mong1W}\).
\end{proof}

\subsubsection{Integrating by Parts}

The proof of the sufficiency part of \cref{theorem_cocancelling} will also rely on an estimate of an integral on a hyperplane.

\begin{lemma}[Integration by parts estimate]
  \label{lemma_parts_general}
  Let \(n \in \Nset\setminus \set{0}\) and let \(E\) and \(F\) be finite-dimensional vector spaces.
If \(L (\Deriv)\) is a homogenous constant coefficients differential operator of order \(k \in \Nset \setminus \set{0}\) from \(E\) to \(F\) on \(\Rset^n\),
then there exists a constant \(C \in \intvo{0}{+\infty}\) such that for every \(f \in (C^\infty \cap L^1) (\Rset^n, E)\) satisfying \(L (\Deriv)f = 0\), every affine hyperplane \(\Sigma \subset \Rset^n\) and every \(\psi \in C^1 (\Sigma, F)\), one has 
\begin{equation}
\label{eq_ahkaek3ceib6OgaeSh4yuchi}
\abs[\Big]{
\int_{\Sigma} \dualprod{\psi}{L (\nu_\Sigma)[f]}\,
}
\le 
C \,\norm{\Deriv \psi}_{L^\infty(\Sigma)}
\int_{\Rset^n} \abs{f}\;.
\end{equation}
\end{lemma}

Here \(\nu_\Sigma\) denotes the unit normal vector to the hyperplane \(\Sigma\). We do not assume \(f\) to be compactly supported to avoid issues with the rigidity of the condition \(L(\Deriv) f = 0\).

In order to outline the ideas, we first prove \cref{lemma_parts_general} for the divergence operator.

\begin{proof}[Proof of \cref{lemma_parts_general} when \(L (\Deriv)= \operatorname{div}\)]
Without loss of generality we assume that \(\Sigma = \Rset^{n - 1} \times \set{0}\) and that \(\nu_\Sigma = \nu \defeq (0, \dotsc, 0, 1)\).
We define the function \(\Psi : \Rset^n_+ \to \Rset\) where \(\Rset^n_+ \defeq \Rset^n \times \intvo{0}{+\infty}\), by setting for each \((x', x_n) \in \Rset^n_+\), \(\Psi (x', x_n) \defeq \psi (x')\).
We fix a function \(\theta \in C^1 (\intvr{0}{+\infty}, \Rset)\) such that \(\theta = 1\) on \(\intvc{0}{1}\) and \(\theta = 0\) on \(\intvr{2}{+\infty}\), and for every \(\lambda > 0\) we define the function \(\theta_\lambda \in C^1_c (\Rset^n_+, \Rset)\) for each \(x = (x', x_n) \in \Rset^{n}_+\) by \(\theta_\lambda (x) \defeq \theta (x_n/\lambda)\).
For every \(\lambda > 0\), the function \(\theta_\lambda \Psi\) is compactly supported in \(\smash{\Bar{\Rset}^n_+}\), and we have thus by the Gauß-Ostrogradski divergence theorem and by the Leibniz rule for the divergence
\begin{equation}
\label{eq_aw2thoo6Ahlei5thaerohlej}
\begin{split}
\int_{\Rset^{n - 1} \times \set{0}} \dualprod{\nu}{f }\, \psi &= \int_{\Rset^{n - 1}\times \set{0}} \dualprod{\nu}{f \Psi \theta_\lambda}
  =   -
  \int_{\Rset^n_+} \divergence (f \Psi \theta_\lambda)\\
&=-
  \int_{\Rset^n_+} (\divergence f)\, \Psi \theta_\lambda - \int_{\Rset^n_+} \dualprod{\Deriv \Psi}{f \theta_\lambda}
- \int_{\Rset^n_+} \dualprod{\Deriv \theta_\lambda}{f \Psi}\;.
\end{split}
\end{equation}
We deduce then from \eqref{eq_aw2thoo6Ahlei5thaerohlej}, since \(\operatorname{div}f = 0\) in \(\Rset^n_+\), that 
\begin{equation}
  \label{eq_imeiz7leeDa2eelaegh}
  \begin{split}
\abs[\Big]{\int_{\Rset^{n - 1} \times \set{0}} \dualprod{\nu}{f}\, \psi\,}
&\le \brk[\Big]{\norm{\Deriv \Psi}_{L^\infty (\Rset^n_+)} 
+ \frac{\Cl{cst_ia1Ia0tae2Hiigaaxohshood}}{\lambda} \norm{\Psi}_{L^\infty(\Rset^n_+)}} \int_{\Rset^{n}_+} \abs{f}\\
&= \brk[\Big]{\norm{\Deriv \psi}_{L^\infty (\Rset^{n - 1})} 
+ \frac{\Cr{cst_ia1Ia0tae2Hiigaaxohshood}}{\lambda} \norm{\psi}_{L^\infty(\Rset^{n - 1})}} \int_{\Rset^{n}_+} \abs{f}\;, 
\end{split}
\end{equation}
Letting \(\lambda \to +\infty\) in \eqref{eq_imeiz7leeDa2eelaegh}, we obtain \eqref{eq_ahkaek3ceib6OgaeSh4yuchi}.
\end{proof}

We proceed now to the proof of \cref{lemma_parts_general} with a general operator \(L(\Deriv)\).

\begin{proof}[Proof of \cref{lemma_parts_general} in the general case]
Up to a rotation of the Euclidean space \(\Rset^n\), we can assume that \(\Sigma = \Rset^{n - 1} \times \set{0}\) and \(\nu_\Sigma = \nu \defeq (0, \dotsc, 0, 1)\).
We fix a function \(\eta \in C^1_c (\Rset^{n - 1}, \Rset)\) such that \(\int_{\Rset^{n - 1}} \eta = 1\).
Given \(\psi \in C^1_c (\Rset^{n - 1}, F)\), we define the function \(\Psi : \Rset^n_+ \to F\)
for each \((x', x_n) \in \Rset^{n}_+\) by
\begin{equation}
\label{eq_eNoo6iu8roh0ulieCepohFah}
  \Psi (x', x_n) \defeq
  \frac{1}{x_n{}^{n - 1}}
  \int_{\Rset^{n - 1}} \psi (y) \eta \bigl(\tfrac{x' - y}{x_n} \bigr) \dif y
  =\int_{\Rset^{n - 1}} \psi (x  - x_n z) \eta (z) \dif z
  \;.
\end{equation}
We immediately have from \eqref{eq_eNoo6iu8roh0ulieCepohFah} that for every  \((x', x_n) \in \Rset^{n}_+\) 
\begin{equation}
  \label{eq_yae4Ahp6eir9reith9xahtae}
  \abs{\Psi (x', x_n)} \le \C \norm{\psi}_{L^\infty (\Rset^{n - 1})}\;.
\end{equation}
We next observe that for every point \((x', x_n) \in \Rset^{n - 1} \times \intvo{0}{+\infty} = \Rset^n_+\) 
and every vector \(v = (v', v_n) \in \Rset^n\), we have 
\begin{equation}
  \label{eq_ExephahvaegheeJie6reW2me}
  \begin{split}
\Deriv \Psi (x', x_n)[v]
&= \int_{\Rset^{n - 1}}  \Deriv \psi (x'  - x_n z)[v' - v_n z] \,\eta (z) \dif z\\
&= \frac{1}{x_n{}^{n - 1}}
 \int_{\Rset^{n - 1}}  \Deriv \psi (y)[H \bigl(\tfrac{x' - y}{x_n}\bigr)] \dif y\;,
\end{split}
\end{equation}
where the function \(H \in C^\infty_c (\Rset^{n - 1}, \Lin (\Rset^n, \Rset^{n - 1}))\) is defined for every \(z \in \Rset^{n - 1}\) and \(v = (v', v_n) \in \Rset^{n}\)
by \(H (z)[v] \defeq (v' - v_n z) \eta (z)\).
Hence we deduce from \eqref{eq_ExephahvaegheeJie6reW2me} that for every \(j \in \Nset \setminus \set{0}\) and \(x \in \Rset^n_+\)
\begin{align}
  \label{eq_di4xeex8Yiekakei4Tei1eli}
\raisetag{2em}
  \abs{\Deriv^j \Psi (x', x_n)}
\le 
\frac{1}{x_n{}^{n + j - 2}}
\int_{B_{x_n} (x)}\hspace{-.8em} \abs{\Deriv \psi(y)} \abs{\Deriv^{j - 1} H(\tfrac{x' - y}{x_n})} \dif y \le \frac{\C \norm{\Deriv \psi}_{L^\infty (\Rset^{n- 1})}}{x_n{}^{j - 1}}\;.
\end{align}
Moreover, for every point \(y\in \Rset^{n - 1}\), we have 
\begin{equation*}
\begin{split}
\lim_{(x', x_n) \to (y, 0)} 
\Psi (x', x_n) & = \lim_{(x', x_n) \to (y, 0)} \int_{\Rset^{n - 1}} \psi (x' - x_n z) \, \eta (x) \dif z \\
&=
  \int_{\Rset^{n - 1}} \psi (y) \, \eta (z) \dif z = \psi (y)\;,
  \end{split}
\end{equation*}
by Lebesgue's dominated convergence theorem, since the function \(\psi\) is continuous and since \(\int_{\Rset^{n - 1}} \eta = 1\).

We fix a function \(\theta \in C^k (\intvo{0}{+\infty}, \Rset)\) such that \(\theta = 1\) on \(\intvl{0}{1}\) and \(\theta = 0\) on \(\intvr{2}{+\infty}\)
and we define for each \(\lambda \in \intvo{0}{+\infty}\) the function \(\theta^k_\lambda : \Rset^n_+ \to \Rset\) for every \((x', x_n) \in \Rset^{n - 1} \times \intvo{0}{+\infty}\) by 
\begin{equation}
\label{eq_GaechahLae3wai4ooYaik7cu}
\theta^k_\lambda (x', x_n) \defeq \frac{x_n{}^{k - 1}}{(k - 1) !} \theta (x_n/\lambda)\; .
\end{equation}
By \(k\) successive integration by parts we have, since the function \(\theta^k_\lambda \Psi\) is compactly supported in \(\Bar{\Rset}^n_+\), 
\begin{multline}
  \label{eq_uf4Koo1wahhic8aiBoo}
\int_{\Rset^{n}_+} \dualprod{\partial_n^k L(\nu)[f]}{\theta^k_\lambda \Psi}
- (-1)^{k} \dualprod{L(\nu)[f]}{\partial_n^k (\theta^k_\lambda \Psi)}\\
  = - \sum_{j = 0}^{k - 1}
  (-1)^j
  \int_{\Rset^{n - 1} \times \set{0}} \dualprod{\partial_n^{k - 1 - j} L(\nu)[f]}{\partial_n^j (\theta^k_\lambda \Psi)}\;.
\end{multline}
The general Leibniz rule yields for each \(j \in \set{0, \dotsc, k}\) that 
\begin{equation}
  \label{eq_aiZootheiX1Dooche7ooqua7}
\partial_n^j (\theta^k_\lambda \Psi)
= \sum_{i = 0}^j {\textstyle \binom{j}{i}} (\partial_n^i \theta^k_\lambda) (\partial_n^{j - i} \Psi)\;.
\end{equation}
Since \(\smash{\partial_n^i \theta^k_\lambda} = 0\) on \(\Rset^{n - 1} \times \set{0}\) if \(i \in \set{0, \dotsc, k - 2}\) 
and \(\smash{\partial_n^{k - 1} \theta^k_\lambda} = 1 \) on \(\Rset^{n - 1} \times \set{0}\), we have by \eqref{eq_aiZootheiX1Dooche7ooqua7}
\(\smash{
\partial_n^j (\theta^k_\lambda \Psi)} = 0
\) on \(\smash{\Rset^{n - 1}} \times \set{0}\) when \(j \in \set{1, \dotsc,  k - 2}\) whereas \(
\partial_n^{k - 1} (\theta^k_\lambda \Psi) = \psi\) on \(\Rset^{n - 1} \times \set{0}\). 
Hence, the right-hand side in \eqref{eq_uf4Koo1wahhic8aiBoo} reduces to
\begin{equation}
  \sum_{j = 0}^{k - 1}
  (-1)^j
  \int_{\Rset^{n - 1} \times \set{0}} \hspace{-2em}\dualprod{\partial_n^{k - j} L (\nu)[f]}{\partial_n^j (\theta^k_\lambda \Psi)}
  = (-1)^{k - 1} \int_{\Rset^{n - 1} \times \set{0}} \dualprod{L(\nu)[f]}{\psi}\;,
\end{equation}
and hence by the identity \eqref{eq_uf4Koo1wahhic8aiBoo}, we have
\begin{equation}
  \label{eq_viifuBaibee7Eephiyais7ph}
  \int_{\Rset^{n - 1} \times \set{0}} \dualprod{L(\nu)[f]}{\psi}
  = \int_{\Rset^{n}_+} (-1)^{k}\dualprod{\partial_n^k L(\nu)[f]}{\theta^k_\lambda \Psi} - \dualprod{L(\nu)[f]}{\partial_n^k (\theta^k_\lambda \Psi)}\;.
\end{equation}

We observe that for every \(j \in \set{0, \dotsc, k - 1}\) and every \((x', x_n) \in \Rset^n_+\), in view of \eqref{eq_GaechahLae3wai4ooYaik7cu} and of the general Leibniz rule
\begin{equation*}
\partial_n^j 
\theta^k_\lambda (x', x_n)
= \sum_{i = 0}^j {\textstyle \binom{j}{i}} \frac{x_n{}^{k - 1 - (j - i)} \theta^{(i)} (x_n/\lambda)}{(k - 1 - (j - i))! \lambda^i}\;,
\end{equation*}
so that, since for every \(i \in \Nset\) and \(t \in \intvo{0}{+\infty}\), \(\abs{\theta^{(i)}(t)} \le \C/t^i\),
\begin{equation}
  \label{eq_ooriey4iishuK4ohjeePh4ae}
\abs{\partial_n^j 
  \theta^k_\lambda (x', x_n) } \le \C \, x_n{}^{k - 1 - j}
\end{equation}
whereas 
\begin{equation*}
\partial_n^k 
\theta^k_\lambda (x', x_n)
= \sum_{i = 1}^k {\textstyle \binom{k}{i}} \frac{x_n{}^{i - 1} \theta^{(i)} (x_n/\lambda)}{(i - 1)! \lambda^i}\;,
\end{equation*}
and thus since for every \(i \in \Nset \setminus \set{0}\) and \(t \in \intvo{0}{+\infty}\), \(\abs{\theta^{(i)}(t)} \le \C/t^{i - 1}\),
\begin{equation}
  \label{eq_beiph8dah2seim7CiDongeig}
\abs{
  \partial_n^k 
  \theta^k_\lambda (x', x_n)
} \le \frac{\C}{\lambda}\;.
\end{equation}
For every \(j \in \set{0, \dotsc, k -  1}\), we have by \eqref{eq_di4xeex8Yiekakei4Tei1eli} and \eqref{eq_ooriey4iishuK4ohjeePh4ae}
\begin{equation}
  \label{eq_xiPh0phupahyoxaiT3jo9lei}
\abs{\partial_n^j \theta^k_\lambda (x', x_n) \partial_n^{k - j} \Psi (x', x_n)}
\le \C \norm{\Deriv \psi}_{L^\infty(\Rset^{n - 1})}\;,
\end{equation}
and by \eqref{eq_yae4Ahp6eir9reith9xahtae} and \eqref{eq_beiph8dah2seim7CiDongeig}
\begin{equation}
\label{eq_eiJungeix7ooph3eima8shoh}
\abs{\partial_n^k \theta^k_\lambda (x', x_n)  \Psi (x', x_n)}
\le \C \frac{\norm{\psi}_{L^\infty(\Rset^{n - 1})}}{\lambda}\;,
\end{equation}
and hence by \eqref{eq_aiZootheiX1Dooche7ooqua7}, \eqref{eq_xiPh0phupahyoxaiT3jo9lei} and \eqref{eq_eiJungeix7ooph3eima8shoh}
\begin{equation}
  \label{eq_laiYei5ohphe9Ri3rohceequ}
\abs{\partial_n^k (\theta^k_\lambda \Psi)}
\le 
\C \brk[\Big]{\norm{\Deriv \psi}_{L^\infty(\Rset^{n - 1})} + \frac{\norm{\psi}_{L^\infty(\Rset^{n - 1})}}{\lambda}}\;,
\end{equation}
so that 
\begin{equation}
  \label{eq_HaiHahzou6ha6Hai5ay}
  \begin{split}
\abs[\Big]{
  \int_{\Rset^{n}_+} \dualprod{L(\nu)[f]}{\partial_n^k (\theta^k_\lambda \Psi)}\,}&\\
\le 
\C &\brk[\Big]{\norm{\Deriv \psi}_{L^\infty(\Rset^{n - 1})} + \frac{\norm{\psi}_{L^\infty(\Rset^{n - 1})}}{\lambda}}
\int_{\Rset^n_+} \abs{L (\nu)[f]}\;.
\end{split}
\end{equation}

Finally, we write for every \(\xi =( \xi', \xi_n) \in \Rset^n\) 
\begin{equation}
\label{eq_lah1mahphahJohleeThathai}
L(\xi) = \sum_{j = 0}^{k} \xi_n^{k - j} L_j (\xi')\;,
\end{equation}
where for every \(j \in \set{0, \dotsc, k}\), \(L_j(\Deriv')\) is a homogeneous linear differential operator on \(\Rset^{n - 1}\) of degree \(j\) from \(E\) to \(F\).
Since \(L (\Deriv) f = 0\) by assumption and since \(L (\nu) = L_0 (\xi)\), we have by \eqref{eq_lah1mahphahJohleeThathai} 
\begin{equation}
  \label{eq_kaisiesequ7yuGhahch}
    \int_{\Rset^{n}_+} \dualprod{\partial_n^k L(\nu)[f]}{\theta^k_\lambda \Psi}
  =
  -  \sum_{j = 1}^{k } \int_{\Rset^{n}_+} \dualprod{\partial_n^{k - j}  L_j (\Deriv') [f]}{\theta^k_\lambda \Psi}\;.
\end{equation}
By integration by parts, we compute for every \(j \in \set{1, \dotsc, k}\),
\begin{equation}
  \label{eq_ciex5Ahh0aephooKaid}
\int_{\Rset^{n}_+} \dualprod{\partial_n^{k - j}  L_j (\Deriv') [f]}{\theta^k_\lambda \Psi}
= (-1)^{k} \int_{\Rset^{n}_+} \dualprod{f}{L_j^* (\Deriv')\partial_n^{k - j} ( \theta^k_\lambda \Psi)}\;,
\end{equation}
where 
\(L_j (\xi')^* \in \Lin (E, F)\) is the adjoint to \(L_j (\xi')\).
We compute, for each \(j \in \set{1, \dotsc, k}\), by the general Leibniz rule again
\begin{equation}
\label{eq_Ahhaish2ieC5zeeMohyeelae}
\begin{split}
L_j (\Deriv')^*
\partial_n^{k - j} ( \theta^k_\lambda \Psi)
&= 
\partial_n^{k - j}  (\theta^k_\lambda L_j (\Deriv')^* \Psi)\\
&= \sum_{i= 0}^{k - j} \tbinom{k - j}{i} (\partial_n^i \theta^k_\lambda) (\partial_n^{k - j - i} L_j (\Deriv')^* \Psi)\;,
\end{split}
\end{equation}
so that if \(j \in \set{1, \dotsc, k}\), we have by \eqref{eq_di4xeex8Yiekakei4Tei1eli}, \eqref{eq_ooriey4iishuK4ohjeePh4ae} and \eqref{eq_Ahhaish2ieC5zeeMohyeelae}
\begin{equation}
\label{eq_wai0ioraipooNa3Looghaeth}
\abs{L_j (\Deriv')^*
  \partial_n^{k - j} ( \theta^k_\lambda \Psi)}
\le \C \norm{\Deriv \psi}_{L^\infty(\Rset^{n - 1})}\;,
\end{equation}
and thus by \eqref{eq_kaisiesequ7yuGhahch}, \eqref{eq_ciex5Ahh0aephooKaid} and \eqref{eq_wai0ioraipooNa3Looghaeth},
\begin{equation}
  \label{eq_ci0gahJahrij1ohphei}
\abs[\Big]{
  \int_{\Rset^{n}_+} \dualprod{\partial_n^k L(\nu)[f]}{\theta^k_\lambda \Psi}\,
}
\le \C\, \norm{\Deriv \psi}_{L^\infty(\Rset^{n -1})} \int_{\Rset^n_+} \abs{f}\;.
\end{equation}

In order to conclude we combine the identity \eqref{eq_viifuBaibee7Eephiyais7ph} and the inequalities \eqref{eq_HaiHahzou6ha6Hai5ay} and \eqref{eq_ci0gahJahrij1ohphei}  to get
\begin{equation}
  \abs[\Big]{
  \int_{\Rset^{n - 1}\times \set{0}} \dualprod{L(\nu)[f]}{\psi}\,}
\le 
   \C \,
   \brk[\Big]{\norm{\Deriv \psi}_{L^\infty (\Rset^{n - 1})}  + \frac{\norm{\psi}_{L^\infty (\Rset^{n - 1})}}{\lambda}} \int_{\Rset^n_+} \abs{f}\;;
\end{equation}
letting \(\lambda \to +\infty\) we obtain \eqref{eq_ahkaek3ceib6OgaeSh4yuchi}.
\end{proof}

\subsubsection{Completing the Proof of the Duality Estimate}

Our main step towards the proof of \cref{theorem_cocancelling} is the following estimate on a component of \(f\).

\begin{lemma}[Unidirectional cocancelling estimate]
  \label{lemma_cocancelling_estimate_Lxi}  
  Let \(n \in \Nset\setminus \set{0}\) and let \(E\) and \(F\) be finite-dimensional vector spaces.
  If \(L (\Deriv)\) is a homogenous constant coefficients differential operator from \(E\) to \(F\) on \(\Rset^n\) of order \(k \in \Nset \setminus \set{0}\),
  then there exists a constant \(C \in \intvo{0}{+\infty}\) such that for every \(f \in C^\infty (\Rset^n, E) \cap L^1 (\Rset^n, E)\) satisfying \(L (\Deriv)f = 0\), for every \(\varphi \in C^\infty (\Rset^n, F)\) and for every \(\xi \in \Rset^n\),
  \begin{equation}
  \label{eq_dahchecheir0vahThaiv8aih}
    \abs[\Big]{
      \int_{\Rset^n} \dualprod{\varphi}{L (\xi)[f]}\,
    }
    \le 
    C \, \abs{\xi}
    \int_{\Rset^n} \abs{f}\;
    \brk[\Big]{\int_{\Rset^n} \abs{\Deriv \varphi}^n}^\frac{1}{n}\;.
  \end{equation}
\end{lemma}

\begin{proof}
Without loss of generality, we assume that \(\xi = (0, \dotsc, 0, 1)\).
By Fubini's theorem, we have 
\begin{equation}
  \label{eq_Aer5ugh2pipohthie}
  \int_{\Rset^n} \dualprod{L (\xi)[f]}{\varphi}
  = \int_{\Rset} \int_{\Rset^{n - 1}} \dualprod{L (\xi) [f (x',x_n)]}{\varphi (x', x_n)}  \dif x' \dif x_n\;.
\end{equation}
We observe that for every \(x_n \in \Rset\) and every \(\psi \in C^1_c (\Rset^{n - 1}, F)\), we have immediately
\begin{equation}
  \label{eq_uwo2chaiHu0sie4aequohphu}
  \abs[\Big]{\int_{\Rset^{n - 1}}\dualprod{L (\xi) [f (x',x_n)]}{\psi (x', x_n)}  \dif x'\,}
  \le \C\,
  \norm{\psi}_{L^\infty (\Rset^{n - 1})} \int_{\Rset^{n - 1}} \abs{f (\cdot, x_n)}\;,
\end{equation}
and, by \cref{lemma_parts_general} we have 
\begin{equation}
  \label{eq_eesh0eedupuimeimie3Ahshu}
  \begin{split}
    \abs[\Big]{\int_{\Rset^{n - 1}} \dualprod{L (\xi)[f (x',x_n)]}{\psi (x', x_n)} \dif x'\,}
&\le 
\C \,
\norm{\Deriv \psi}_{L^\infty (\Rset^{n - 1})} \int_{\Rset^{n}} \abs{f}\;.
\end{split}
\end{equation}
By the interpolation into Hölder spaces (\cref{lemma_Holder_interpolation}) with \(\smash\alpha = \frac{1}{n}\) and then by the Morrey-Sobolev embedding (\cref{proposition_Morrey_Sobolev}), we deduce from \eqref{eq_uwo2chaiHu0sie4aequohphu} and \eqref{eq_eesh0eedupuimeimie3Ahshu} that 
\begin{equation}
  \label{eq_Upo6xaigohwa4cu4a}
\begin{split}
  \abs[\Big]{\int_{\Rset^{n - 1}}&\dualprod{L (\xi)[f (x',x_n)]}{\psi (x', x_n)} \dif x'}\\
&\le 
\C\, \seminorm{\psi}_{C^{0, 1/n} (\Rset^{n - 1})} \brk[\Big]{\int_{\Rset^{n - 1}} \abs{f (\cdot, x_n)} }^{1 - \frac{1}{n}}\brk[\Big]{\int_{\Rset^{n}} \abs{f}}^\frac{1}{n}\\
&
\le \C \brk[\Big]{\int_{\Rset^{n - 1}} \abs{\Deriv \psi}^n}^\frac{1}{n} 
  \brk[\Big]{\int_{\Rset^{n - 1}} \abs{f (\cdot, x_n)} }^{1 - \frac{1}{n}}
  \brk[\Big]{\int_{\Rset^{n}} \abs{f}\,}^\frac{1}{n}\;.
\end{split}
\end{equation}

We now deduce  from \eqref{eq_Aer5ugh2pipohthie} and \eqref{eq_Upo6xaigohwa4cu4a} 
and by Hölder's inequality that 
\begin{equation}
\label{eq_angu7aexuu0ahwaino7Einge}
  \begin{split}
  \abs[\Big]{&\int_{\Rset^n} \dualprod{L (\xi)[f]}{\varphi}\,}\\
  &\le \Cl{cst_Aith3ohCaipeel7Xa}
  \int_{\Rset} \brk[\Big]{\int_{\Rset^{n - 1}} \abs{\Deriv \varphi (\cdot, x_n)}^n }^\frac{1}{n} 
  \brk[\Big]{\int_{\Rset^{n - 1}} \abs{f (\cdot, x_n)} }^{1 - \frac{1}{n}}
  \brk[\Big]{\int_{\Rset^{n}} \abs{f}}^\frac{1}{n} \dif x_n\\
  &\le \Cr{cst_Aith3ohCaipeel7Xa} 
  \brk[\Big]{\int_{\Rset}\brk[\Big]{\int_{\Rset^{n - 1}} \abs{\Deriv \varphi (\cdot, x_n)}^n} \dif x_n }^\frac{1}{n}  \brk[\Big]{\int_{\Rset} \Bigl(\int_{\Rset^{n - 1}} \abs{f (\cdot, x_n)}}\dif x_n \Bigr)^{1 - \frac{1}{n}}
  \brk[\Big]{\int_{\Rset^{n}} \abs{f}}^\frac{1}{n}\;,\\
%
  \end{split}
  \raisetag{5em}
\end{equation}
which proves \eqref{eq_dahchecheir0vahThaiv8aih}.
\end{proof}

Our last step to prove \cref{theorem_cocancelling} is the following algebraic lemma.

\begin{lemma}[Intersecting kernel quotient]
  \label{lemma_decomposition_L_xi_j} 
  Let \(n \in \Nset \setminus \set{0}\), let \(E\), \(F\) and \(W\) be finite-dimensional vector spaces, let \(L (\Deriv)\) be a homogeneous constant coefficient differential operator from \(E\) to \(F\) on \(\Rset^n\), and let \(Q \in \Lin (E, W)\).
If 
\[
\bigcap_{\xi \in \Rset^n \setminus \set{0}} \ker L (\xi) \subseteq \ker Q\;,
\]
then there exist \(m \in \Nset\), 
  \(\xi_1, \dotsc, \xi_m \in \Rset^n \setminus \set{0}\), and \(Q_1, \dotsc, Q_m \in \Lin (F, W)\) such that 
\begin{equation}
\label{eq_ax1Zeiteecei0Ce3uk7thoo4}
  Q 
  = 
  \sum_{j = 1}^m Q_j \compose L (\xi_j)\;.
\end{equation}
\end{lemma}
\begin{proof}
  Since the space \(E\) has finite dimension, there exists a nonnegative integer \(m \in \Nset\) and vectors \(\xi_1, \dotsc, \xi_m \in \Rset^n \setminus \set{0}\)
  such that 
\begin{equation*}
  \bigcap_{\xi \in \Rset^n \setminus \set{0}} \ker L (\xi)
  = \bigcap_{1 \le i \le m} \ker L (\xi_i)\;.
\end{equation*}
If we define \(R \in \Lin(E, F^m)\) for each \(e \in E\) by \(R(e) \defeq (L (\xi_1)[e], \dotsc, L (\xi_m)[e]) \in F^m\),
we have \(\ker R \subseteq \ker Q\) and there exists thus  \(Q_1, \dotsc, Q_m \in \Lin (F, W)\) such that \eqref{eq_ax1Zeiteecei0Ce3uk7thoo4} holds.
\end{proof}

\begin{proof}[Proof of the sufficiency part of \cref{theorem_cocancelling}]
By \cref{lemma_decomposition_L_xi_j} with \(Q\) being taken to be the identity \(\operatorname{id}_E\), there exists \(m \in \Nset\), 
\(\xi_1, \dotsc, \xi_m \in \Rset^n \setminus \set{0}\) and \(Q_1, \dotsc, Q_m \in \Lin (F, E)\) such that 
\(\smash{
    \operatorname{id}_E
  = 
    \sum_{j = 1}^m Q_j \compose L (\xi_j)}
\).
We have then, 
\[
\int_{\Rset^n} \dualprod{f}{\varphi} = \sum_{j = 1}^m \int_{\Rset^n} \dualprod{L (\xi_j)[f]}{Q_j^*[\varphi]}\;,
\]
and the conclusion \eqref{eq_koo4lui9Eipo7XeenaeghaoR} follows then from \cref{lemma_cocancelling_estimate_Lxi}.
\end{proof}

\begin{remark}
The estimate in \cref{theorem_cocancelling} can be proved by replacing the assumption that \(L (\Deriv)\) is cocancelling by the assumption that \(\varphi (x) \in \bigl(\bigcap_{\xi \in \Rset^n \setminus\set{0}} \ker L (\xi))^{\perp}\).
In order to prove this one follows the proof above, taking \(P : E \to E\) to be a projection on the space \(\smash{(\bigcap_{\xi \in \Rset^n \setminus \set{0}} \ker L (\xi))^\perp}\) so that \(\smash{\ker P^* = P[E]^\perp = \bigcap_{\xi \in \Rset^n \setminus \set{0}} \ker L (\xi)}\)
and one can then  write \(\smash{P^*
  = 
    \sum_{j = 1}^m Q_j \compose L (\xi_j)}\) for some  \(Q_1, \dotsc, Q_m \in \Lin (F, E)\) thanks to \cref{lemma_decomposition_L_xi_j}.
\end{remark}

Although the proof of \cref{theorem_cocancelling} given here relies heavily on the Euclidean structure of the underlying space \(\Rset^n\) through the application of the Fubini theorem in \eqref{eq_Aer5ugh2pipohthie} and \eqref{eq_angu7aexuu0ahwaino7Einge} and the decomposition of the vector space \eqref{eq_ax1Zeiteecei0Ce3uk7thoo4}, suitable integral-geometric formulae allow to obtain similar results for the hyperbolic plane \cite{Chanillo_VanSchaftingen_Yung_2017_Variations} and symmetric spaces of noncompact type \cite{Chanillo_VanSchaftingen_Yung_2017_Symmetric}.

\subsection{Back to Cancelling Operators}
We apply now \cref{theorem_cocancelling} to get the endpoint Sobolev estimate for cancelling operators.

\begin{lemma}[Singular integral representation of order \(-1\)]
\label{lemma_representation_Dk_1uF}
Under the assumptions of \cref{proposition_fundamental_solution} and with \(G_A : \Rset^n \setminus \set{0} \to \Lin (E, V)\) given by the same proposition, for every \(u \in C^\infty_c(\Rset^n, V)\), if \(A(\Deriv u) = \operatorname{div} F\), with \(F \in C^\infty (\Rset^n, \Lin(\Rset^n, E))\) and if  
\(
\int_{\Rset^n} \abs{F (x)}/(1 + \abs{x}^n) \dif x < \infty
\),
then for every \(r \in \intvo{0}{+\infty}\), 
\begin{equation}
\label{eq_EeNiechiquuajaeNgahh4aer}
\begin{split}
 \Deriv^{k - 1} u (x) &= \int_{\Sset^{n - 1}} \Deriv^{k - 1} G_A (z)\, [F (x)[z]] \dif z + \int_{\Rset^n \setminus B_r(x)} \Deriv^{k} G_A (x - y)[F (y)] \dif y\\\
 & \qquad  + \int_{B_r(x)} \Deriv^{k} G_A (x - y)[F (y) - F(x)] \dif y\\
 &= \int_{\Sset^{n - 1}} \Deriv^{k - 1} G_A (z)\, [F (x)[z]] \dif z + \lim_{\varepsilon \to 0} \int_{\Rset^n \setminus B_\varepsilon(0)} \Deriv^{k} G_A (h)[F (x - h)] \dif h\;.
 \end{split}
\end{equation}
\end{lemma}

\begin{proof}
The first \(k -1\) derivatives of \(G_A\) are locally integrable in view of \cref{proposition_fundamental_solution} \eqref{it_oocee4ao4DaiK8sheWeechee} and thus differentiating \(k-1\) times \eqref{eq_jal9Niez6Vu2ohved5OorooG}, we get that \eqref{eq_Xoo6theoxuch2iesoos2wini} holds for every \(x \in \Rset^n\).
We fix a function \(\eta \in C^\infty_c (\Rset^n, \Rset)\) such that \(\eta = 1\) on \(B_{1/2}(0)\) and \(\eta = 0\) on \(\Rset^n \setminus B_1(0)\). If \(R > 0\) is large enough so that \(\supp u \subseteq B_R(0)\), we deduce from \eqref{eq_Xoo6theoxuch2iesoos2wini} that for every \(\varepsilon \in \intvo{0}{R/2}\),
\begin{equation}
  \label{eq_ieD7le3lah3ut1Aitei7aiza}
\begin{split}
 \Deriv^{k - 1} u (x) &=
 \int_{\Rset^n} \eta\brk[\big]{\tfrac{\abs{x - y}}{R}} \brk[\big]{1- \eta\brk[\big]{\tfrac{\abs{x - y}}{\varepsilon}}}\Deriv^{k - 1} G_A (x - y)[\operatorname{div} \brk{F - F  \brk{x}} (y)] \dif y\\
 &\quad +  \int_{\Rset^n} \eta\brk[\big]{\tfrac{\abs{h}}{\varepsilon}} \Deriv^{k - 1} G_A (h)[ A (\Deriv) u (x - h) ] \dif h\;.
 \end{split}
 \raisetag{2em}
 \end{equation}
Integrating by parts in the first integral on the right-hand side of \eqref{eq_ieD7le3lah3ut1Aitei7aiza}, we get
\begin{equation}
\begin{split}
 &\int_{\Rset^n} \eta\brk[\big]{\tfrac{\abs{x - y}}{R}} \brk[\big]{1- \eta\brk[\big]{\tfrac{\abs{x - y}}{\varepsilon}}}\Deriv^{k - 1} G_A (x - y)[\operatorname{div} \brk{F - F \brk{x}} (y)] \dif y\\
 &
 = \int_{\Rset^n} \eta\brk[\big]{\tfrac{\abs{x - y}}{R}} \brk[\big]{1- \eta\brk[\big]{\tfrac{\abs{x - y}}{\varepsilon}}}\Deriv^{k} G_A (x - y)[ (F - F (x))(y)] \dif y\\
 & \quad+ \int_{\Rset^n} \brk[\Big]{\tfrac{1}{R} \eta'\brk[\big]{\tfrac{\abs{x - y}}{R}}
 - \tfrac{1}{\varepsilon} \eta'\brk[\big]{\tfrac{\abs{x - y}}{\varepsilon}}}
 \Deriv^{k - 1} G_A (x - y)\, [(F(y) - F (x))[\tfrac{x - y}{\abs{x - y}}]] \dif y\;.
\end{split}
\raisetag{6em}
\end{equation}
Applying \cref{proposition_fundamental_solution} \eqref{it_Zigh3aiy3joo5maewiechua6}, letting \(\varepsilon \to 0\) and \(R \to \infty\), we reach the conclusion since
\[
\begin{split}
   &\frac{1}{R} \int_{\Rset^n} \eta'\brk[\big]{\tfrac{\abs{x - y}}{R}}
 \Deriv^{k - 1} G_A (x - y)\, [F (x)[\tfrac{x - y}{\abs{x - y}}]] \dif y\\
   & \qquad  = \frac{1}{R} \int_0^\infty \int_{\Sset^{n - 1}}\eta' \brk{\tfrac{\rho}{R}} \Deriv^{k - 1} G_A (z)\, [F (x)[z]] \dif z \dif \rho\\
   &\qquad = - \int_{\Sset^{n - 1}} \Deriv^{k - 1} G_A (z)\, [F (x)[z]] \dif z.
  \end{split}
\]
\end{proof}

\begin{proof}%
[Proof of sufficiency of injective ellipticity and cancellation in \cref{theorem_cancelling_necessary_Sobolev}]%
As the operator \(A (\Deriv)\) is injectively elliptic, by \cref{proposition_compatibility_conditions}, there exists a homogeneous constant coefficient differential operator
from 
\(E\) to \(F\) on \(\Rset^n\) such that for every \(\xi \in \Rset^n \setminus \set{0}\),
\(
    \ker L (\xi)
  = 
    A (\xi)[V] 
\)
and thus by \cref{definition_cancelling} and \cref{definition_cocancelling},  the operator \(L (\Deriv)\) is cocancelling.
By \cref{theorem_cocancelling} and the representation of bounded linear functionals on Sobolev spaces, there exists \(F \in L^\frac{n}{n - 1} (\Rset^n, E \otimes \Rset^n)\) such that 
\(\operatorname{div} F = A (\Deriv) u\) in the sense of distributions
and 
\[
\brk[\Big]{\int_{\Rset^n} \abs{F}^\frac{n}{n - 1}}^{1 - \frac{1}{n}}
\le \C \int_{\Rset^n} \abs{A (\Deriv) u}\;.
\]
Letting \(\eta \in C^\infty_c (\Rset^n, \Rset)\) such that \(\int_{\Rset^n} \eta = 1\) and setting for \(\varepsilon \in \intvo{0}{+\infty}\) and \(x \in \Rset^n\), \(\eta_\varepsilon (x) \defeq\eta (x/\varepsilon)/\varepsilon^n\), we have \(A(\Deriv )(\eta_\varepsilon \ast u) = \operatorname{div} (\eta_\varepsilon \ast F)\), we get by \cref{theorem_singular_integral} and \cref{lemma_representation_Dk_1uF},
\begin{equation}
  \brk[\Big]{\int_{\Rset^n} \abs{\Deriv^{k - 1} (\eta_\varepsilon \ast u)}^\frac{n}{n - 1}}^{1 - \frac{1}{n}}
\le \C \,\brk[\Big]{\int_{\Rset^n} \abs{F}^\frac{n}{n - 1}}^{1 - \frac{1}{n}}
\le 
\C
\int_{\Rset^n} \abs{A (\Deriv)[u]}\;,
\end{equation}
letting \(\varepsilon \to 0\), we get \eqref{eq_eeghaib2eiviu0cib4Pei8ut} when \(\ell = k - 1\); the case \(\ell \in \set{1, \dotsc, k - 2}\) then follows from the classical Sobolev embedding theorem (\cref{theorem_Sobolev}).
\end{proof}

\subsection{Characterization of Cancelling Operators}

As a byproduct of the characterization of cocancelling operators (\cref{proposition_equiv_cocanc}), we can characterize cancelling operators under the assumption of injective ellipticity.

\begin{proposition}[Characterization of cancelling operators]
  \label{propositionEquivalentCancelling}
  Let \(n \in \Nset\setminus \set{0}\), let \(V\) and \(E\) be finite-dimensional vector spaces, and let \(A(\Deriv)\) be a homogeneous differential operator of order \(k \in \Nset \setminus \set{0}\) on \(\Rset^n\) from \(V\) to \(E\).
  If the operator \(A(\Deriv)\) is injectively elliptic, then the following are equivalent
  \begin{enumerate}[(i)]
    \item\label{itCancelling} the operator \(A(\Deriv)\) is cancelling, 
    \item\label{itCancellingVanishL1} for every \(u \in L^1_{\mathrm{loc}}(\Rset^n, V)\) such that 
    \(A(\Deriv)u \in L^1(\Rset^n, E)\), one has 
        \begin{equation}
    \label{eq_vumeDamool8Maechu1eiz0co}
      \int_{\Rset^n} A(\Deriv)u = 0\;,
    \end{equation}
    \item\label{itCancellingVanishCinfty} 
    for every \(u \in C^\infty(\Rset^n, V)\) such that 
    \(\supp A(\Deriv)u\) is compact, \eqref{eq_vumeDamool8Maechu1eiz0co} holds,
    \item \label{itCancellingDelta} for every distribution \(u \in C^\infty_c (\Rset^n, V)^*\) and every \(e \in E\) such that \(A (\Deriv) u= e\delta_0\), one has 
    \(e = 0\).
  \end{enumerate}
\end{proposition}

It is crucial that no decay assumption is imposed on \(u\) in \eqref{itCancellingVanishCinfty}.
Indeed, if for every \(j \in \{0, \dotsc, k - 1\}\), we assume that 
\(\lim_{\abs{x} \to \infty} \abs{\Deriv^j u (x)} \abs{x}^{n - j} = 0 \), then one has 
\(
\int_{\Rset^n} A(\Deriv) u = 0
\),
for any constant coefficient differential operator \(A(\Deriv)\) of order \(k \in \Nset \setminus \set{0}\).

\begin{proof}[Proof of \cref{propositionEquivalentCancelling}]
  We first note that since \(A(\Deriv)\) is injectively elliptic, 
  proposition~\ref{proposition_compatibility_conditions} applies and yields a homogeneous differential operator \(L(\Deriv)\) on \(\Rset^n\) from \(E\) to \(F\)
  We first note that since \(A(\Deriv)\) is injectively elliptic, proposition~\ref{proposition_compatibility_conditions} applies and yields a homogeneous differential operator \(L(\Deriv)\) on \(\Rset^n\) from \(E\) to \(F\)
  such that for every \(\xi \in \Rset^n \setminus \set{0}\), we have \(\ker L (\xi) = A (\xi)[V]\) and \(L(\Deriv)\) is cocancelling.
  
We first prove that \eqref{itCancelling} implies \eqref{itCancellingVanishL1}. 
We assume that \(u \in \smash{L^1_{\mathrm{loc}}(\Rset^n, E)}\) and \(A(\Deriv)u \in  \smash{L^1}(\Rset^n, E)\). By construction of \(L(\Deriv)\), we have 
  \( L(\Deriv) (A(\Deriv)u)=0\).
  Since \(A(\Deriv)\) is cancelling, the operator \(L(\Deriv)\) is cocancelling, \eqref{eq_vumeDamool8Maechu1eiz0co} holds in view of proposition~\ref{proposition_equiv_cocanc} \eqref{itVanishL1}.
  
  It is clear that \eqref{itCancellingVanishL1} implies \eqref{itCancellingVanishCinfty}. Assume now that assertion \eqref{itCancellingVanishCinfty} holds and that \(A(\Deriv) u = e \delta_0\). 
  If \(\eta \in C^\infty_c (\Rset^n, \Rset)\) and \(\smash{\int_{\Rset^n} \eta = 1}\), then 
  \(
   A(\Deriv) (\eta \ast u) = e \eta\),
 and thus by assertion \eqref{itCancellingVanishCinfty}, \(e = \smash{\int_{\Rset^n} A(\Deriv) (\eta \ast u) = 0}\) so that \eqref{itCancellingDelta} holds.
  
 Finally we assume that \eqref{itCancellingDelta} holds and that  \(e \in \bigcap_{\xi \in \Rset^n \setminus \set{0}} A (\xi)[V]\). 
 Since the operator \(A (\Deriv)\) is injectively elliptic, then by \cref{proposition_fundamental_solution} \eqref{it_Aix3Ahshe4ahchiekai4Ieph}
 \(A (\Deriv)[G_A [e]] = e \delta_0\), where \(G_A\) is the representation kernel given by \cref{proposition_fundamental_solution}, 
 and thus by \eqref{itCancellingDelta} we have \(e = 0\) so that assertion \eqref{itCancelling} holds.
\end{proof}

\section{Related Questions and Variants}
\label{section_variants}

\subsection{Cocancelling estimates and BMO}

The cocancelling duality estimates of \cref{proposition_equiv_cocanc} can be seen as a replacement for the failure of the endpoint Sobolev embedding of \(\dot{W}^{1, n} (\Rset^n, \Rset)\) into \(L^\infty (\Rset^n, \Rset)\).
On the other hand \(\smash{\dot{W}^{1, n}} (\Rset^n, \Rset)\) is known to be embedded into the space of functions of bounded mean oscillation \(\mathrm{BMO} (\Rset^n,\Rset)\). 
It turns out that cocancelling estimates capture in general a stronger property than vanishing mean oscillation. In order to state this precisely, we define the following semi-norm associated to an operator.

\begin{definition}[Cocancellation spaces%
\footnote{%
These space generalize the spaces \(\mathrm{D}_k (\Rset^n)\) defined in \cite{VanSchaftingen_2006} which correspond to the case where \(L (\Deriv)\) is the exterior differential acting on \(\smash{C^\infty_c (\Rset^n, \bigwedge^k \Rset^n)}\).%
}]
Let \(n \in \Nset \setminus \set{0}\), let \(E\) and \(F\) be finite-dimensional vector spaces, and let \(L (\Deriv)\) be a homogeneous constant coefficient differential operator from \(E\) to \(F\) on \(\Rset^n\).
We define for every \(\varphi \in C^\infty_c (\Rset^n, \Rset)\), the semi-norm
\begin{equation*}
\norm{\varphi}_{\Deriv_L (\Rset^n)} 
\defeq
\sup 
\,
\set[\bigg]{\abs[\Big]{
  \int_{\Rset^n} \varphi f\,} \st 
f \in C^\infty (\Rset^n, E),\ 
L (\Deriv) f = 0 \text{ and } \int_{\Rset^n} \abs{f} \le 1 }\;.
\end{equation*}
\end{definition}

This semi-norm turns out to be stronger than the bounded mean oscillation seminorm.

\begin{proposition}[Cocancelling estimates imply bounded mean oscillation%
\footnote{%
\Cref{proposition_cocancelling_bmo} generalizes the author's result when the operator \(L (\Deriv)\) is the exterior differential \(d\) \cite{VanSchaftingen_2006}.%
}]
  \label{proposition_cocancelling_bmo}
  Let \(n \in \Nset \setminus \set{0}\), let \(E\) and \(F\) be finite-dimensional vector spaces, and let \(L (\Deriv)\) be a homogeneous constant coefficient differential operator from \(E\) to \(F\) on \(\Rset^n\).
If there exists a finite-dimensional vector space \(V\) and a homogeneous constant coefficients differential operator \(A (\Deriv)\) from \(V\) to \(E\) on \(\Rset^n\)
such that for every \(\xi \in \Rset^n \setminus \set{0}\), one has \(A (\xi) \ne 0\) and \(L (\xi)\compose A(\xi) = 0\),
then there exists a constant \(C \in \intvo{0}{+\infty}\) such that for every \(\varphi \in C^\infty_c (\Rset^n, \Rset)\), one has
\begin{equation*}
\norm{\varphi}_{\mathrm{BMO} (\Rset^n)}
\le 
C \norm{\varphi}_{\Deriv_L (\Rset^n)}\;.
\end{equation*}
\end{proposition}

The assumption of \cref{proposition_cocancelling_bmo} is always satisfied when the operator \(L (\Deriv)\) is given by \cref{proposition_compatibility_conditions} with \(A (\Deriv)\).
The assumption is also satisfied when \(L (\xi)\) is surjective for every \(\xi \in \Rset^n \setminus \set{0}\): indeed, similarly to \eqref{eq_xooch1Ocui9jechohnaeFo2I} in the proof of \cref{proposition_compatibility_conditions} one can define then the operator \(A (\Deriv)\) by setting for each \(\xi \in \Rset^n \setminus \set{0}\)
\begin{equation}
 A (\xi) \defeq \det \bigl(L (\xi) \compose L (\xi)^*\bigr)\,
 \brk[\big]{\operatorname{id}_E\, -\, L (\xi)^* \compose (L (\xi) \compose L (\xi)^*)^{-1} \compose L (\xi)}\;.
\end{equation}

\begin{proof}[Proof of \cref{proposition_cocancelling_bmo}]
By assumption, for every \(\xi \in \Rset^n \setminus \set{0}\), we have \(A (\xi) \ne 0\) and thus \(\trace (A (\xi) \compose A (\xi)^*) > 0\).
By the theory of multipliers on the real Hardy space \(\mathcal{H}^1 (\Rset^n)\),\footnote{%
For the theory of multipliers on the Hardy space \(\mathcal{H}^1 (\Rset^n)\), see for example \cite{Stein_1993}*{Ch.\ 3 \S 3 and Ch.\ I \S 6.2.1}
} 
given \(g \in C^\infty (\Rset^n, \Rset)\cap \mathcal{H}^1 (\Rset^n, \Rset)\), we can define \(f \in C^\infty (\Rset^n, E \otimes E)\cap \mathcal{H}^1 (\Rset^n, E \otimes E)\) by the condition that for every \(\xi \in \Rset^n \setminus \set{0}\), 
\begin{equation*}
  \mathcal{F} f (\xi) \defeq \frac{A (\xi) \compose A (\xi)^*}{\trace (A (\xi) \compose A (\xi)^*)} \mathcal{F} g (\xi)\;;
\end{equation*}
the function \(f\) also satisfies the estimate
\begin{equation*}
\begin{split}
 \norm{f}_{L^1 (\Rset^n, E \otimes E)} &\le \C  \norm{f}_{\mathcal{H}^1 (\Rset^n, E \otimes E)}\\
 &\le \C \norm{g}_{\mathcal{H}^1 (\Rset^n, \Rset)}\;.
 \end{split}
\end{equation*}
By construction, we have \(\trace f = g\) on \(\Rset^n\) and \(L (\Deriv) f = 0\) on \(\Rset^n\) and hence
\begin{equation*}
\begin{split}
\abs[\Big]{\int_{\Rset^n} g \varphi\,}
&=
\abs[\Big]{\,
  \trace \brk[\Big]{\int_{\Rset^n} f \varphi}\,}\\
&\le 
\abs[\Big]{
  \int_{\Rset^n} f \varphi }\\
&\le 
\C
\norm{\varphi}_{\Deriv_L (\Rset^n)} \int_{\Rset^n} \abs{f}\\
&\le 
\C \norm{\varphi}_{\Deriv_L (\Rset^n)} \norm{g}_{\mathcal{H}^1 (\Rset^n, \Rset)}\;,
\end{split}
\end{equation*}
and we conclude by the characterization of bounded mean oscillation (BMO) functions 
by duality with the Hardy space \(\mathcal{H}^1 (\Rset^n)\).\footnote{For the duality between \(\mathrm{BMO} (\Rset^n)\) and the real Hardy space \(\mathcal{H}^1 (\Rset^n)\), see for instance \citelist{\cite{Stein_1993}*{Ch. IV \S 1.2}\cite{Fefferman_1971}\cite{Fefferman_Stein_1972}}.}
\end{proof}

\begin{remark}
In the particular case where \(n \ge 2\) and \(L (\Deriv) = \operatorname{\divergence}\), we can take \(V \defeq \Rset^n \otimes \Rset^n\) and 
\(A (\xi)[v] \defeq  v \cdot \xi - v^* \cdot \xi\) in \cref{proposition_cocancelling_bmo}. 
Hence we obtain in the proof \(A (\Deriv)^* \compose A (\Deriv) (u)= \Delta u- \nabla \nabla \cdot u\),
and thus \(f = (1 - \mathcal{R} \otimes \mathcal{R})g\), where \(\mathcal{R}\) is the vector Riesz transform operator.
\end{remark}

\subsection{Fractional Estimates}
\Cref{theorem_cancelling_necessary_Sobolev} can be generalized to Sobolev embeddings into fractional Sobolev spaces \(W^{k - 1 + s, p} (\Rset^n, V)\) with \(s \in \intvo{0}{1}\).

\begin{theorem}[Fractional embedding for cancelling operators%
\footnote{%
\Cref{theorem_cancelling_fractional} is due to the author \cite{VanSchaftingen_2010}*{Th.\ 8.1}; particular cases where obtained in  \citelist{\cite{Mitrea_Mitrea_2009}\cite{VanSchaftingen_2010}}.%
}%
]
\label{theorem_cancelling_fractional}
Let \(n \in \Nset\setminus\set{0}\), let \(V\) and \(E\) be finite-dimensional vector spaces, let \(A (\Deriv)\) be a homogeneous constant coefficient differential operator
of order \(k \in \Nset \setminus \set{0}\) from 
\(V\) to \(E\) on \(\Rset^n\).
If \(A (\Deriv)\) is injectively elliptic, and if \(s \in \intvo{0}{1}\) and \(p \in \intvr{1}{+\infty}\) satisfy \(\frac{1}{p} = 1 - \frac{1 - s}{n}\), then there exists a constant \(C \in \intvo{0}{+\infty}\) such that for every \(u \in C^\infty_c (\Rset^n, V)\), 
\[
\brk[\Big]{\int_{\Rset^n} \int_{\Rset^n}  \frac{\abs{\Deriv^{k - 1} u (y) - \Deriv^{k - 1} u (x)}^p} {\abs{y -x}^{n + sp}} \dif y \dif x}^\frac{1}{p} 
\le 
C
\int_{\Rset^n} \abs{A (\Deriv)[u]}\;,
\]
if and only if the operator \(A (\Deriv)\) is cancelling.
\end{theorem}

In particular, the endpoint fractional Sobolev inequality
\begin{equation}
\label{eq_ahNgaodumahngeMoht0hirai}
\brk[\Big]{\int_{\Rset^n} \int_{\Rset^n}  \frac{\abs{u (y) - u (x)}^p} {\abs{y -x}^{n + sp}} \dif y \dif x}^\frac{1}{p} 
\le 
C
\int_{\Rset^n} \abs{\Deriv u}\;,
\end{equation}
with \(\frac{1}{p} = 1 - \frac{1 - s}{n}\),
holds for every function \(u \in C^\infty_c (\Rset, \Rset)\) if and only if the condition \(n \ge 2\) holds.\footnote{
  The inequality \eqref{eq_ahNgaodumahngeMoht0hirai} can also be obtained as a consequence of the inequality \(
 \smash{\norm{u}_{B^{s, p}_1(\Rset^n)}} \le C \norm{\Deriv u}_{L^1(\Rset^n)}
\)
of Solonnikov 
\citelist{\cite{Solonnikov_1975}*{Th.\ 2}\cite{Kolyada_1993}*{Th.\ 4}}
or as a consequence of the interpolation estimate
\[
\norm{u}_{B^{s, p}_1 (\Rset^n)} \le C \norm{\Deriv u}_{L^1(\Rset^n)}^s \norm{u}_{L^\frac{d}{d-1}(\Rset^n)}^{1-s}
\]
by Albert \familyname{Cohen}, Wolfgang \familyname{Dahmen}, Ingride \familyname{Daubechies} and Ronald \familyname{DeVore} \cite{Cohen_Dahmen_Daubechies_DeVore_2003}*{Th.\ 1.4} (see also Jean \familyname{Bourgain}, Haïm \familyname{Brezis} and Petru \familyname{Mironescu} \cite{Bourgain_Brezis_Mironescu_2004}*{Lem. D.2})
together with the standard embeddings between Besov spaces and the identification between Besov spaces and fractional Sobolev-Slobo\-decki\u \i{} spaces \cite{Triebel_1983}*{2.3.2(5), 2.3.5(3) and 2.5.7(9)}. A counterexample for \(n=1\) is obtained by regularizing a characteristic function \cite{Schmitt_Winkler_2000}.}

The main new tool we need to prove \cref{theorem_cancelling_fractional} is a fractional counterpart of \cref{theorem_cocancelling}.

\begin{theorem}[Fractional cocancelling estimate%
\footnote{%
\Cref{theorem_cocancelling_fractional} is due to the author \cite{VanSchaftingen_2013}, following several particular cases statements \citelist{\cite{Bourgain_Brezis_2004}*{Rem.\ 2}\cite{Bourgain_Brezis_2007}*{Rem. 11}\cite{VanSchaftingen_2004_div}*{Rem.\ 2}\cite{Mitrea_Mitrea_2009}\cite{VanSchaftingen_2010}};
    in some cases a strong version à la Bourgain-Brezis could be obtained relying on the constructions of Pierre \familyname{Bousquet}, Petru \familyname{Mironescu}, Emmanuel \familyname{Russ}, \familyname{Wang} Yi and \familyname{Yung} Po-Lam  \citelist{\cite{Bousquet_Mironescu_Russ_2013}\cite{Bousquet_Russ_Wang_Yung_2019}}.%
  }%
  ]
  \label{theorem_cocancelling_fractional}
  Let \(n \in \Nset \setminus \{0, 1\}\), let \(V\) and \(E\) be finite-dimensional vector spaces, let \(L (\Deriv)\) be a homogeneous constant coefficient differential operator from \(E\) to \(F\) on \(\Rset^n\), and let \(p \in \intvo{n}{+\infty}\). 
  There exists a constant \(C \in \intvo{0}{+\infty}\) such that
for every \(f \in C^\infty (\Rset^n, E)\cap L^1 (\Rset^n, E)\) that satisfies \(L (\Deriv) f = 0\) in the sense of distributions and for every \(\varphi \in C^\infty_c (\Rset^n, E)\), one has
\begin{equation}
  \abs[\Big]{\int_{\Rset^n} \dualprod{f}{\varphi}}
  \le 
  C \int_{\Rset^n} \abs{f} \; \brk[\Big]{\int_{\Rset^n} \int_{\Rset^n} \frac{\abs{\varphi(y) - \varphi (x)}^p}{\abs{y - x}^{2n}}\dif y \dif x}^\frac{1}{p}
\end{equation}
if and only if the operator \(L (\Deriv)\) is cocancelling.
\end{theorem}

The proof of \cref{theorem_cocancelling_fractional} follows the proof of \cref{theorem_cocancelling}.
This approach is possible first because of the following fractional counterpart of the Morrey-Sobolev embedding of \cref{proposition_Morrey_Sobolev}.

\begin{proposition}[Fractional Sobolev-Morrey embedding%
\footnote{%
For statements and proofs of the fractional Sobolev-Morrey embedding of \cref{proposition_Morrey_Sobolev_fractional},
see \citelist{\cite{DiNezza_Palatucci_Valdinoci_2012}*{\S 8}\cite{Taibleson_1964}*{Lem.\ 11}\cite{Triebel_1983}*{\S 2.7.1}\cite{Leoni_2017}*{Th.\ 14.17}}.
}%
]
    \label{proposition_Morrey_Sobolev_fractional}
    Let \(\ell \in \Nset \setminus \set{0}\), let \(V\) be a finite-dimensional vector space, let \(s \in (0, 1)\) and let \(p \in (\ell/s, +\infty)\). 
    There exists a constant \(C \in \intvo{0}{+\infty}\) such that for every 
    \(\varphi \in C^\infty_c (\Rset^\ell, V)\), one has
    \begin{equation*}
      \abs{\varphi (y) - \varphi (x)} \le C \abs{y - x}^{1 - \frac{\ell}{p}}
      \brk[\Big]{\int_{\Rset^\ell} \int_{\Rset^\ell}\frac{\abs{\varphi (y) - \varphi (x)}^p}{\abs{y - x}^{\ell + sp}} \dif y \dif x}^\frac{1}{p}\;.
    \end{equation*}
\end{proposition}

The second ingredient that make the generalization possible is the following counterpart of Fubini's theorem for fractional Sobolev spaces.

\begin{proposition}[Fractional Fubini theorem]
  \label{lemma_fractional_Fubini}
  Let \(n \in \Nset \setminus \set{0}\), let \(\ell \in \set{1, \dotsc, n - 1}\), let \(s \in (0, 1)\) and let \(p \in [1, +\infty)\).
There exists a constant \(C \in \intvo{0}{+\infty}\) such that for every Borel-measurable function \(\varphi : \Rset^n \to \Rset\) and every \(t \in \Rset^{n - \ell}\) one has
\begin{equation}
\label{eq_aleeziexe1zi4Yaechee5Qui}
  \int_{\Rset^\ell} 
    \int_{\Rset^\ell} 
      \frac{\abs{\varphi (y, t) - \varphi (x,  t)}^p}{\abs{y - x}^{\ell + sp}}
      \dif y
    \dif x
  \le 
  C 
  \int_{\Rset^\ell} 
  \int_{\Rset^n} 
  \frac{\abs{\varphi (z) - \varphi (x, t)}^p}{\abs{z - (x, t)}^{n + sp}}
  \dif z
  \dif x
  \;.
\end{equation}
\end{proposition}

\Cref{lemma_fractional_Fubini} and Fubini's theorem imply then that we have the following fractional counterpart of Fubini's theorem 
\[
\int_{\Rset^{n - \ell}}
\int_{\Rset^\ell} 
\int_{\Rset^\ell} 
\frac{\abs{\varphi (y, t) - \varphi (x, t)}^p}{\abs{y - x}^{\ell + sp}}
\dif y
\dif x
\dif t 
\le 
C 
\int_{\Rset^n} 
\int_{\Rset^n} 
\frac{\abs{\varphi (y) - \varphi (x)}^p}{\abs{y - x}^{n + sp}}
\dif y
\dif x\;.
\]

\begin{proof}[Proof of \cref{lemma_fractional_Fubini}]
We have for every \(x, y \in \Rset^\ell\), 
\begin{multline}
  \label{eq_peepaaye9onguPh3}
\abs{\varphi (x, t) - \varphi (y, t)}^p\\
\le 2^{p - 1} 
\brk[\Big]{\fint_{B_{\frac{\abs{y - x}}{2}}(\frac{x + y}{2}, t)} \abs{\varphi - \varphi (x, t)}^p
+ \fint_{B_{\frac{\abs{y - x}}{2}}(\frac{x + y}{2}, t)} \abs{\varphi - \varphi (y, t)}^p}\;.
\end{multline}
By monotonicity of the integral, we have for every \(x \in \Rset^\ell\), 
\begin{equation}
  \label{eq_Eeghug8auQueiv7Aepai3goo}
\begin{split}
\int_{\Rset^\ell} 
\fint_{B_{\frac{\abs{y - x}}{2}}(\frac{x + y}{2}, t)}&
    \frac
      {\abs{\varphi (z) - \varphi (x, t)}^p}
      {\abs{y - x}^{n + sp}} \dif z \dif y\\
  &\le 
    \Cl{cst_uahah3phahT1Yimu}
      \int_{\Rset^\ell} 
        \int_{B_{\abs{y - x}}(x, t)} 
          \frac
            {\abs{\varphi (z) - \varphi (x, t)}^p}
            {\abs{y - x}^{n + \ell + sp}} 
          \dif z 
        \dif y
\end{split}
\end{equation}
and thus exchanging the integrals we deduce from \eqref{eq_Eeghug8auQueiv7Aepai3goo} that 
\begin{equation}
  \label{eq_iehieb9Imiaziezo}
\begin{split}
\int_{\Rset^\ell} 
\fint_{B_{\frac{\abs{y - x}}{2}}(\frac{x + y}{2}, t)}&
    \frac
      {\abs{\varphi (z) - \varphi (x, t)}^p}
      {\abs{y - x}^{n + sp}} \dif z \dif y\\
  &\le \Cr{cst_uahah3phahT1Yimu}
    \int_{\Rset^n} 
    \int_{\Rset^\ell \setminus B_{\abs{z - (x, t)}}(x)} 
      \frac
      {\abs{\varphi (z) - \varphi (x, t)}^p}
      {\abs{y - x}^{n + \ell + sp}} 
      \dif y 
    \dif z \\
    &= \Cl{cst_maixiengai0xoo6B}
    \int_{\Rset^n} 
    \frac
    {\abs{\varphi (z) - \varphi (x, t)}^p}
    {\abs{z - (x, t)}^{n + sp}} 
    \dif z\;. 
\end{split}
\end{equation}
Combining the inequalities \eqref{eq_peepaaye9onguPh3} and \eqref{eq_iehieb9Imiaziezo}, we have 
\begin{multline*}
\int_{\Rset^\ell} 
\int_{\Rset^\ell} 
\frac{\abs{\varphi (y, t) - \varphi (x, t)}^p}{\abs{y - x}^{\ell + sp}}
\dif y
\dif x\\
\le 
2^{p - 1}  \Cr{cst_maixiengai0xoo6B}
\brk[\Big]{
\int_{\Rset^\ell} 
\int_{\Rset^n} 
\frac{\abs{\varphi (z) - \varphi (x, t)}^p}{\abs{z - (x, t)}^{n + sp}}
\dif z
\dif x 
+ 
\int_{\Rset^\ell} 
\int_{\Rset^n} 
\frac{\abs{\varphi (z) - \varphi (y, t)}^p}{\abs{z - (y, t)}^{n + sp}}
\dif z
\dif y
}\;,
\end{multline*}
and the conclusion \eqref{eq_aleeziexe1zi4Yaechee5Qui} follows.
\end{proof}

\subsection{Hardy Estimates}

The cancellation condition is also a necessary and sufficient condition for a family of endpoint Hardy inequalities.

\begin{theorem}[Hardy estimate for cancelling operators%
\footnote{%
The Hardy estimate of \cref{theorem_Hardy} is due to Vladimir Gilelevich \familyname{Maz\cprime{}ya} \cite{Mazya_2010} for the operator \((\Delta, \nabla \operatorname{div})\). 
  Our proof follows the proof by Pierre \familyname{Bousquet} and Petru \familyname{Mironescu} \cite{Bousquet_Mironescu_2011} for the latter operator and its adaptations to cancelling operators by Pierre \familyname{Bousquet} and the author \cite{Bousquet_VanSchaftingen_2014}.%
}%
]
  \label{theorem_Hardy}
Let \(n \in \Nset\setminus\set{0}\), let \(V\) and \(E\) be finite-dimensional vector spaces, and let \(A (\Deriv)\) be a homogeneous constant coefficient differential operator
of order \(k \in \Nset \setminus \set{0}\) from 
\(V\) to \(E\) on \(\Rset^n\).
If \(A (\Deriv)\) is injectively elliptic and if \(\ell \in \Nset \setminus \set{0}\) satisfies \(0 < k - \ell < n \), then there exists a constant \(C \in \intvo{0}{+\infty}\) such that for every \(u \in C^\infty_c (\Rset^n, V)\), 
\begin{equation}
\label{eq_ahHaichooxah6pha5thah6ie}
\int_{\Rset^n} \frac{\abs{\Deriv^\ell u(x)}}{\abs{x}^{k - \ell}} \dif x 
\le 
C
\int_{\Rset^n} \abs{A (\Deriv)[u]}\;,
\end{equation}
if and only if the operator \(A (\Deriv)\) is cancelling.
\end{theorem}

Our main tool in the proof of \cref{theorem_Hardy} will be the following duality estimate.

\begin{proposition}[Weighted duality estimate%
\footnote{%
\Cref{proposition_L1_compensation} is due to
Pierre \familyname{Bousquet} and to the author \cite{Bousquet_VanSchaftingen_2014}
under the additional assumption that \(L (\Deriv)\) is cocancelling,
and in the general case to Bogdan \familyname{Raiță} \cite{Raita_2019}.%
  }%
  ]
  \label{proposition_L1_compensation}
Let \(n \in \Nset \setminus \{0, 1\}\), let \(E\) and \(F\) be finite-dimensional vector spaces, let \(L (\Deriv)\) be a homogeneous constant coefficient differential operator of order \(\ell\) from \(E\) to \(F\) on \(\Rset^n\). There exists a constant \(C \in \intvo{0}{+\infty}\) such that
for every \(f \in C^\infty (\Rset^n, E) \cap L^1 (\Rset^n, E)\) that satisfies \(L (\Deriv) f = 0\) in \(\Rset^n\) in the sense of distributions and every \(\varphi \in C^1_c (\Rset^n, E)\) such that for every \(x \in \Rset^n\), \begin{equation}
\label{eq_xai4DiegaiH4eChee9ohshoo}
\varphi (x) \in \brk[\Big]{\bigcap_{\xi \in \Rset^n \setminus\set{0}} \ker L (\xi) }^{\perp},                                                                                                                                                                   \end{equation}
one has
  \begin{equation}
  \label{eq_xeiGh4eefai0xiephu2Choo7}
    \abs[\Big]{\int_{\Rset^n} \dualprod{f}{\varphi}}
    \le 
    \sum_{j = 1}^\ell
    C \int_{\Rset^n} \abs{f (x) } \abs{x}^{j} \abs{\Deriv^j \varphi (x)}\dif x\;.
  \end{equation}
\end{proposition}

\begin{proof}
We define the mapping \(T : E \to E\) to be the orthogonal projection on the space \((\bigcap_{\xi \in \Rset^n \setminus \set{0}} \ker L (\xi))^\perp\).
This implies that 
\(
\ker T= T[E]^\perp = \bigcap_{\xi \in \Rset^n \setminus \set{0}} \ker L (\xi) 
\)
 and thus that, by \cref{lemma_decomposition_L_xi_j}, there exists \(m \in \Nset\), 
vectors \(\xi_1, \dotsc, \xi_m \in \Rset^n \setminus \set{0}\) and \(Q_1, \dotsc, Q_m \in \Lin (F, E)\) such that 
\(
    T
  = 
    \smash{\sum_{i = 1}^m Q_i \compose L (\xi_i)}
\).
If we define the polynomial \(P: \Rset^n \to \Lin (E, F)\) for every \(x \in \Rset^n\) by 
\begin{equation*}
P (x) \defeq \sum_{i = 1}^m \frac{\dualprod{\xi_i}{x}^m Q_i^*}{m!},
\end{equation*}
 we have then \(T = L (\Deriv)^*[P]\), and thus in view of \eqref{eq_xai4DiegaiH4eChee9ohshoo}
\[
\int_{\Rset^n} \dualprod{f}{\varphi} = 
\int_{\Rset^n} \dualprod{f}{T[\varphi]} = 
\int_{\Rset^n} \dualprod{f}{L (\Deriv)^*[P] \varphi}.
\]
Since \(L(\Deriv)f = 0\) and by integration by parts, we deduce  that 
\[
\begin{split}
\int_{\Rset^n} \dualprod{f}{L (\Deriv)^*[P] \varphi}
&= \int_{\Rset^n} \dualprod{f}{L (\Deriv)^*[P] \varphi}
- \int_{\Rset^n} \dualprod{L (\Deriv) f}{P \varphi}\\
&=\int_{\Rset^n} \dualprod{f}{L (\Deriv)^*[P] \varphi
- L(\Deriv)^*[ P \varphi]}\;,
\end{split}
\]
and the conclusion \eqref{eq_xeiGh4eefai0xiephu2Choo7} then follows.
\end{proof}

\begin{proof}[Proof of \cref{theorem_Hardy}]
  Let \(G_A : \Rset^n \setminus \set{0} \to \Lin (V, E)\) be the representation kernel given by \cref{proposition_fundamental_solution}. 
  Fixing a function \(\varrho \in C^\infty_c(\Rset^n, \Rset)\) such that \(\varrho = 1\) on \(B_{1/4}(0)\) and \(\varrho=0\) on \(\Rset^n \setminus B_{1/2}(0)\), we define the kernels \(H_A : \Rset^n \times \Rset^n \to \Lin (V, E)\) and \(K_A : \Rset^n \times \Rset^n \to \Lin (V, E)\) for \(x, y \in (\Rset^n \setminus \set{0})  \times \Rset^n \) with \(x \ne y\) by
\[
 H_A (x, y) \defeq \varrho \Bigl( \frac{y}{\abs{x}} \Bigr) \Deriv^{\ell} G_A (x)
\]
and
\[
 K_A (x, y) \defeq \Deriv^{\ell} G_A (x - y) - \varrho \Bigl( \frac{y}{\abs{x}} \Bigr) \Deriv^{\ell} G_A (x)\;,
\]
so that \(\Deriv^\ell G_A (x - y) = H_A (x, y) + K_A (x, y)\).
Letting \(L (\Deriv)\) be the differential operator of order \(m\) given by \cref{proposition_compatibility_conditions}, \(L (\Deriv)\) is cocancelling and \(L (\Deriv) A (\Deriv)u = 0\) in \(\Rset^n\).
By \cref{proposition_L1_compensation} and by the homogeneity of \(\Deriv^{\ell}_A G\) following from the condition that \(\ell > k - n\) (\cref{proposition_fundamental_solution} \eqref{it_oocee4ao4DaiK8sheWeechee}), we have
for every \(x \in \Rset^n \setminus \set{0}\),
\begin{equation}
\begin{split}
\abs[\Big]{\int_{\Rset^n}  H_A (x, y) &[A (\Deriv) u (y)] \dif y}\\
&\le \C \sum_{j = 1}^m \int_{B_{\abs{x}/2}(0)} \abs{A (\Deriv) u(y)} \abs{y}^j \abs{\Deriv_y^j H(x, y)} \dif y\\
&\le \C \sum_{j = 1}^m \int_{B_{\abs{x}/2}(0)} \frac{ \abs{A (\Deriv) u (y)} \abs{y}^j}{\abs{x}^{n - k + \ell + j}}  \dif y\\
&\le \C  \int_{B_{\abs{x}/2}(0)} \frac{ \abs{A (\Deriv) u (y)} \abs{y}}{\abs{x}^{n - k + \ell + 1}}  \dif y\;,
\end{split}
\end{equation}
and thus 
\begin{equation}
  \label{eq_yoh9Acohb7cah3Ohgoo7eish}
  \begin{split}
\int_{\Rset^n} &\abs[\Big]{\int_{\Rset^n}  H_A (x, y) [A (\Deriv) u (y)] \dif y}\frac{\mathrm{d} x}{\abs{x}^{k - \ell}} \\
&\le \Cl{cst_eiL2tahg3iigeidoozeeChei} \int_{\Rset^n} \int_{B_{\abs{x}/2}(0)} \frac{\abs{y}\abs{A (\Deriv) u (y)} }{\abs{x}^{n + 1}} \dif y \dif x\\
&=  \Cr{cst_eiL2tahg3iigeidoozeeChei} \int_{\Rset^n} \int_{\Rset^{n} \setminus B_{2 \abs{y}}(0)} \frac{\abs{y}\abs{A (\Deriv) u (y)} }{\abs{x}^{n + 1}} \dif x \dif y
\le \C \int_{\Rset^n} \abs{A (\Deriv) u}\;.
\end{split}
\end{equation}
Next we have 
\begin{equation}
  \label{eq_uFiel4aemou4aiNgizoow3ob}
\int_{\Rset^n} \abs[\Big]{\int_{\Rset^n}  K_A (x, y) [A (\Deriv) u (y)] \dif y} \frac{\mathrm{d} x}{\abs{x}^{k - \ell}} 
 \le  \int_{\Rset^n} \int_{\Rset^n}  \frac{\abs{K_A (x, y)}}{\abs{x}^{k - \ell}} \dif x \abs{A (\Deriv) u (y)}\dif y\;.
\end{equation}
Again by homogeneity of \(\Deriv^\ell G_A\), we have if \(\abs{x} < 2 \abs{y}\)
\[
 \abs{K_A (x, y)} \le \frac{\C}{\abs{x-y}^{n - k + \ell} }\:,
\]
and if \(\abs{x} \ge 2 \abs{y}\)
\[
\abs{K_A (x, y)} \le  \frac{\C\abs{y}}{\abs{x}^{n - k + \ell + 1} }\;,
\]
so that, since \(k - \ell < n\),
\begin{equation}
  \label{eq_ne4dae2Ocai4Uh2ar1iuk4so}
  \begin{split}
 \int_{\Rset^n}&  \frac{\abs{K_A (x, y)}}{\abs{x}^{k - \ell}} \dif x \\
 &\le \C \Bigl(\int_{B_{2 \abs{y}}(0)} \frac{\mathrm{d} x}{\abs{x - y}^{n - k + \ell}\abs{x}^{k - \ell}}  + \int_{\Rset^n \setminus B_{2 \abs{y}}(0)} \ \frac{\abs{y}}{\abs{x}^{n + 1}} \dif x\Bigr) \le \C\;.
\end{split}
\end{equation}
We conclude from \eqref{eq_uFiel4aemou4aiNgizoow3ob} and \eqref{eq_ne4dae2Ocai4Uh2ar1iuk4so}
that 
\begin{equation}
  \label{eq_datahlieghaivi0ahdooshaY}
  \int_{\Rset^n} \abs[\Big]{\int_{\Rset^n} K_A (x, y) [A (\Deriv) u (y)] \dif y} \frac{\mathrm{d} x}{\abs{x}^{k - \ell}}  \le \C \int_{\Rset^n} \abs{A (\Deriv) u }\;.
\end{equation}
The estimate \eqref{eq_ahHaichooxah6pha5thah6ie} then follows from the inequalities \eqref{eq_yoh9Acohb7cah3Ohgoo7eish} and \eqref{eq_datahlieghaivi0ahdooshaY}.

For the necessity of the cancellation, taking a vector 
\(
e \in \bigcap_{\xi \in \Rset^n \setminus \set{0}} A (\xi)[V] 
\)
and letting the sequence \((u_j)_{j \in \Nset}\) be given by \cref{lemma_regularization_G_A} for this vector \(e\), we have by \eqref{eq_ahHaichooxah6pha5thah6ie} and by Fatou's lemma
\[
\begin{split}
\int_{B_1(0)} \frac{\abs{\Deriv^{\ell} G_A(x)[e]}}{\abs{x}^{k - \ell}} \dif x 
&\le 
 \liminf_{j \to \infty} \int_{\Rset^n} \frac{\abs{\Deriv^{\ell} u_j(x)}}{\abs{x}^{k - \ell}} \dif x \\
& \le 
 \lim_{j \to \infty}
 \Cl{cst_tieFeengeidee0ovaih9yee1} \int_{\Rset^n} \abs{A (\Deriv) u_j}\\
 &\le \C \abs{e}\;.
 \end{split}
\]
Since \(\ell > n - k\), \(\Deriv^\ell G_A[e] : \Rset^n \setminus \set{0} \to V\) is homogeneous of degree \(n - (k - \ell)\) (by \cref{proposition_fundamental_solution} \eqref{it_oocee4ao4DaiK8sheWeechee}); this implies that \(G_A[e] = 0\) on \(\Rset^n\) and thus \(e = 0\).
\end{proof}

\subsection{Uniform Estimate and Weakly Cancelling Operators}

By \cref{theorem_Hardy} with \(\ell = k - (n - 1) > 0\) and the Sobolev representation formula, one has when \(A (\Deriv)\) is injectively elliptic and cancelling that for every \(x \in \Rset^n\)
\begin{equation}
 \abs{\Deriv^{k - n} u (x)} \le  \C \int_{\Rset^n}\frac{\abs{\Deriv^{k - n - 1} u(y)}}{\abs{x - y}^{n - 1}} \dif x 
\le 
\C \int_{\Rset^n} \abs{A (\Deriv)[u]}\;.
\end{equation}
It turns out however that the cancellation condition is not necessary for such an \(L^\infty\) estimate and can be replaced, as showed by Bogdan \familyname{Raiță}, by a weaker condition.

\begin{theorem}[\(L^\infty\) estimate for cancelling operators%
\footnote{%
\Cref{theorem_Linfinity} is due to Bogdan \familyname{Raiță} \cite{Raita_2019}.
  The stronger cancellation condition was proved to be sufficient by Pierre \familyname{Bousquet} and the author \cite{Bousquet_VanSchaftingen_2014}, as a consequence of the Hardy inequality of \cref{theorem_Hardy}.%
  }%
  ]
  \label{theorem_Linfinity}
Let \(n \in \Nset\setminus\set{0}\), let \(V\) and \(E\) be finite-dimensional vector spaces, and let \(A (\Deriv)\) be a homogeneous constant coefficient differential operator
of order \(k \in \Nset \setminus \set{0}\) from 
\(V\) to \(E\) on \(\Rset^n\). 
If \(A (\Deriv)\) is injectively elliptic and if \(k \ge n\), then there exists a constant \(C \in \intvo{0}{+\infty}\) such that for every \(u \in C^\infty_c (\Rset^n, V)\) and every \(x \in \Rset^n\)
\begin{equation}
\label{eq_jupeefaeshie4ainaeci3Quu}
\abs{\Deriv^{k - n} u(x)}
\le 
C
\int_{\Rset^n} \abs{A (\Deriv)[u]}\;,
\end{equation}
if and only if for every 
\(e \in \bigcap_{\xi \in \Rset^n \setminus \set{0}} A (\xi)[V],
\)
one has
\begin{equation}
  \label{eq_aiboh2iesaeMisee2oyahah4}
  \int_{\Sset^{n - 1}} \xi^{\otimes k - n} A(\xi)^{-1} [e] \dif \xi = 0\;.
\end{equation}
\end{theorem}

A vector differential operator \(A (\Deriv)\) satisfying the condition \eqref{eq_aiboh2iesaeMisee2oyahah4} is known as a \emph{weakly cancelling} operator \cite{Raita_2019}.

\begin{proof}[Proof of \cref{theorem_Linfinity}]  
We first assume that the condition \eqref{eq_aiboh2iesaeMisee2oyahah4} holds.
We let \(G_A : \Rset^n\setminus \set{0} \to \Lin (V, E)\) be the representation kernel given by \cref{proposition_fundamental_solution}. Differentiating \(k - n\) times the identity \eqref{eq_jal9Niez6Vu2ohved5OorooG}, we get for every \(x \in \Rset^n\)
\begin{equation}  
\label{eq_oog9ik9Nohz1ua9eevaiph1U}
\begin{split}
\Deriv^{k - n} u (x) &= \int_{\Rset^n} \Deriv^{k - n} G_A (x - y)[A (\Deriv) u (y)] \dif y\\
&=  \int_{\Rset^n} \Deriv^{k - n} G_A (h)[A (\Deriv) u (x - h)] \dif h\;.
\end{split}
\end{equation}
With \(P_A : \Rset^n \to \Lin (V, E)\) being the homogeneous polynomial of degree \(k - n\) given by \cref{proposition_fundamental_solution}  \eqref{it_oocee4ao4DaiK8sheWeechee}, 
we observe that \(\Deriv^{k - n} P_A : \Rset^n \to \Lin^{k - n}_{\mathrm{sym}} (\Rset^n, \Lin(E, V))\) is a constant polynomial that we identify to some \(P_A \in \Lin (V, E)\) and 
we have for every \(h \in \Rset^n \setminus \set{0}\),
\begin{equation}
\label{eq_ohhohs2caivo6eeL5Nuquoog}
 \Deriv^{k - n} G_A (h) = \Deriv^{k - n} G_A (h/\abs{h}) + \Deriv^{k - n} P_A \ln \abs{h}\;.
\end{equation}
Since \(D^{k - n} G_A\) is bounded on the unit sphere of \(\Rset^n\), we first have immediately
\begin{equation}
\label{eq_Rek7yeH3noy5uShaiZoof5ph}
 \abs[\Big]{\int_{\Rset^n} \Deriv^{k - n} G_A (h/\abs{h})[A (\Deriv) u (x - h)] \dif y\,}
 \le \C \int_{\Rset^n} \abs{A (\Deriv)u}\;.
\end{equation}
Next, we take a function \(\theta \in C^\infty_c (\Rset, \Rset)\) satisfying \(\theta (t)= t\) when \(\abs{t} \le 1/2\) and \(\theta (t) = 0\) when \(\abs{t} \ge 1\). 
For every \(\lambda \in \intvo{0}{+\infty}\) and every \(\mu \in \Lin_{\mathrm{sym}}^{k - n} (\Rset^n, V)\), 
we define the function \(\varphi_{\lambda, \mu} \in C^\infty_c (\Rset^n, E)\) by the condition that for every \(h \in \Rset^n \setminus \set{0}\) and every \(e \in E\), we have
\begin{equation}
\label{eq_Chieve7raish9jai5haeZahr}
\dualprod{\varphi_{\lambda, \mu} (h)}{e} = \dualprod{\Deriv^{k - n} P_A[e]}{\mu} \frac{\theta(\lambda \ln \abs{h} )}{\lambda}\;. 
\end{equation}
Next, by \cref{proposition_fundamental_solution} \eqref{it_supahJuacocei6efiey7aife} and by our assumption \eqref{eq_aiboh2iesaeMisee2oyahah4}, for every  vector 
\begin{equation*}
e \in \bigcap_{\xi \in \Rset^n \setminus \set{0}} A (\xi)[V]\,, 
\end{equation*}
we have \(\Deriv^{k - n} P_A [e] = 0\).
Hence, for every \(h \in \Rset^n\) and every \(\lambda \in \intvo{0}{+\infty}\), we have 
\begin{equation}
\label{eq_uDee0phohJaiKo9phie9ohLi}
\varphi_{\lambda, \mu} (h) \in \brk[\Big]{\bigcap_{\xi \in \Rset^n \setminus \set{0}} A (\xi)[V]}^\perp
\end{equation}
 in view of \eqref{eq_Chieve7raish9jai5haeZahr}.
Moreover,
for every \(j \in \Nset \setminus \set{0}\) and \(x \in \Rset^n \setminus \set{0}\), we have as a consequence of \eqref{eq_Chieve7raish9jai5haeZahr}
\begin{equation}
  \abs{\Deriv^j \varphi_{\lambda, \mu}(x)} \le \C (1 + \lambda^{j - 1}) \abs{\mu}/\abs{x}^j.
\end{equation}
In view of \eqref{eq_uDee0phohJaiKo9phie9ohLi}, \cref{proposition_L1_compensation} is applicable and yields for every \(\lambda \in \intvo{0}{1}\), 
\begin{equation}
\label{eq_gahj6phie0cheeThoh2yu3Ai}
    \abs[\Big]{\int_{\Rset^n} \dualprod{A (\Deriv)u(x-h)}{\varphi_{\lambda, \mu}(h)} \dif h}
    \le \C\, \abs{\mu} \int_{\Rset^n} \abs{A (\Deriv)u} \;.
\end{equation}
Since on the other hand, we have by \eqref{eq_Chieve7raish9jai5haeZahr}
\begin{equation}
\label{eq_Ohshaihaifa1yeihoo9iXuch}
\begin{split}
 \int_{\Rset^n} &\dualprod{\Deriv^{k - n} P_A [A (\Deriv) u (x - h)]}{\mu} \ln \abs{h} \dif h\\
 &\hspace{2cm}= \lim_{\lambda \to 0} \int_{\Rset^n} \dualprod{\Deriv^{k - n} P_A [A (\Deriv) u (x - h)]}{\mu} \frac{\theta (\lambda \ln \abs{h})}{\lambda}\dif h\\
 &\hspace{2cm}= \lim_{\lambda \to 0} \int_{\Rset^n}  \dualprod{\varphi_{\lambda, \mu}(h)}{A (\Deriv)u(x-h)} \dif h\;,
\end{split}
\end{equation}
we reach the inequality \eqref{eq_jupeefaeshie4ainaeci3Quu} through the identities \eqref{eq_oog9ik9Nohz1ua9eevaiph1U},  \eqref{eq_ohhohs2caivo6eeL5Nuquoog} and \eqref{eq_Ohshaihaifa1yeihoo9iXuch}, and the estimates \eqref{eq_Rek7yeH3noy5uShaiZoof5ph} and \eqref{eq_gahj6phie0cheeThoh2yu3Ai}.

For the necessity of the condition \eqref{eq_aiboh2iesaeMisee2oyahah4}, taking a vector 
\(
e \in \bigcap_{\xi \in \Rset^n \setminus \set{0}} A (\xi)[V] 
\)
and letting the sequence \((u_j)_{j \in \Nset}\) be given by \cref{lemma_regularization_G_A} for this vector \(e\), we have by \eqref{eq_jupeefaeshie4ainaeci3Quu} and by Fatou's lemma
\begin{equation}
\label{eq_oShahkoh9eeb8oijauyeemie}
\begin{split}
\limsup_{j \to \infty} \sup_{x \in \Rset^n} \, \abs{u_j (x)} 
& \le 
 \lim_{j \to \infty}
 \Cl{cst_heoN6aeNgeetie8ush3She7p} \int_{\Rset^n} \abs{A (\Deriv) u_j}
 \le \C \abs{e}\;.
 \end{split}
\end{equation}
On the other hand by \cref{proposition_fundamental_solution} \eqref{it_oocee4ao4DaiK8sheWeechee} and \eqref{it_supahJuacocei6efiey7aife}, we have for every \(e \in E\) and \(x \in \Rset^n \setminus \set{0}\)
\begin{equation}
\label{eq_eisaigohtix5EeTu7eengair}
\begin{split}
  \Deriv^{k - n} G_A (x)[e]
  &=  \Deriv^{k - n} G_A (x/\abs{x})[e] + \ln \abs{x} \, P_A (x)[e]\\
  &=\Deriv^{k - n} G_A (x/\abs{x})[e] +  \ln \abs{x} \int_{\Sset^{n - 1}} \frac{\xi^{\otimes k - n} A (\xi)^{-1} [e]}{(2 \pi i)^n} \dif \xi\;.
\end{split}
\end{equation}  
Since \(u_j \to G_A [e]\) on \(B_1(0)\), we conclude by \eqref{eq_oShahkoh9eeb8oijauyeemie} and \eqref{eq_eisaigohtix5EeTu7eengair} that the condition \eqref{eq_aiboh2iesaeMisee2oyahah4} holds.
\end{proof}

\end{document}